\documentclass[a4paper, 10pt, oneside, reqno]{amsart}

\usepackage[a4paper]{geometry}
\usepackage[utf8]{inputenc}
\usepackage[english]{babel}
\usepackage[T1]{fontenc}
\usepackage{textcomp}
\usepackage{lmodern}

\usepackage{enumitem}
\usepackage{fmtcount,refcount}

\usepackage{stmaryrd}

\usepackage[mathscr]{euscript}

\usepackage{tikz-cd}

\usepackage{todonotes}

\usepackage{amsmath}
\usepackage{amsfonts}
\usepackage{amsthm}
\usepackage{latexsym}
\usepackage{amssymb}

\DeclareMathOperator{\Hom}{Hom}

\DeclareMathOperator{\Fun}{Fun}
\DeclareMathOperator{\Ob}{Ob}
\DeclareMathOperator{\coker}{coker}

\DeclareMathOperator{\Res}{Res}

\DeclareMathOperator{\Ind}{Ind}
\DeclareMathOperator{\Mod}{Mod}

\DeclareMathOperator{\hproj}{h-proj}
\DeclareMathOperator{\pretr}{pretr}
\DeclareMathOperator{\cone}{C}
\DeclareMathOperator{\dual}{D}

\DeclareMathOperator{\hfp}{hfp}
\DeclareMathOperator{\hfpbd}{hfp^b}
\DeclareMathOperator{\fp}{fp}

\DeclareMathOperator{\compdg}{dgm}

\DeclareMathOperator{\dercomp}{\mathsf{D}}
\DeclareMathOperator{\dercompdg}{\mathsf{D}_{\mathrm{dg}}}
\newcommand{\dercompdgmin}{\mathsf{D}^{-}_{\mathrm{dg}}}
\newcommand{\dercatabplus}{\mathfrak D^{+}}
\newcommand{\dercatabdgplus}{\mathfrak D^{+}_{\mathrm{dg}}}
\newcommand{\dercatabdgbd}{\mathfrak D^{b}_{\mathrm{dg}}}

\DeclareMathOperator{\RHom}{\mathbb R\!\Hom}
\DeclareMathOperator{\lotimes}{\overset{\mathbb L}{\otimes}}
\newcommand{\cat}{\mathscr}
\newcommand{\opp}[1]{{#1}^{\mathrm{op}}}
\newcommand{\kat}{\mathsf}
\newcommand{\smallcat}{\mathfrak}
\newcommand{\injcat}{\mathbf}
\newcommand{\HH}{\operatorname{HH}}
\newcommand{\Inj}{\operatorname{Inj}}
\newcommand{\DGInj}{\operatorname{DGInj}}

\newcommand{\Hqe}{\kat{Hqe}}

\newcommand{\basering}[1]{\mathbf{#1}}

\newtheorem{introtheorem}{Theorem}
\newtheorem*{introdef}{Definition}

\newtheorem{theorem}{Theorem}[subsection]
\newtheorem{proposition}[theorem]{Proposition}
\newtheorem{corollary}[theorem]{Corollary}

\newtheorem{lemma}[theorem]{Lemma}
\newtheorem{sublemma}[theorem]{Sublemma}

\theoremstyle{remark}
\newtheorem{remark}[theorem]{Remark}
\newtheorem{example}[theorem]{Example}

\newtheorem*{introremark}{Remark}

\theoremstyle{definition}
\newtheorem{definition}[theorem]{Definition}

\newlist{primenumerate}{enumerate}{1}
\setlist[primenumerate,1]{label={(\arabic*$'$)}}

\let\oldmarginpar\marginpar
\def\marginpar#1{\oldmarginpar{\tiny \raggedright #1}}

\numberwithin{equation}{section}

\title{T-structures on dg-categories and derived deformations}

\author{Francesco Genovese}
\address[Francesco Genovese]{Univerzita Karlova, Matematicko-fyzik\'aln\'i fakulta \\ Sokolovsk\'a 49/83, 186 75 Praha 8 \\ Czech Republic}
\email{genovese@karlin.mff.cuni.cz}

\author{Wendy Lowen} 
\address[Wendy Lowen]{Universiteit Antwerpen, Departement Wiskunde, Middelheimcampus, Middelheimlaan 1, 2020 Antwerp, Belgium}
\email{wendy.lowen@uantwerpen.be}

\author{Michel Van den Bergh}
\address[Michel Van den Bergh]{Universiteit Hasselt\\ Campus Diepenbeek\\ Agoralaan Gebouw D \\ 3590 Diepenbeek \\Belgium}
\email{michel.vandenbergh@uhasselt.be}


\thanks{The authors acknowledge the support of the Research Foundation Flanders (FWO) under Grant No G.0D86.16N. This project has received funding from the European Research Council (ERC) under the European Union’s Horizon 2020 research and innovation programme (grant agreement No. 817762). The first named author also acknowledges the support of the Czech Science Foundation grant [GA \v{C}R 20-13778S]}

\begin{document}
\maketitle

\begin{abstract}
	This paper is a sequel to ``t-structures and twisted complexes on derived injectives'' by the same authors. We develop the foundations of the infinitesimal derived deformation theory of pretriangulated dg-categories endowed with t-structures. This generalizes the deformation theory of abelian categories developed by the last two authors. We show how deformations of dg-categories of derived injectives yield derived deformations of the associated t-structures.
\end{abstract}

\tableofcontents

\section{Introduction} 
\subsection{Context}
As a sequel to \cite{genovese-lowen-vdb-dginj}, this paper is the second one in an ongoing project to develop the deformation theory of triangulated categories with t-structure, inspired by the deformation theory of abelian categories from \cite{lowen-vdb-hochschild,lowen-vandenbergh-deformations-abelian}.

Consider the dual numbers $D_0 = \basering k[\epsilon]/(\epsilon^2)$ over a field
$\basering k$. Then every $\basering k$-linear category gives rise to a $D_0$-linear category via restriction for the canonical map $D_0 \rightarrow k$. Algebra deformations in the sense of Gerstenhaber are (flat) lifts along the left adjoint $\basering k \otimes_{D_0} -$ of the restriction functor, and this notion of deformation generalizes to linear or dg-categories by viewing those as (dg-)algebras with several objects - where in the dg-case we lift along the derived tensor product. In this introduction we will refer to such deformations as \emph{left deformations}. In the linear case, left deformations are controlled by ``Hochschild cohomology'' (denoted by ``$\HH$'' below), but in the dg-case in general Hochschild cohomology describes deformations with curvature, leading to rather serious technical difficulties \cite{keller-lowen-nicolas}.

In \cite{lowen-vandenbergh-deformations-abelian} it was observed that for abelian categories, it is actually more appropriate to lift along the right adjoint $\Hom_{D_0}(\basering k,-)$ of restriction, leading to a type of deformation we will refer to as \emph{right deformations} in this introduction.  Thanks to the 1-1 correspondence $\mathfrak A \leftrightarrow \Inj(\mathfrak A)$ (*) between abelian categories with enough injectives on the one hand and categories of injectives - which can be characterised intrinsically - on the other hand, a deformation equivalence between right deformations of $\mathfrak A$ and left deformations of $\Inj(\mathfrak A)$ was established in \cite{lowen-vandenbergh-deformations-abelian}. In particular, right deformations of $\cat C$ are controlled by
$\HH(\Inj(\mathfrak A))$.

A main aim of the project is to obtain an analogous deformation equivalence for appropriate pretriangulated dg-categories $\cat{A}$ with t-structure. Here, injectives have to be replaced by \emph{derived injectives} $\DGInj(\cat{A})$ \cite{rizzardo-nonFM} and right deformations are defined using $\RHom$ rather than $\Hom$ and take t-structures into account. Since the dg-category $\DGInj(\cat{A})$ has trivial cohomology in strictly positive degrees, no curvature is picked up under deformation so left deformations are controlled by Hochschild cohomology. Apart from the added level of generality, an additional advantage of such a  setup lies in the fact that now one can deform in the direction of commutative dgas like $D_n = \basering k[\epsilon]/(\epsilon^2)$ for $\deg(\epsilon) = n\le 0$.

In \cite{genovese-lowen-vdb-dginj}, the relevant generalization of (*) was proven for pretriangulated categories with t-structure with enough derived injectives.  In the present paper we construct one direction of the resulting deformation equivalence, proving that a left deformation of $\DGInj(\cat{A})$ gives rise to a right deformation of $\cat{A}$ (see Theorem \ref{introthm:grand_summary} below).  As a corollary we find that if  $\mathfrak{A}$ is a Grothendieck abelian category then a right deformation $\mathfrak{A}'$ of $\mathfrak{A}$ gives rise to a right deformation $\mathfrak{D}^+(\mathfrak{A}')$ of $\mathfrak{D}^+(\mathfrak{A})$ between the bounded below dg-derived categories (see Theorem \ref{introtheorem:abdeform_derdeform} below). We remark that the proof actually works under the more general assumption that $\mathfrak A$ is an abelian category with enough injectives.

In a subsequent paper, the other direction of the deformation equivalence will be treated and the proof of the equivalence will be completed.

\subsection{A note on set-theoretic universes} \label{subsec:sizes}
We shall fix a Grothendieck universe $\mathbb U$, and a bigger universe $\mathbb V$ such that $\mathbb U \in \mathbb V$. Unless otherwise specified, individual objects such as (dg\nobreakdash-)rings, modules etc. are implicitly taken to be $\mathbb U$-small. As for (dg\nobreakdash-)categories: unless otherwise specified, lowercase fraktur letters ($\smallcat a, \smallcat b$,\,\ldots) will denote \emph{$\mathbb U$-small} (dg\nobreakdash-)categories; uppercase calligraphic letters ($\cat A, \cat B$,\,\ldots) and sometimes uppercase bold letters ($\injcat I, \injcat J,\,\ldots)$ will denote \emph{locally $\mathbb U$-small} (dg\nobreakdash-)categories. From now on, the expression ``small (dg\nobreakdash-) category'' will mean ``$\mathbb U$-small (dg\nobreakdash-)category'', and ``(dg\nobreakdash-)category'' will mean ``locally $\mathbb U$-small (dg\nobreakdash-)category''.

If a dg-category $\smallcat a$ is $\mathbb U$-small, then its dg-category of right ($\mathbb U$-small) dg-modules $\compdg(\smallcat a)$ will be locally $\mathbb U$-small, hence $\mathbb V$-small. This setup will allow us to work with the locally $\mathbb V$-small dg-categories of dg-modules over locally $\mathbb U$-small dg-categories, or dg/quasi-functors between locally $\mathbb U$-small dg-categories.

\subsection{Structure of the paper}
Throughout the paper, we work with dg-categories over dg-rings. 

In \S \ref{sec:preliminaries} we present some preliminaries on dg-categories and quasi-functors. We first check (\S \ref{ssec:qf}, Proposition \ref{prop: ind_rqr_firstargument}, Corollary \ref{coroll: ind_rqr_firstargument}) that the dg-category of quasi-functors $\RHom(\cat A, \cat B)$ can be described as the dg-category $\hproj^\mathrm{rqr}(\opp{\cat A} \otimes \cat B)$ of right quasi-representable h-projective bimodules, where \emph{either $\cat A$ or $\cat B$ can be taken to be h-flat}. This will be essential when we deal with the change of base dg-ring. The following \S \ref{sec:duality} deals with opposite and adjoint quasi-functors, whereas in \S \ref{sec:trunc} we deal with truncations of dg-categories and compatibility results with dg-categories of quasi-functors (cf. Proposition \ref{prop:H0_equiv_iftarget_H0} and Lemma \ref{lemma:truncations_image_qfun}).

In \S \ref{sec:tstruct_dgcat} we explore t-structures on pretriangulated dg-categories, which are understood as t-structures on their underlying homotopy categories. We address in \S \ref{subsec:truncationfunctors} the quasi-functoriality of truncations, which does not hold ``on the nose''. The main result of this section is the construction in \S \ref{subsec:tstruct_quasifunctors} of a natural t-structure on the dg-category of quasi-functors:
\begin{introtheorem}[cf. Theorem \ref{thm:tstruct_quasifunctors}, Proposition \ref{prop:tstruct_quasifunctors_heart}] \label{introtheorem:tstruct_qfun}
Let $R$ be a dg-ring strictly concentrated in nonpositive degrees. All dg-categories here will be over $R$.

Let $\cat B$ be a pretriangulated dg-category endowed with a t-structure $(\cat B_{\leq 0}, \cat B_{\geq 0})$, and let $\smallcat a$ be a small dg-category with cohomology concentrated in nonpositive degrees. Then, the dg-category of quasi-functors $\RHom(\smallcat a, \cat B)$ has a t-structure such that
\begin{align*}
    \RHom(\smallcat a,\cat B)_{\leq 0} &= \RHom(\smallcat a,\cat B_{\leq 0}), \\
    \RHom(\smallcat a,\cat B)_{\geq 0} &= \RHom(\smallcat a,\cat B_{\geq 0}).
\end{align*}
The heart of this t-structure can be identified with the abelian category $\Fun(H^0(\smallcat a),H^0(\cat B)^\heartsuit)$.
\end{introtheorem}
The assumption that $R$ is a dg-ring concentrated in nonpositive degrees is essential when we work with t-structures: for instance, it ensures that the derived category $\dercomp(R)$ of $R$ has a well-behaved non-degenerate ``canonical'' t-structure. For technical ease, we assume \emph{strict} instead of cohomological nonpositivity.

\S \ref{section:changeofrings} deals with the main technical tool of this paper, namely, the \emph{change of base dg-ring} of dg-categories. This is clearly essential for deformations. \S \ref{subsec:coextension} and \S \ref{subsection:scalar_leftextension} yield the following:
\begin{introtheorem}[cf. Theorem \ref{thm:coext_univproperty}, Proposition \ref{prop:scalar_leftextension}]
 Let $R \to S$ be a morphism of commutative dg-rings. If $\cat A$ is an $S$-linear dg-category, we denote by $\cat A_R$ the $R$-linear dg-category obtained by restriction of scalars along $R \to S$.
 
 For any $R$-linear dg-category $\cat B$, there is an $S$-linear dg-category
 \[
 \cat B_{(S)} = \RHom_R(S,\cat B),
 \]
 together with a natural $R$-linear quasi-equivalence
 \[
 \RHom_R(\cat A_R, \cat B) \cong \RHom_S(\cat A, \cat B_{(S)})_R.
 \]
 
 Moreover, for any $R$-linear dg-category $\cat b$, there is an $S$-linear dg-category
 \[
 S \lotimes_R \cat B,
 \]
 together with a natural $R$-linear quasi-equivalence
 \[
 \RHom_S(S \lotimes_R \cat B, \cat A)_R \cong \RHom_R(\cat B, \cat A_R).
 \]
\end{introtheorem}
\begin{introremark}
The dg-category $\cat B_{(S)}$ is more concretely defined as 
\[
\cat B_{(S)} = \hproj_R^\mathrm{rqr}(S,Q(\cat B)),
\]
where $Q(\cat B) \to \cat B$ is an h-flat resolution of $\cat B$. Not having to take a resolution of $S$ is essential to get a (strict!) $S$-linear structure on $\cat B_{(S)}$.
\end{introremark}
We can apply Theorem \ref{introtheorem:tstruct_qfun} to endow the coextension of scalars $\cat B_{(S)}$ with a natural t-structure; this is discussed in \S \ref{subsection:changeofrings_tstructures}, together with a number of compatibility results.

\S \ref{section:derived_deformations} is where we finally deal with derived deformations. In \S \ref{subsection:recollections_before_deformations} we recollect the technical results of the previous sections. We can then define derived deformations as follows:
\begin{introdef}[cf. Definition \ref{def:derived_deformations} and Definition \ref{def:left_derived_deformation}]
We fix a morphism $R \to S$ of commutative dg-rings strictly concentrated in nonpositive degrees.

Let $\cat B$ be an $S$-linear dg-category with a t-structure. A \emph{derived deformation} of $\cat B$ along $R \to S$ is an $R$-linear dg-category $\cat A$ with a t-structure together with a t-exact quasi equivalence $\cat B \cong \cat A_{(S)}$.

Let $\injcat B$ be an $S$-linear dg-category. A \emph{left derived deformation} of $\injcat B$ along $R \to S$ is an $R$-linear dg-category $\injcat A$ together with an $S$-linear quasi-equivalence $S \lotimes_R \injcat A \cong \injcat B$.
\end{introdef}

In \S \ref{subsec:fromabelian_toderived} we prove that flat abelian deformations of abelian categories (in the sense of \cite{lowen-vandenbergh-deformations-abelian}) induce derived deformation of the corresponding derived categories. More precisely, in the Grothendieck abelian case we get derived deformations of the left bounded derived categories, whereas for ``smaller'' abelian category we get derived deformations of the bounded derived categories:
\begin{introtheorem}[cf. Theorem \ref{thm:abdeform_derdeform}, Theorem \ref{thm:smallabdeform_derdeform}] \label{introtheorem:abdeform_derdeform}
Let $\mathfrak B$ be an $S$-linear Grothendieck abelian category (resp. an $S$-linear small abelian category), and let $\mathfrak A$ be a $R$-linear flat deformation of $\mathfrak B$ (which is again Grothendieck if $\mathfrak A$ is such, cf. \cite[Theorem 6.29]{lowen-vandenbergh-deformations-abelian} ). Then, $\dercatabdgplus(\mathfrak A)$ (resp. $\dercatabdgbd(\mathfrak A)$) is a $R$-linear derived deformation of the $S$-linear dg-category $\dercatabdgplus(\mathfrak B)$ (resp. $\dercatabdgbd(\mathfrak B)$).
\end{introtheorem}

The proof of Theorem \ref{introtheorem:abdeform_derdeform} for the ``Grothendieck case'' uses the fact that deformations of Grothendieeck abelian categories correspond to left deformations of their injective objects; then, one can reconstruct the derived categories from such injective objects. We again remark that the conclusions of Theorem \ref{introtheorem:abdeform_derdeform} also hold if we relax the ``Grothendieck'' hypothesis to the weaker ``having enough injectives'', see Remark \ref{remark:abdeform_derdeform_enoughinj}. The part regarding small abelian categories and their bounded derived categories is proved by passing to the (Grothendieck) ind-categories and then using \emph{homotopically finitely presented objects} (cf. \S \ref{subsubsec:hfp_obj}).

Recalling the more general (re)construction results in \cite{genovese-lowen-vdb-dginj}, we see that understanding left derived deformations of \emph{dg-categories of derived injectives} is essential for the study of derived deformations of t-structures. In this respect, we prove in \S \ref{subsec:deformations_dginj} the last main result of our paper:
\begin{introtheorem}[cf. Proposition \ref{prop:properties_lifting_leftdeformation}, \label{introthm:grand_summary} Theorem \ref{thm:dginj_derived_deformation}]
Let $R \to S$ be a morphism of dg-rings which satisfies suitable assumptions (see \S \ref{subsubsection:dginj_deform_setup}).

Let $\injcat J$ be an $S$-linear \emph{dg-category of derived injectives}, namely, $\injcat J$ is a (left) homotopically locally coherent dg-category (cf. \S \ref{subsubsection:dginj_deform_results}) such that $H^0(\injcat J)$ is Karoubian. Let $\injcat I$ be a $R$-linear left derived deformation of $\injcat J$. Then, $\injcat I$ is also a dg-category of derived injectives.

Let $\hfp(-)$ stand for objects with finitely presented cohomology (see Definition \ref{def:hfp} for more details). We set
\begin{align*}
    \cat B &= \opp{\hfp({\dercompdgmin}_{,S}(\opp{\injcat J}))}, \\
    \cat A &= \opp{\hfp({\dercompdgmin}_{,R}(\opp{\injcat I}))},
\end{align*}
endowed with the t-structures of \cite[Theorem 1.2]{genovese-lowen-vdb-dginj}. Then, $\cat A$ is a derived deformation of $\cat B$. In short, left derived deformations of dg-categories of derived injectives induce derived deformations of the t-structures associated to them.
\end{introtheorem}

\section{Preliminaries on dg-categories and quasi-functors} \label{sec:preliminaries}
\subsection{Resolutions of dg-categories}
We fix some base commutative dg-ring $R$. Unless otherwise specified our constructions are in the $R$-linear context, although we do not always say it.

We denote by $\kat{dgCat}(R)$ the category of small $R$-linear dg-categories. This is endowed with a model structure whose weak equivalences are the quasi-equivalences: this is discussed in \cite[\S 2]{lunts-schnurer-smoothness-equivariant} specifically for dg-categories \emph{over commutative dg-rings}, directly generalizing \cite{tabuada-dgcat}. The homotopy category of $\kat{dgCat}(R)$ is denoted by $\Hqe(R)$. Every small dg-category (with respect to a given universe, e.g. $\mathbb U$ or
$\mathbb V$) $\cat A$ has a \emph{cofibrant replacement}
\begin{equation}
    Q(\cat A) \xrightarrow{\sim} \cat A,
\end{equation}
and the dg-category $Q(\cat A)$ is, in particular, \emph{h-projective}: namely, the hom-complexes $Q(\cat A)(A,B)$ are h-projective $R$-dg-modules. This also implies that $Q(\cat A)$ is \emph{h-flat}, namely the hom-complexes $Q(\cat A)(A,B)$ are h-flat $R$-dg-modules. H-flat resolutions may be used to compute the derived tensor product in $\Hqe(R)$. E.g.\
\begin{equation}
    \cat A \lotimes \cat B = Q(\cat A) \otimes \cat B \cong \cat A \otimes Q(\cat B).
\end{equation}
We mention a useful fact:
\begin{lemma} \label{lemma:cofibrantresolutions_nonpositive}
Let $\cat A$ be a dg-category which is strictly concentrated in nonpositive degrees. Then, there is an h-projective (hence h-flat) resolution $Q(\cat A) \xrightarrow{\sim} \cat A$ such that $Q(\cat A)$ is also strictly concentrated in nonpositive degrees. 
\end{lemma}
\begin{proof}
This follows from \cite[Proposition 3.25]{tabuada-dgvssimplicial}.\footnote{To be completely rigorous, we should first generalize the cited result to dg-categories over commutative dg-rings.}
\end{proof}

\subsection{Modules and bimodules}
Let $\cat A$ and $\cat B$ be $R$-linear dg-categories.  The category of right $\cat A$-modules is denoted by $\compdg(\cat A)$. Recall that a right $\cat A$-module is a $R$-linear dg-functor
\[
\opp{\cat A}\to \compdg(R).
\]
We denote by $\dercomp(\cat A)$ the derived category of right $\cat A$-modules, which is enriched in $H^0(R)$-modules. Sometimes we will also use its graded version $\dercomp_{\operatorname{gr}}(\cat A)$ (the latter category is 
enriched in graded $H^*(R)$-modules). The full subcategory of $\compdg(\cat A)$ consisting of h-projective right $\cat A$-modules will denoted by $\hproj(\cat A)$. The dg-category $\hproj(\cat A)$ provides a dg-model for $\dercomp(\cat A)$. I.e.\ we have
\begin{align*}
\dercomp_{\operatorname{gr}}(\cat A)\cong H^\ast(\hproj(\cat A)), \\
\dercomp(\cat A) \cong H^0(\hproj(\cat A)).
\end{align*}
We also recall  that if
$f \colon \cat A \to \cat A'$ is a quasi-equivalence between
dg-categories, then the induction
$\Ind_f \colon \hproj(\cat A) \to \hproj(\cat A')$ is again a
quasi-equivalence. The inverse of $H^*(\Ind_f)$ is given by the
restriction functor
\begin{equation*}
    \Res_f \colon \dercomp_{\mathrm{gr}}(\cat A') \to \dercomp_{\mathrm{gr}}(\cat A),
\end{equation*}
upon identifying $H^*(\hproj(-))$ with $\dercomp_{\mathrm{gr}}(-)$.

\medskip

By considering suitable tensor products of dg-categories we may define analogous notions for bimodules. We spell this out explicitly. An \emph{$\cat A$-$\cat B$-dg-bimodule} is a $R$-linear dg-functor 
\begin{equation}
\label{eq:F}
    F \colon \opp {\cat B} \otimes {\cat A} \to \compdg(R).
\end{equation}
The dg-category $\compdg({\cat B}\otimes \opp {\cat A})$ of $\cat A$-$\cat B$-bimodules will also be denoted by $\compdg(\cat A,\cat B)$, thinking of bimodules as generalized morphisms from $\cat A$ to $\cat B$.

The tensor product $\opp {\cat B} \otimes {\cat A}$  is only compatible with quasi-equivalence if either $\cat A$ or $\cat B$ is h-flat  (this is true if they are cofibrant or  in general if $R$ is a field).
In general we should replace ${\cat B}\otimes \opp {\cat A}$ by ${\cat B}\lotimes \opp {\cat A}$.

If $\cat A$ or $\cat B$ is h-flat then we shall denote by $\hproj(\cat A,\cat B)$ the full dg-subcategory of $\compdg(\cat A, \cat B)$ of h-projective $\cat A$-$\cat B$-dg-bimodules.  
In this case we also denote by $\dercomp(\cat A,\cat B)$ the derived category of $\cat A$-$\cat B$-bimodules and by  $\dercomp_{\operatorname{gr}}(\cat A,\cat B)$  its graded counterpart.

\subsubsection{The covariant-contravariant convention}
Throughout we shall adopt the notation ``co\-variant-below, contravariant-above'', namely for $F$ as in \eqref{eq:F} we write:
\begin{equation}
    F_A^B = F(B,A).
\end{equation}
More generally, we shall adopt this notation for any dg-functor defined on any tensor product of dg-categories with values in $R$. For example, if $G \colon \opp{\cat B} \otimes \cat A \otimes \cat C \to \compdg(R)$, then we shall write
\begin{equation*}
    F_{A,C}^B = F(B,A,C).
\end{equation*}
Sometimes, we shall also denote by $\cat A$ or $h_{\cat A}$ or even $h$ the diagonal bimodule of a dg-category $\cat A$, namely:
\begin{equation}
    \cat A_A^B = h_A^B = \cat A(B,A).
\end{equation}

\subsubsection{(Co)end notation}

Let $\cat A$ be a dg-category, and let
\begin{equation*}
    T \colon \opp{\cat A} \otimes \cat A \to \compdg(R)
\end{equation*}
be an $\cat A$-$\cat A$-dg-bimodule. We shall use the \emph{end} notation
\begin{equation} \label{eq:end}
    \int_A T_A^A = \{ (\varphi_A)_A \in \prod_A T_A^A : f \varphi_A = (-1)^{|f||\varphi|}  \varphi_{A'} f \quad \forall\, f \in \cat A(A,A') \}. 
\end{equation}
and the \emph{coend} notation
\begin{equation} \label{eq:coend}
\begin{split}
    \int^A T_A^A = \coker \left( \bigoplus_{A_1,A_2 \in \cat A} \cat A_{A_1}^{A_2} \right. & \otimes T_{A_2}^{A_1} \left. \to \bigoplus_{A \in \cat A} T^A_A \right) \\
    f \otimes x & \mapsto fx - (-1)^{|f||x|} xf.
\end{split}
\end{equation}
(see Remark \ref{remark:coend_size} below for size issues). The complex of morphisms between two bimodules $F,G \in \compdg(\cat A,\cat B)$ is the following end:
\begin{equation}
    \compdg(\cat A, \cat B)(F,G) = \int_{A,B} \compdg(R) (F_A^B, G_A^B).
\end{equation}
Let $F \in \compdg(\cat A,\cat B)$ and $G \in \compdg(\cat B,\cat C)$ be dg-bimodules. Then, the tensor product $F \otimes_{\cat B} G \in \compdg(\cat A,\cat C)$ is given precisely by the following coend:
\begin{equation}
    (F \otimes_{\cat B} G)_A^C = \int^B F_A^B \otimes G^C_B.
\end{equation}
We remark that, if $F_A \cong \cat B(-,\Phi_F(A))$ in $\compdg(\cat B)$ (respectively in $H^0(\compdg(\cat B))$) for a suitable $\Phi_F(A) \in \cat B$, then we have:
\begin{equation} \label{eq:coYoneda}
\begin{split} 
    \int^B F_A^B \otimes G_B & \cong \int^B \cat B(B,\Phi_F(A)) \otimes G_B \\
    & \cong G_{\Phi_F(A)},
\end{split}
\end{equation}
with isomorphisms in $\compdg(\cat C)$ (respectively in $H^0(\compdg(\cat C))$.

A ``Fubini theorem'' holds for ends and coends, namely if $T \in \compdg(\cat A \otimes \cat B, \cat A \otimes \cat B)$, then we have natural isomorphisms
\begin{equation}
    \int_{A,B} T_{A,B}^{A,B} = \int_A \int_B T_{A,B}^{A,B} = \int_B \int_A T_{A,B}^{A,B},
\end{equation}
and analogously for coends. Moreover, ends and coends are compatible with hom-complexes:
\begin{align}
    \compdg(R)(V, \int_ A F_A^A) & \cong \int_A  \compdg(R)(V, F_A^A), \\
    \compdg(R)(\int^A F_A^A, V) & \cong \int_A \compdg(R)(F_A^A, V).
\end{align}
The reader can find a general and detailed account on (co)ends in \cite{loregian-coend}.
\begin{remark} \label{remark:coend_size} If $\cat A$ is not
  $\mathbb U$-small, then the (co)end of
  $T \colon \opp{\cat A} \otimes \cat A \to \compdg(R)$ will not be,
  in general, a $\mathbb U$-small $R$-dg-module, but still a
  $\mathbb V$-small $R$-dg-module (cf. \S
  \ref{subsec:sizes}). Whenever necessary, we will implicitly view the
  dg-category of $\mathbb U$-small dg-modules (over $R$ or even over
  suitable dg-categories) as a full dg-subcategory of the dg-category
  of $\mathbb V$-small dg-modules. We will not keep track of the
  universes in our notation.
\end{remark}
\subsection{Quasi-functors}
\label{ssec:qf}
Assume $\cat A$ and $\cat B$ are $R$-linear dg-categories with \emph{either $\cat A$ or $\cat B$} having h-flat complexes of morphisms.
A bimodule $F \in \compdg(\cat A, \cat B)$ is by definition right quasi-representable if $F_A$ is quasi-isomorphic to $\cat B(-,\Phi_F(A))$ for some $\Phi_F(A) \in \cat B$, for all $A \in \cat A$.  We shall denote by $\compdg^{\mathrm{rqr}}(\cat A,\cat B)$ the category of right quasi-representable $\cat A$-$\cat B$-dg-bimodules and  $\hproj^{\mathrm{rqr}}(\cat A,\cat B)=\hproj(\cat A, \cat B)\cap \compdg^{\operatorname{rqr}}(\cat A,\cat B)$.

We shall denote by $\dercomp^{\mathrm{rqr}}_{\mathrm{gr}}(\cat A, \cat B)$ the full (graded) subcategory of $\dercomp_{\mathrm{gr}}(\cat A,\cat B)$ spanned by right quasi-representable bimodules, which clearly satisfies
\begin{equation}
    \dercomp^{\mathrm{rqr}}_{\mathrm{gr}}(\cat A, \cat B) \cong H^*(\hproj^{\mathrm{rqr}}(\cat A, \cat B)).
\end{equation}
Right quasi-representable bimodules are also called \emph{quasi-functors}. In the next proposition we carefully check that $\hproj^\mathrm{rqr}(\cat A, \cat B)$ is compatible with quasi-equivalences in the first argument.
\begin{proposition} \label{prop: ind_rqr_firstargument} Let $\cat B$
  be a dg-category with h-flat complexes of morphisms. Let
  $f \colon \cat A' \to \cat A$ be a dg-functor. Consider the
  restriction (graded) functor along
  $f'=f \otimes 1 \colon \cat A' \otimes \opp{\cat B} \to \cat A
  \otimes \opp{\cat B}$:
\begin{align*}
\Res_{f'} \colon \dercomp_{\mathrm{gr}}(\cat A, \cat B) & \to \dercomp_{\mathrm{gr}}(\cat A', \cat B), \\
F \mapsto F_{f}. 
\end{align*}
Then, $\Res_{f'}$ restricts to a functor
\begin{equation}
    \Res_{f'} \colon \dercomp^{\mathrm{rqr}}_{\mathrm{gr}}(\cat A, \cat B)  \to \dercomp^{\mathrm{rqr}}_{\mathrm{gr}}(\cat A', \cat B).
\end{equation}
If $f$ is such that $H^0(f)$ is essentially surjective, then for $F \in \dercomp_{\mathrm{gr}}(\cat A, \cat B)$ we have:
\begin{equation*}
    \Res_{f'}(F) \in \dercomp^{\mathrm{rqr}}_{\mathrm{gr}}(\cat A', \cat B) \iff F \in \dercomp^{\mathrm{rqr}}_{\mathrm{gr}}(\cat A, \cat B).
\end{equation*}
In particular, if $f$ is a quasi-equivalence, then the equivalence $\Res_{f'} \colon \dercomp_{\mathrm{gr}}(\cat A, \cat B)  \to \dercomp_{\mathrm{gr}}(\cat A', \cat B)$ restricts to an equivalence
\[
\Res_{f'} \colon \dercomp^{\mathrm{rqr}}_{\mathrm{gr}}(\cat A, \cat B)  \to \dercomp^{\mathrm{rqr}}_{\mathrm{gr}}(\cat A', \cat B).
\]
Hence, the quasi-equivalence
\begin{equation*}
    \Ind_{f'} \colon \hproj(\cat A', \cat B) \to \hproj(\cat A, \cat B)
\end{equation*}
restricts to a quasi-equivalence
\begin{equation}
    \Ind_{f'} \colon \hproj^{\mathrm{rqr}}(\cat A', \cat B) \to \hproj^{\mathrm{rqr}}(\cat A,\cat B).
\end{equation}
\begin{proof}
Assume that $F \in \dercomp_{\mathrm{gr}}^\mathrm{rqr}(\cat A', \cat B)$, and let $A' \in \cat A'$. By assumption, there is a quasi-isomorphism of right $\cat B$-dg-modules
\begin{equation*}
    \cat B(-. \Phi_F(f(A'))) \xrightarrow{\mathrm{qis}} F_{f(A')},
\end{equation*}
which proves that $F_f = \Res_{f'}(F)$ is right quasi-representable.

Next, assume that $H^0(f)$ is essentially surjective, let $F \in \dercomp_{\mathrm{gr}}(\cat A, \cat B)$ and assume that $\Res_{f'}(F) \in \dercomp^{\mathrm{rqr}}_{\mathrm{gr}}(\cat A', \cat B)$. Let $A \in \cat A$. By hypothesis, there is $A' \in \cat A'$ and a closed degree $0$ morphism $f(A') \to A$ in $\cat A$ which induces an isomorphism in $H^0(\cat A)$. Hence, we get a morphism
\begin{equation}
    F_{f(A')} \to F_{A}
\end{equation}
which induces an isomorphism in $H^0(\compdg(\cat B))$. In particular, it is a quasi-isomorphism. By hypothesis, we then have a quasi-isomorphism
\begin{equation*}
    \cat B(-, \Phi_F(f(A'))) \xrightarrow{\mathrm{qis}} F_{f(A')} \xrightarrow{\mathrm{qis}} F_{A}
\end{equation*}
of right $\cat B$-dg-modules. This means that $F_f$ is right quasi-representable, as claimed.

Finally, assume that $f$ is a quasi-equivalence, and let $G \in \hproj^\mathrm{rqr}(\cat A',\cat B)$. By the above part of the proof, in order to prove that $\Ind_{f'}(G) \in \hproj^\mathrm{rqr}(\cat A, \cat B)$ it is enough to check that $\Res_{f'} \Ind_{f'}(G)$ is right quasi-representable. But $\Ind_{f'}$ is a quasi-equivalence and its inverse in cohomology is $\Res_{f'}$, so we have an isomorphism
\[
G \cong \Res_{f'} \Ind_{f'}(G)
\]
in $\dercomp_{\mathrm{gr}}(\cat A',\cat B)$. So, since $G$ is right quasi-representable, the same is true for $\Res_{f'} \Ind_{f'}(G)$. It remains to check that if $F \in H^0(\hproj^\mathrm{rqr}(\cat A, \cat B)) \cong \dercomp^\mathrm{rqr}(\cat A,\cat B)$, then there exists $G \in H^0(\hproj^\mathrm{rqr}(\cat A', \cat B)) \cong \dercomp^\mathrm{rqr}(\cat A',\cat B)$ such that $\Ind_{f'}(G) \cong F$ in $H^0(\hproj^\mathrm{rqr}(\cat A, \cat B))$. To this purpose, it is enough to take (an h-projective resolution of) $G=\Res_{f'}(F)$, thanks to the above part of the proof and the isomorphism
\[
F \cong \Ind_{f'}(G)
\]
in $\dercomp_{\mathrm{gr}}(\cat A,\cat B)$, recalling that $\Ind_{f'}$ is a quasi-equivalence.
\end{proof}
\end{proposition}
\begin{corollary} \label{coroll: ind_rqr_firstargument}
Let $\smallcat a$ and $\smallcat b$ be small dg-categories with either $\smallcat a$ or $\smallcat b$ having h-flat complexes of morphisms. Then, $\hproj^\mathrm{rqr}(\smallcat a, \smallcat b)$ is the internal hom in $\Hqe(R)$ between $\smallcat a$ and $\smallcat b$:
\begin{equation*}
    \hproj^\mathrm{rqr}(\smallcat a, \smallcat b) \cong \RHom(\smallcat a, \smallcat b).
\end{equation*}
\end{corollary}
\begin{proof}
This result is known in the case that $\smallcat a$ is h-flat \cite{canonaco-stellari-internalhoms,toen-morita}.

If $\smallcat b$ has h-flat $\Hom$-complexes then we let $q \colon Q(\smallcat a) \to \smallcat a$ be a quasi-equivalence with $Q(\smallcat a)$ having h-flat complexes of morphisms. Then we use the fact that, by Proposition \ref{prop: ind_rqr_firstargument}, the induction dg-functor
\begin{equation*}
    \Ind_{q'} \colon \hproj^{\mathrm{rqr}}(Q(\smallcat a),\smallcat b) \to \hproj^\mathrm{rqr}(\smallcat a, \smallcat b)
\end{equation*}
is a quasi-equivalence.
\end{proof}
\begin{remark}
Even for (locally small) dg-categories $\cat A$ and $\cat B$, with either $\cat A$ or $\cat B$ being h-flat, we may and will identify
\[
\hproj^\mathrm{rqr}(\cat A, \cat B) \cong \RHom(\cat A, \cat B).
\]
\end{remark}
\begin{remark}
In virtue of the above discussion, we shall sometimes abuse notation and write $\hproj^{\mathrm{rqr}}(\cat A,\cat B)$ for $\hproj^{\mathrm{rqr}}(Q(\cat A),\cat B)$ or $\hproj^{\mathrm{rqr}}(\cat A,Q(\cat B))$, if neither $\cat A$ nor $\cat B$ are h-flat. We will sometimes abuse notation similarly with regards to $\dercomp^\mathrm{rqr}(\cat A,\cat B)$.
\end{remark}
In the following remark, we collect a few useful facts and notations about quasi-functors.
\begin{remark} \label{remark:qfun_notations}
We say that two quasi-functors $F,G \colon \cat A \to \cat B$ are \emph{isomorphic} if they are isomorphic in $\dercomp^\mathrm{rqr}(\cat A,\cat B)$. Quasi-functors can be composed using coends, see for example \cite[\S 7]{genovese-adjunctions}. In particular, if $F \colon \cat A \to \cat B$ and $G \colon \cat B \to \cat C$ are quasi-functors, with either $F$ or $G$ being h-projective as a dg-bimodule, we have:
\begin{equation} \label{eq:qfun_composition}
    (G \circ F)_A^C = \int^B F_A^B \otimes G^C_B.
\end{equation}

If $F \colon \cat A \to \cat B$ is a quasi-functor, we will usually denote by $\Phi_F(A)$ the object in $\cat B$ which weakly represents $F_A$, with the quasi-isomorphism
\[
\cat B(-,\Phi_F(A)) \xrightarrow{\approx} F_A
\]
of right $\cat B$-dg-modules. Then, the quasi-functor $F$ yields a graded functor
\begin{align*}
H^*(F) \colon H^*(\cat A) & \to H^*(\cat B), \\
A &\mapsto \Phi_F(A).
\end{align*}

Let $F \colon \cat A \to \cat B$ be a quasi-functor. We say that $F$ is \emph{quasi-fully faithful} (resp. \emph{quasi-essentially surjective}, resp. a \emph{quasi-equivalence}) if the induced graded functor
\[
H^*(F) \colon H^*(\cat A) \to H^*(\cat B)
\]
is fully faithful (resp. essentially surjective, resp. an equivalence). It can be proved (for example using the dual discussed in the following \S \ref{sec:duality}) that the $F \colon \cat A \to \cat B$ is a quasi-equivalence if and only there is a quasi-functor $G \colon \cat B \to \cat A$ such that $GF \cong 1_{\cat A}$ in $\dercomp^\mathrm{rqr}(\cat A,\cat A)$ and $FG \cong 1_{\cat B}$ in $\dercomp^\mathrm{rqr}(\cat B, \cat B)$.
\end{remark}

\subsection{Duality and adjoint quasi-functors}
\label{sec:duality}
If $F \colon \cat A \to \cat B$ is a (strict) dg-functor between dg-categories, the definition of its opposite
\[
\opp{F} \colon \opp{\cat A} \to \opp{\cat B}
\]
is straightforward, and clearly $\opp{(\opp{F})}=F$. On the other hand, if $F \colon \cat A \to \cat B$ is a \emph{quasi-functor}, we have to be more careful to define $\opp{F}$, which will be a quasi-functor $\opp{\cat A} \to \opp{\cat B}$. The right way to do so is using the (derived) duality of dg-bimodules.
\begin{proposition}
Let $\cat A$ and $\cat B$ be dg-categories. There is a natural dg-functor
\begin{equation}
    \begin{split}
       \dual = \dual_{\cat A, \cat B} \colon \compdg(\cat A, \cat B) & \to \opp{\compdg(\opp{\cat A}, \opp{\cat B})}, \\
        F &\mapsto \dual_{\cat A, \cat B}(F)_A^B = \compdg(\cat B)(F_A, h_B).
    \end{split}
\end{equation}
($\dual_{\cat A, \cat B}(F)$ is ``covariant in $A \in \opp{\cat A}$'' and ``contravariant in $B \in \opp{\cat B}$''). This dg-functor $\dual_{\cat A, \cat B}$ has a right adjoint given by 
\[
\opp{\dual}_{\opp{\cat A}, \opp{\cat B}} \colon \opp{\compdg(\opp{\cat A}, \opp{\cat B})} \to \compdg(\cat A, \cat B).
\]
\end{proposition}
\begin{proof}
This is \cite[Proposition 5.4]{genovese-adjunctions}, rephrased.
\end{proof}
This duality of dg-bimodules is precisely what enables us to rigorously define opposite quasi-functors. More precisely, we can prove the following result.
\begin{proposition} \label{prop:duality_quasifunctors}
Let $\cat A$ and $\cat B$ be dg-categories, and assume that either $\cat A$ or $\cat B$ is h-flat. Then, the duality dg-functor $\dual_{\cat A, \cat B}$ induces a natural quasi-equivalence
\begin{equation} \label{eq:quasifunctors_duality}
    \dual = \dual_{\cat A, \cat B} \colon \hproj^\mathrm{rqr}(\cat A, \cat B)  \xrightarrow{\sim} \opp{\hproj^\mathrm{rqr}(\opp{\cat A}, \opp{\cat B})}.
\end{equation}
\end{proposition}
\begin{proof}
Clearly, $\dual = \dual_{\cat A, \cat B}$ restricts to a natural \emph{quasi-functor}\footnote{Note that as a \emph{dg-functor} this is not correct. $\dual_{\cat A, \cat B}$ does not take values in
 $\opp{\hproj(\opp{\cat A}, \opp{\cat B})}$.}.
 Note that 
\[
\dual_{\cat A, \cat B} \colon \hproj^\mathrm{rqr}(\cat A, \cat B)  \rightarrow \opp{\hproj(\opp{\cat A}, \opp{\cat B})},
\]
using suitable h-projective resolutions. $\dual_{\cat A, \cat B}$ is natural with respect to $\cat A$ and $\cat B$, so we can assume without loss of generality that both $\cat A$ and $\cat B$ are h-projective, if necessary. Then, we are left to check that the above quasi-functor induces an equivalence
\[
H^*(\dual) \colon \dercomp_{\mathrm{gr}}^\mathrm{rqr}(\cat A, \cat B) \xrightarrow{\sim} \opp{\dercomp_{\mathrm{gr}}^\mathrm{rqr}(\opp{\cat A}, \opp{\cat B})}.
\]
This is achieved using the techniques of \cite[Proposition 5.9]{genovese-adjunctions}.
\end{proof}
For a given quasi-functor $F \colon \cat A \to \cat B$, the dual $\dual(F) \colon \opp{\cat A} \to \opp{\cat B}$ is what we would want to call the \emph{opposite quasi-functor}. The above Proposition \ref{prop:duality_quasifunctors} says that taking opposites is a actually an involution, in a weak sense. Namely, we have a natural isomorphism
\begin{equation}
    \dual(\dual(F)) \cong F
\end{equation}
in the derived category $\dercomp^\mathrm{rqr}(\cat A, \cat B)$ of quasi-functors.
\begin{remark} \label{remark:left_qrep_bimod_dual}
We observe that $\dual(F) = \dual_{\cat A, \cat B}(F)$ can also be viewed as a $\cat B$-$\cat A$-dg-bimodule, namely an object of $\dercomp(\cat B, \cat A)$. If $F$ is a quasi-functor, then $\dual(F)$ has the property of being \emph{left quasi-representable} as a $\cat B$-$\cat A$-dg-bimodule, namely $\dual(F)^A$ is quasi-isomorphic to a representable $\cat B$-dg-bimodule, for all $A \in \cat A$.
\end{remark}
\begin{remark} \label{remark:dual_quasifun_cohomology}
From the definition, it is clear that for a given quasi-functor $F \colon \cat A \to \cat B$, then
\begin{equation}
    H^*(\dual(F))\cong \opp{H^*(F)} \colon \opp{H^*(\cat A)} \to \opp{H^*(\cat B)}
\end{equation}
\end{remark}

Next, we state some useful results on adjoint quasi-functors, which can be deduced from \cite{genovese-adjunctions}. First, we recall that given quasi-functors
\[
F \colon \cat A \to \cat B, \quad G \colon \cat B \to \cat A,
\]
we have that \emph{$F$ is left adjoint to $G$} (in symbols, $F \dashv G$) if there is an isomorphism
\begin{equation}
    \dual_{\cat A, \cat B}(F) \cong G
\end{equation}
in the derived category $\dercomp(\cat B, \cat A)$. We replace $F$ with a suitable h-projective resolution if necessary so that the dual $\dual(F)$ is correctly defined up to quasi-isomorphism. We recall \cite[Proposition 7.1]{genovese-adjunctions} that a quasi-functor $F$ has a left adjoint if and only if it is left quasi-representable as a bimodule (see the above Remark \ref{remark:left_qrep_bimod_dual}), and it has a right adjoint if and only if $\dual_{\cat A, \cat B}(F)$ is a quasi-functor $\cat B \to \cat A$ (namely, it is right quasi-representable as a $\cat B$-$\cat A$-dg-bimodule).

\begin{lemma} \label{lemma:adjoints_quasifunctors_cohomology}
Let $F \colon \cat A \to \cat B$ be a quasi-functor. Assume that $H^*(F) \colon H^*(\cat A) \to H^*(\cat B)$ has a right adjoint $R$ (resp. a left adjoint $L$) as a graded functor. Then, $F$ has a right adjoint $\widetilde{R}$ (resp. a left adjoint $\widetilde{L}$) as a quasi-functor. Moreover, $\widetilde{R}$ (resp. $\widetilde{L}$) satisfies $H^*(\widetilde{R}) \cong R$ (resp. $H^*(\widetilde{L}) \cong L$.
\end{lemma}
\begin{proof}
Upon taking a suitable resolution, we may assume that $F$ is h-projective as a bimodule. We may also assume that $\cat A$ is h-projective, so that $F_A$ is h-projective as $\cat B$-dg-bimodule for all $A \in \cat A$ \cite[Lemma 3.4]{canonaco-stellari-internalhoms}. $F$ being a quasi-functor means that for all $A \in \cat A$ there is a quasi-isomorphism
\[
\cat B(-,\Phi_F(A)) \xrightarrow{\sim} F_A,
\]
in $\compdg(\cat B)$, for a suitable $\Phi_F(A) \in \cat B$. This implies that there is a quasi-isomorphism
\[
\cat B(\Phi_F(A),-) \cong \dual(F)^A
\]
in $\compdg(\opp{\cat B})$, where we set $\dual(F)=\dual_{\cat A, \cat B}(F)$. This induces an isomorphism
\[
H^*(\cat B)(\Phi_F(A),B) \cong H^*(\dual(F)^A_B),
\]
natural in both $A \in H^*(\cat A)$ and $B \in H^*(\cat B)$. The functor $H^*(F)$ is actually given by $A \mapsto \Phi_F(A)$.

Now, assume that $H^*(F)$ has a right adjoint functor $R \colon H^*(\cat B) \to H^*(\cat A)$. This means that for all $B \in \cat B$ there is an object $R(B) \in \cat A$ and an isomorphism of graded $H^*(\cat A)$-modules: 
\begin{equation} \label{eq:graded_iso_adjoint}
H^*(\cat A)(-,R(B)) \xrightarrow{\sim} H^*(\cat B)(\Phi_F(-),B) \cong H^*(\dual(F))_B. \tag{$\ast$}
\end{equation}
By the graded Yoneda lemma (applied to $H^*(\cat A)$), this isomorphism corresponds to an element
\[
e \in H^0(\dual(F))^{R(B)}_B.
\]
Now, the derived Yoneda lemma tells that $e$ yields a morphism
\[
\cat A(-,R(B)) \to \dual(F)_B,
\]
in $\dercomp(\cat A)$, which actually induces the above isomorphism \eqref{eq:graded_iso_adjoint}. This means that $\dual(F)$ is right quasi-representable as a $\cat B$-$\cat A$-dg-bimodule, and it yields the right adjoint $\widetilde{R}$ of $F$. By definition, $\widetilde{R}$ satisfies $H^*(\widetilde{R}) \cong R$. The case where $H^*(F)$ has a left adjoint is dealt with analogously, with the goal of proving that $F$ is left quasi-representable.
\end{proof}
We add the following lemma for comfort.
\begin{lemma} \label{lemma:quasifunctor_adjoint_unit_cohomology}
  Let $F \colon \cat A \to \cat B$ and $G \colon \cat B \to \cat A$ be quasi-functors. Assume there exists $\eta \colon 1 \to GF$ in $\dercomp(\cat A, \cat A)$ such that $H^*(\eta) \colon 1 \to H^*(GF)$ is the unit morphism of an adjunction $H^*(F) \dashv H^*(G)$. Then, $\eta$ is the unit morphism of an adjunction of quasi-functors $F \dashv G$.
\end{lemma}
\begin{proof}
We may assume that both $F$ and $G$ are h-projective as dg-bimodules. We will also assume, without loss of generality, that $\cat A$ is an h-projective dg-category. We consider the following morphism $\dual(F) \to G$ in $\dercomp(\cat B, \cat A)$, given by:
\begin{equation} \label{eq:quasifunct_adjoint_qis_unit}
\dual(F)^A_B \cong \compdg(\cat B)(F_A, h_B) \xrightarrow{G} \compdg(\cat A)((GF)_A, G_B) \xrightarrow{\eta^*} \compdg(\cat A)(h_A, G_B) \cong G^A_B.
\end{equation}
We remark that $G$ (or, more precisely, tensoring with $G$) and $\eta^*$ are well defined as morphisms in $\dercomp(\cat B, \cat A)$ thanks to the fact that $F$ and $G$ are h-projective, and $\cat A$ is h-projective. Indeed, we may identify $\eta$ as the following ``roof'' of morphisms of $\cat A$-$\cat A$-dg-bimodules:
\[
\cat A \xleftarrow{\sim} Q(\cat A) \to GF,
\]
where $Q(\cat A) \to \cat A$ is an h-projective resolution of the diagonal bimodule. Since $\cat A$ is h-projective, the $\cat A$-dg-module $Q(\cat A)_A$ is h-projective for all $A \in \cat A$ \cite[Lemma 3.4]{canonaco-stellari-internalhoms}, and the same holds for $h_A=\cat A(-,A)$ for all $A \in \cat A$, thanks to the Yoneda Lemma. Hence, the induced morphism
\[
\compdg(\cat A)(Q(\cat A)_A, G_B) \leftarrow \compdg(\cat A)(h_A, G_B) \cong G^A_B
\]
is a quasi-isomorphism, natural in $A \in \cat A$ and $B \in \cat B$. We conclude that
\[
\eta ^* \colon \compdg(\cat A)((GF)_A, G_B) \to \compdg(\cat A)(Q(\cat A)_A, G_B) \xleftarrow{\sim} \compdg(\cat A)(h_A, G_B) \cong G^A_B
\]
gives a genuine morphism in $\dercomp(\cat B, \cat A)$.

Now, we conclude if we prove that \eqref{eq:quasifunct_adjoint_qis_unit} is a quasi-isomorphism (compare with \cite[\S 6]{genovese-adjunctions}). Indeed, the morphism $H^*(\dual(F)) \to H^*(G)$ can be identified with
\[
H^*(\cat B)(\Phi_F(A), B) \xrightarrow{H^*(G)} H^*(\cat A)(\Phi_G \Phi_F(A),\Phi_F(B)) \xrightarrow{H^*(\eta)^*} H^*(\cat A)(A,\Phi_G(B)),
\]
where $\Phi_F(A)$ and $\Phi_G(B)$ are the objects quasi-representing $F_A$ and $G_B$. By hypothesis, the above composition $H^*(\eta)^* \circ H^*(G)$ is an isomorphism, so \eqref{eq:quasifunct_adjoint_qis_unit} is actually a quasi-isomorphism, and we conclude.
\end{proof}
\begin{remark} \label{remark:adjoints_dual_bicategory_bimodules}
If $F \colon \cat A \to \cat B$ is a quasi-functor, we know from \cite[Corollary 6.6]{genovese-adjunctions} that there is an adjunction $F \dashv \dual_{\cat A, \cat B}(F)$ in the ``derived bicategory of dg-bimodules''. In particular, we remark that for all $G \in \dercomp(\cat A, \cat B)$, there is a natural isomorphism
\begin{equation}
\dercomp(\cat A, \cat B)(F,G) \cong \dercomp(\cat A, \cat A)(h_{\cat A}, \dual(F) \circ G).
\end{equation}
\end{remark}

\subsection{Truncations of dg-categories}
\label{sec:trunc}
In this part, we assume our base dg-ring $R$ to be strictly concentrated in nonpositive degrees. 

Let $\cat A$ be a dg-category. We define a dg-category $\tau^{\leq 0} \cat A$ as follows: $\Ob \tau^{\leq 0} \cat A = \Ob \cat A$, and for any $A,B \in \Ob \cat A$ we set
\begin{equation*}
    (\tau^{\leq 0} \cat A)(A,B) = \tau_{\leq 0} (\cat A(A,B)),
\end{equation*}
where $\tau_{\leq 0}$ is the (smart) truncation of $R$-dg-modules. More explicitly:
\begin{equation*}
\begin{cases}
(\tau^{\leq 0} \cat A)(A,B)^n = \cat A(A,B)^n & \text{if $n < 0$}, \\
(\tau^{\leq 0} \cat A)(A,B)^n = 0 & \text{if $n > 0$}, \\
(\tau^{\leq 0} \cat A)(A,B)^0 = Z^0(\cat A(A,B)).
\end{cases}
\end{equation*}
There is a natural dg-functor
\begin{equation}
i^{\leq 0} \colon \tau^{\leq 0} \cat A \rightarrow \cat A,
\end{equation}
which is the identity on objects and the inclusion map $\tau_{\leq 0} \cat A(A,B) \hookrightarrow \cat A(A,B)$ on hom-complexes. Moreover, there is a natural dg-functor
\[
\tau^{\leq 0}\cat A \to H^0(\cat A),
\]
which is the identity on objects and the natural projection $\tau_{\leq 0} \cat A(A,B) \to H^0(\cat A(A,B))$ on morphisms. If $\cat A$ is cohomologically concentrated in nonpositive degrees, then $\tau^{\leq 0}\cat A \to \cat A$ is a quasi-equivalence; if $\cat A$ is cohomologically concentrated in degree $0$, then we have quasi-equivalences
\[
H^0(\cat A) \leftarrow \tau^{\leq 0} \cat A \to A.
\]

The construction $\cat A \mapsto \tau^{\leq 0} \cat A$ is clearly functorial, namely, given a dg-functor $F \colon \cat A \to \cat B$ we get an induced dg-functor $\tau^{\leq 0} F \colon \tau^{\leq 0} \cat A \to \tau^{\leq 0} \cat B$, which in particular is the same as $F$ on objects.

We now mention a result which describes the quasi-functors from a dg-category which has cohomology concentrated in degree $0$.
\begin{proposition} \label{prop:H0_equiv_iftarget_H0}
Let $\cat B$ be a ($R$-linear) dg-category which is cohomologically concentrated in degree $0$, and let $\cat A$ be a dg-category which is cohomologically concentrated in nonpositive degrees. Assume that either $\cat A$ or $\cat B$ is h-flat. Then, we have a natural equivalence of categories:
\begin{equation} \label{eq:H0underlyingfunctor}
\begin{split}
    H^0(\hproj^{\mathrm{rqr}}(\cat A, \cat B)) & \xrightarrow{\sim} \Fun_{H^0(R)}(H^0(\cat A),H^0(\cat B)), \\
    F & \mapsto H^0(F)
\end{split}
\end{equation}
\end{proposition}
\begin{proof}
Thanks to Lemma \ref{lemma:cofibrantresolutions_nonpositive} and Proposition \ref{prop: ind_rqr_firstargument} we may assume that $\cat A$ is an h-flat dg-category which is strictly concentrated in nonpositive degrees. Moreover, we can and will replace $\hproj^{\mathrm{rqr}}(\cat A, \cat B)$ with $\hproj^{\mathrm{rqr}}(\cat A, H^0(\cat B))$, since $\cat B$ is quasi-equivalent to $H^0(\cat B)$ by hypothesis.

We now continue the proof under these assumptions.  We first show
essential surjectivity. Let $G \colon H^0(\cat A) \to H^0(\cat B)$ be
an $H^0(R)$-linear functor. We view both $H^0(\cat A)$ and
$H^0(\cat B)$ as $R$-linear dg-categories via the restriction along
$R \to H^0(R)$. Then, taking the composition
\[
F \colon \cat A \to H^0(\cat A) \xrightarrow{G} H^0(\cat B),
\]
it is clear that we get a quasi-functor (indeed, a dg-functor) $F$ such that $H^0(G) \cong F$.

We now concentrate on fully faithfulness. First, we view $H^0(\hproj^{\mathrm{rqr}}(\cat A, H^0(\cat B))$ as a subcategory of the derived category $\dercomp(\opp{\cat A} \otimes H^0(\cat B))$. By assumption, $\opp{\cat A} \otimes H^0(\cat B)$ is (strictly) concentrated in nonpositive degrees, so the derived category $\dercomp(\opp{\cat A} \otimes H^0(\cat B))$ has the natural t-structure of Proposition \ref{prop:dercomp_naturaltstruct}. If $F \in H^0(\hproj^{\mathrm{rqr}}(\cat A, H^0(\cat B))$, then by hypothesis there is a quasi-isomorphism
\[
H^0(\cat B)(-,\Phi_F(A)) \xrightarrow{\sim} F_A,
\]
for all $A \in \cat A$. In particular, for all $A \in \cat A$ and $B \in H^0(\cat B)$, this means that
\[
H^k(F_A^B)=0,
\]
for $k \neq 0$. So, $F$ lives in the heart $\Mod(H^0(\opp{\cat A} \otimes H^0(\cat B))$. Hence, we have a commutative diagram of categories and functors:
\[
\begin{tikzcd}
{H^0(\hproj^\mathrm{rqr}(\cat A, H^0(\cat B)))} \arrow[r, "H^0"] \arrow[d, hook] & {\Fun(H^0(\cat A),H^0(\cat B))} \arrow[d, hook] \\
\Mod(H^0(\opp{\cat A} \otimes_R H^0(\cat B))) \arrow[r]                                & \Mod(\opp{H^0(\cat A)} \otimes_{H^0(R)} H^0(\cat B)),    
\end{tikzcd}
\]
where the left vertical fully faithful functor is obtained by viewing a functor $H^0(\cat A) \to H^0(\cat B)$ as a $H^0(\cat A)$-$H^0(\cat B)$-bimodule. In particular, if we are able to prove that the lower horizontal functor
\begin{equation}
    \Mod(H^0(\opp{\cat A} \otimes_R H^0(\cat B))) \to \Mod(\opp{H^0(\cat A)} \otimes_{H^0(R)} H^0(\cat B)),
\end{equation}
induced by restriction along the natural functor
\begin{equation} \label{eq:functor_comparison_H0}
\begin{split}
\opp{H^0(\cat A)} \otimes_{H^0(R)} H^0(\cat B) &\to H^0(\opp{\cat A} \otimes_R H^0(\cat B)), \\
[f] \otimes [g] & \mapsto [f \otimes g],
\end{split}
\end{equation}
is fully faithful, then we immediately conclude. This claim follows from the fact that \eqref{eq:functor_comparison_H0} is actually an isomorphism, thanks to Lemma \ref{lemma:H^0_comparison_complexes} below.
\end{proof}

Below, we prove Lemma \ref{lemma:H^0_comparison_complexes} together with a few other technical results about derived tensor products of dg-modules, which will be useful later on.

\begin{lemma} \label{lemma:tensorproduct_nonpositive}
Let $V$ and $W$ two $R$-dg-modules which have cohomology concentrated in nonpositive degrees. Then, the same is true for the derived tensor product $V \lotimes_R W$.
\end{lemma}
\begin{proof}
By hypothesis, the natural morphisms $\tau_{\leq 0}V \to V$ and $\tau_{\leq 0}W \to W$ are quasi-isomorphisms. Next, we may replace $\tau_{\leq 0} V$ with an h-flat complex strictly concentrated in nonpositive degrees $Q(\tau_{\leq 0} V)$ (recall that $R$ is strictly concentrated in nonpositive degrees), and compute:
\[
V \lotimes_R W \cong Q(\tau_{\leq 0} V) \otimes_R \tau_{\leq 0} W.
\]
 Then, the result is immediate.
\end{proof}
\begin{lemma} \label{lemma:tensorproduct_nonpositive_surjectivemorphism} 
Let $W$ be an $R$-dg-module cohomologically concentrated in nonpositive degrees, and let $f \colon V \to V'$ be a morphism in $\dercomp(R)$ between dg-modules cohomologically concentrated in nonpositive degrees, such that the induced morphism
\[
H^0(f) \colon H^0(V) \to H^0(V')
\]
is surjective. Then, the morphism
\[
H^0(f \otimes 1) \colon H^0(V \lotimes_R W) \to H^0(V' \lotimes_R W)
\]
is also surjective.
\end{lemma}
\begin{proof}
We consider the following distinguished triangle in $\dercomp(R)$:
\[
\cone(f)[-1] \to V \xrightarrow{f} V'.
\]
Taking the long exact sequence in cohomology and using that $H^0(f)$ is surjective, we see that $\cone(f)[-1]$ is also cohomologically concentrated in nonpositive degrees. By tensoring $- \lotimes_R W$, we obtain a distinguished triangle
\[
\cone(f)[-1] \lotimes_R W \to V \lotimes_R W \xrightarrow{f \otimes 1} V' \lotimes_R W 
\]
Every object in the above distinguished triangle is cohomologically concentrated in nonpositive degrees, being a derived tensor product of dg-modules with such property (cf. Lemma \ref{lemma:tensorproduct_nonpositive}). Taking the long exact sequence in cohomology, we conclude that $H^0(f \otimes 1)$ is indeed surjective.
\end{proof}
\begin{lemma} \label{lemma:H^0_comparison_complexes}
Let $V$ and $W$ be $R$-dg-modules which are cohomologically concentrated in nonpositive degrees. Then, the natural ($H^0(R)$-linear) morphism
\begin{equation}
\begin{split} \label{eq:H^0_comparison_complexes}
H^0(V) \otimes_{H^0(R)} H^0(W) & \to H^0(V \lotimes_R W), \\
[f] \otimes [g] & \mapsto [f \otimes g],
\end{split}
\end{equation}
is an isomorphism.
\end{lemma}
\begin{proof}
As in the proof of Lemma \ref{lemma:tensorproduct_nonpositive}, we may replace $V$ and $W$ with h-flat dg-modules strictly concentrated in nonpositive degrees, so that we may identify $V \lotimes W = V \otimes W$. Then, the result follows from a diagram chase, which is left to the reader.
\end{proof}

The following result is a precise formulation of the intuition that if $\cat A$ is concentrated in non-positive degrees then a quasi-functor
$\cat A\rightarrow \cat B^{\le 0}$ is the same as a quasi-functor $\cat A\rightarrow \cat B$, concentrated in non-positive degrees (recall that a quasi-functor is a bimodule). The proof turns out to be a bit involved. 

\begin{lemma} \label{lemma:truncations_image_qfun}
Let $\cat A$ and $\cat B$ be dg-categories, and assume that $\cat A$ has cohomology concentrated in nonpositive degrees and h-projective complexes of morphisms. We denote by $i^{\leq 0} \colon \tau^{\leq 0} \cat B \to \cat B$ the natural dg-functor. This induces a quasi-equivalence
\begin{equation*}
    \tau^{\leq 0} i^{\leq 0}_* \colon \tau^{\leq 0} \hproj^{\mathrm{rqr}}(\cat A, \tau^{\leq 0}\cat B) \to \tau^{\leq 0}\hproj^{\mathrm{rqr}}(\cat A, \cat B).
\end{equation*}
\end{lemma}
\begin{proof}
By assumption, the natural dg-functor $\tau^{\leq 0} \cat A \to \cat A$ is a quasi-equivalence. Using this and suitable resolutions (in particular, Lemma \ref{lemma:cofibrantresolutions_nonpositive}), we can assume that $\cat A$ is strictly concentrated in nonpositive degrees.

The dg-functor $i^{\leq 0}_*$ is given by induction along
\begin{equation*}
    \opp{\cat A} \otimes \tau^{\leq 0} \cat B \xrightarrow{1 \otimes i^{\leq 0}} \opp{\cat A} \otimes \cat B.
\end{equation*}
It can also be described by the following formula:
\begin{equation}
    (i^{\leq 0}_*(F))_A^B \cong \int^{B' \in \tau^{\leq 0} \cat B} F_A^{B'} \otimes \cat B(B, i^{\leq 0}(B')),
\end{equation}
which also satisfies the adjunction property:
\begin{equation} \label{eq:truncations_image_qfun_strict_adj}
    \compdg(\cat A, \cat B)(i^{\leq 0}_*(F), G) \cong \compdg(\cat A, \tau^{\leq 0}\cat B)(F, G^{i^{\leq 0}}). \tag{$\star$}
\end{equation}
We remark that if $F$ is a quasi-functor and h-projective as a bimodule, then the same is true for $i^{\leq 0} _*F$. In particular, we have a quasi-isomorphism of right $\cat B$-dg-modules:
\[
\cat B(-,i^{\leq 0}(\Phi_F(A))) \xrightarrow{\sim} (i^{\leq 0}_*F)_A,
\]
for a suitable $\Phi_F(A) \in \Ob(\tau^{\leq 0}\cat B)$, for all $A \in \cat A$. So, if $F$ is h-projective, taking nonpositive cohomology in \eqref{eq:truncations_image_qfun_strict_adj} we get an isomorphism
\[
\dercomp(\cat A, \cat B)(i^{\leq 0}_*(F)[n], G) \cong \dercomp(\cat A, \tau^{\leq 0} \cat B)(F[n], G^{i^{\leq 0}})
\]
for $n \geq 0$. Now, since both $\cat A$ and $\tau^{\leq 0} \cat B$ are concentrated in nonpositive degrees, the derived category $\dercomp(\cat A, \tau^{\leq 0} \cat B)$ is endowed with the natural t-structure. Moreover, since $F$ is a quasi-functor, we have:
\[
H^i(F_A^B) \cong H^i(\tau^{\leq 0}\cat B(B, \Phi_F(A)) = 0,
\]
if $i>0$. So, $F[n]$ lies in the aisle $\dercomp(\cat A, \tau^{\leq 0}\cat B)_{\leq 0}$ for all $n\geq 0$, and we have an isomorphism
\[
\dercomp(\cat A, \tau^{\leq 0} \cat B)(F[n], G^{i^{\leq 0}}) \cong \dercomp(\cat A, \tau^{\leq 0} \cat B)(F[n], \tau_{\leq 0}G^{i^{\leq 0}}).
\]
In the end, we get an isomorphism
\[
\dercomp(\cat A, \cat B)(i^{\leq 0}_*(F)[n], G) \cong \dercomp(\cat A, \tau^{\leq 0} \cat B)(F[n], \tau_{\leq 0}G^{i^{\leq 0}})
\]
for $n \geq 0$. Moreoer, if $G$ is a quasi-functor, we see that $\tau_{\leq 0} G^{i^{\leq 0}}$ is also a quasi functor, indeed for all $A \in \cat A$ we have:
\begin{align*}
    \tau_{\leq 0}G_A^{i^{\leq 0}} & \cong \tau_{\leq 0}\cat B(i^{\leq 0}(-),\Phi_G(A)) \\
    &= (\tau^{\leq 0}\cat B)(-,\Phi_G(A)),
\end{align*}
for a suitable object $\Phi_G(A) \in \Ob(\cat B)=\Ob(\tau^{\leq 0}\cat B)$. This implies that the (graded) functor
\[
H^*(\tau^{\leq 0} i^{\leq 0}_*) \colon H^*(\tau^{\leq 0} \hproj^{\mathrm{rqr}}(\cat A, \tau^{\leq 0}\cat B)) \to H^*(\tau^{\leq 0}\hproj^{\mathrm{rqr}}(\cat A, \cat B))
\]
has a right adjoint defined (up to suitable resolutions) by
\[
G \mapsto \tau_{\leq 0} G^{i^{\leq 0}}.
\]

To conclude, we show that both the unit and the counit morphisms of this adjunction are isomorphisms in the suitable (graded) categories. First, we consider
\[
\eta_F \colon F \to \tau_{\leq 0} (i_*^{\leq 0}F)^{i^{\leq 0}},
\]
for a given quasi-functor $F \colon \cat A \to \tau^{\leq 0}\cat B$. We want to show that $H^*(\eta_F)$ is an isomorphism. For all $A \in \cat A$, we have a commutative diagram:
\[
\begin{tikzcd}
H^*(F_A) \arrow[r, "H^*(\eta_F)_A"] \arrow[d, "\sim"]                     & H^*(\tau_{\leq 0} (i_*^{\leq 0}F)^{i^{\leq 0}}_A) \arrow[d, "\sim"]     \\
{H^*(\tau^{\leq 0}\cat B(-,\Phi_F(A)))} \arrow[r] & {H^*(\tau_{\leq 0} \cat B(i^{\leq 0}(-), i^{\leq 0}(\Phi_F(A)))}
\end{tikzcd}
\]
the lower horizontal arrow is clearly an isomorphism, hence the same is true for $H^*(\eta_F)_A$, for all $A$, as we wanted.

Now, we consider:
\[
\varepsilon_G \colon i^{\leq 0}_*(\tau_{\leq 0} G^{i^{\leq 0}}) \to G,
\]
for a given quasi-functor $G \colon \cat A \to \cat B$. We want to show that $H^*(\epsilon_G)$ is an isomorphism. For all $A \in \cat A$, we have a commutative diagram:
\[
\begin{tikzcd}
H^*(i^{\leq 0}_*(\tau_{\leq 0} G^{i^{\leq 0}})_A) \arrow[r, "H^*(\varepsilon_G)_A"] \arrow[d, "\sim"]                     & H^*(G_A) \arrow[d, "\sim"]     \\
{H^*(\cat B(-,\Phi_G(A)))} \arrow[r, equal] & {H^*(\cat B(-,\Phi_G(A)))}
\end{tikzcd}
\]
The left vertical isomorphism is obtained recalling that we have a quasi-isomorphism
\[
\tau_{\leq 0} G^{i^{\leq 0}}_A \cong (\tau^{\leq 0}\cat B)(-,\Phi_G(A))
\]
and hence a quasi-isomorphism
\[
i^{\leq 0}_*(\tau_{\leq 0} G^{i^{\leq 0}})_A \cong \cat B(-,i^{\leq 0}(\Phi_G(A))) = \cat B(-,\Phi_G(A)).
\]
We conclude that $H^*(\varepsilon_G)_A$ is an isomorphism for all $A \in \cat A$, as we wanted.
\end{proof}

% t-structures on dg-categories

\section{t-structures on dg-categories} \label{sec:tstruct_dgcat}
Throughout this section we shall work with $R$-linear dg-categories, where $R$ is a commutative dg-ring (strictly) concentrated in nonpositive degrees. We recall that the homotopy category of an $R$-linear dg-category is naturally an $H^0(R)$-linear category.
\subsection{t-structures on derived categories}
We recall the following result:
\begin{proposition} \label{prop:dercomp_naturaltstruct}
Let $\smallcat a$ be a small dg-category with cohomology concentrated in nonpositive degrees. Then, the derived category $\dercomp(\smallcat a)$ has a natural (non-degenerate) t-structure such that:
\begin{align*}
    \dercomp(\smallcat a)_{\leq 0} &= \{ M \in \dercomp(\smallcat a) : H^i(M) = 0 \ \forall\, i > 0 \}, \\
    \dercomp(\smallcat a)_{\geq 0} &= \{ M \in \dercomp(\smallcat a) : H^i(M) = 0 \ \forall\, i < 0 \}.
\end{align*}
The heart $\dercomp(\smallcat a)$ is identified with the abelian category $\Mod(H^0(\smallcat a))$.
\end{proposition}
\begin{proof}
This can be proved by adapting \cite[Lemma 2.2, Proposition 2.3]{amiot-cluster}.
\end{proof}
We remark that in particular the derived category $\dercomp(R)$ of the base dg-ring is endowed with such natural t-structure. 

We now note the simple result:
\begin{lemma} \label{lemma:tstruct_tensorprod}
Let $\smallcat a$ and $\smallcat b$ be small dg-categories with cohomology concentrated in nonpositive degrees. Then, the derived tensor product $\smallcat a \otimes^{\mathbb L} \smallcat b$ has cohomology concentrated in nonpositive degrees. In particular, $\dercomp(\smallcat a \otimes^{\mathbb L} \smallcat b)$ can be endowed with the natural t-structure of Proposition \ref{prop:dercomp_naturaltstruct}.
\end{lemma} 
\begin{proof}
This follows for example from Lemma \ref{lemma:cofibrantresolutions_nonpositive}, upon replacing $\smallcat a$ and $\smallcat b$ with the quasi-equivalent truncations $\tau^{\leq 0}\smallcat a$ and $\tau^{\leq 0} \smallcat b$.
\end{proof}

\subsection{t-structures on dg-categories}
t-structures on dg-categories are first understood as t-structures on their homotopy categories.
\begin{definition}
A \emph{t-structure} on a pretriangulated dg-category $\cat A$ is by definition a t-structure on the homotopy category $H^0(\cat A)$ (see \cite{beilinson-pervers}). Given a quasi-functor $F \colon \cat A \to \cat B$ between dg-categories with t-structures, we say that it is \emph{t-exact}  if $H^0(F)$ is, namely if it commutes with the truncation functors (or, equivalently, if it preserves the aisles).
\end{definition}
We shall  denote by $\cat A_{\leq n}$ and $\cat A_{\geq n}$ 
the full
  dg-subcategories of $\cat A$ whose objects are the same as the
  aisles $H^0(\cat A)_{\leq n}$ and $H^0(\cat A)_{\geq n}$. Note that, being full dg-subcategories of the $R$-linear category $\cat A$,  $\cat A_{\leq n}$ and $\cat A_{\geq n}$  are
  in particular $R$-linear.
  
\begin{remark} \label{remark:tstruct_dgcat_aisles_heart}
The full dg-subcategory $\cat A_{\leq 0} \cap \cat A_{\geq 0}$ of $\cat A$ is clearly a dg-enhancement of the heart:
\[
H^0(\cat A_{\leq 0} \cap \cat A_{\geq 0}) = H^0(\cat A)^\heartsuit.
\]
It is worth noticing that $\cat A_{\leq 0} \cap \cat A_{\geq 0}$ has cohomology concentrated in \emph{nonnegative degrees}. Indeed, if $A,B \in \cat A_{\leq 0} \cap \cat A_{\geq 0}$:
\begin{align*}
    H^k(\cat A(A,B))=H^0(\cat A(A,B[k])=0
\end{align*}
for all $k< 0$.
\end{remark}
\subsection{Truncation functors} \label{subsec:truncationfunctors}
We fix a t-structure $(\cat A_{\leq 0}, \cat A_{\geq 0})$ on a dg-category $\cat A$. At the homotopy category level, we know that truncation functors are suitable adjoints of the inclusions $H^0(\cat A_{\leq n}) \hookrightarrow H^0(\cat A)$ and $H^0(\cat A_{\geq n}) \hookrightarrow H^0(\cat A)$. One would hope that those truncation functors lift to quasi-functors $\cat A \to \cat A_{\leq n}$ and $\cat A \to \cat A_{\geq n}$, but this is not the case. What we \emph{do} obtain are truncation quasi-functors
\begin{align*}
    \tau_{\leq n} \colon \tau^{\leq 0} \cat A & \to \tau^{\leq 0} (\cat A_{\leq n}), \\
    \tau_{\geq n} \colon \tau^{\leq 0} \cat A & \to \tau^{\leq 0} (\cat A_{\geq n}),
\end{align*}
where $\tau^{\leq 0} \cat B$ if the truncation of $\cat B$ in degrees $\le 0$, which was defined in \S\ref{sec:trunc}.
\begin{proposition} \label{prop:truncations_adjoints}
Let 
\begin{equation}
    \begin{split}
        i_{\leq n} \colon \tau^{\leq 0} (\cat A_{\leq n}) & \hookrightarrow \tau^{\leq 0} \cat A, \\
        i_{\geq n} \colon \tau^{\leq 0} (\cat A_{\geq n}) & \hookrightarrow \tau^{\leq 0}\cat A
    \end{split}
\end{equation}
be the inclusion dg-functors. Then, $i_{\leq n}$ has a quasi-functor right adjoint
\begin{equation}
    \tau_{\leq n} \colon \tau^{\leq 0} \cat A \to \tau^{\leq 0}( \cat A_{\leq n}),
\end{equation}
and $i_{\geq n}$ has a quasi-functor left adjoint
\begin{equation}
    \tau_{\geq n} \colon \tau^{\leq 0} \cat A \to \tau^{\leq 0} (\cat A_{\geq n}).
\end{equation}
Hence, $i_{\leq n} \dashv \tau_{\leq n}$ and $\tau_{\geq n} \dashv i_{\geq n}$.
\end{proposition}
\begin{proof}
 Thanks to Lemma \ref{lemma:adjoints_quasifunctors_cohomology}, it is enough to find a right adjoint of $H^*(i_{\leq n})$ and a left adjoint of $H^*(i_{\geq n})$. We concentrate on the case of $H^*(i_{\leq n})$, the other being similar. We may use the well-known fact that $H^0(i_{\leq n})$ has a right adjoint (denoted by $\tau_{\leq n}$ on objects) and compute, for $k \leq 0$:
\begin{align*}
    H^k(\cat A(i_{\leq n}(X),Y) &\cong H^0(\cat A)(i_{\leq n}(X[-k]),Y) \quad \text{($X[-k]$ lies in $\cat A_{\leq n}$)} \\
    & \cong H^0(\cat A_{\leq n})(X[-k], \tau_{\leq n}(Y)) \quad \text{(right adjoint of $H^0(i_{\leq n})$)} \\
    & \cong H^k(\cat A(X,\tau_{\leq n}(Y)).
\end{align*}
The above isomorphisms give an isomorphism
\[
H^*(\tau^{\leq 0}\cat A(i_{\leq n}(X),Y)) \cong H^*(\tau^{\leq 0}(\cat A_{\leq n})(X, \tau_{\leq n}(Y))),
\]
natural in $X \in H^*(\tau^{\leq 0}(\cat A_{\leq n}))$ and $Y \in H^*(\tau^{\leq 0} \cat A)$, thus yielding a right adjoint $Y \mapsto \tau_{\leq n} Y$ of $H^*(i_{\leq n})$.
\end{proof}

We now want to discuss the higher functoriality of the distinguished triangle $A'' \to A \to A'$ associated to any $A \in \cat A$. We first prove a general lemma.
\begin{lemma} \label{lemma:qfunct_orthogonal}
Let $\cat C$ and $\cat D$ be dg-categories, and let $F, G \colon \cat C \to \cat D$ be quasi-functors. Moreover, let $\cat D_1$ and $\cat D_2$ be full dg-subcategories of $\cat D$ such that:
\begin{itemize}
\item The quasi-essential image of $F$ lies in $\cat D_1$ and the quasi-essential image of $G$ lies in $\cat D_2$.
\item For all $D_1 \in \cat D_1$ and $D_2 \in \cat D_2$, the complex $\cat D(D_1,D_2)$ is acyclic:
\[
H^*(\cat D(D_1,D_2))=0.
\]
\end{itemize}
Then,
\[
\dercomp(\cat C,\cat D)(F,G)=0
\]
\end{lemma}
\begin{proof}
We can assume without loss of generality that $\cat C$ is h-projective, and that $F$ and $G$ are h-projective bimodules. Consider the dual bimodule of $F$: 
\[
\dual(F)^C_D = \compdg(\cat D)(F_C, h_D).
\]
From \cite[Corollary 6.6]{genovese-adjunctions} we have that $\dual(F)$ is the right adjoint of $F$ in the (derived) bicategory of dg-bimodules. In particular this entails a natural isomorphism
\[
\dercomp(\cat C, \cat D)(F,G) \cong 
\dercomp(\cat C, \cat C)(h_{\cat C},\dual(F) \circ G).
\]
Next, we have isomorphisms in $H^0(\compdg(R))$:
\begin{align*}
(\dual(F) \circ G)_C^{C'} & = \int^D G_C^D \otimes \dual(F)^{C'}_D \qquad \text{(cf. \eqref{eq:qfun_composition})} \\
& \cong \int^D \cat D (D,\Phi_G(C)) \otimes \dual(F)^{C'}_D \qquad \text{($G_C \cong \cat D(-,\Phi_G(C))$ in $H^0(\compdg(\cat D))$)}\\
&\cong \dual(F)^{C'}_{\Phi_G(C)} \qquad \text{(cf. \eqref{eq:coYoneda})} \\
& = \compdg(\cat D)(F_{C'}, h_{\Phi_G(C)}) \\
& \cong \compdg(\cat D)(h_{\Phi_F(C')}, h_{\Phi_G(C)}) \qquad \text{($F_{C'} \cong \cat D(-,\Phi_F(C'))$ in $H^0(\compdg(\cat D))$)}\\
& \cong \cat D(\Phi_F(C'), \Phi_G(C)).
\end{align*}
Now, the complex $\cat D(\Phi_F(C'), \Phi_G(C))$ is acyclic by hypothesis, since $\Phi_F(C') \in \cat D_1$ and $\Phi_G(C) \in \cat D_2$. We conclude that the bimodule $\dual(F) \circ G$ is acyclic, hence
\[
\dercomp(\cat C,  \cat D)(F,G) \cong 
\dercomp(\cat C ,  \cat C)(h_{\cat C},\dual(F) \circ G) \cong 0,
\]
as claimed.
\end{proof}
\begin{proposition}
Let $\cat A$ be our given pretriangulated dg-category endowed with a t-structure. There is a distinguished triangle in the derived category $\dercomp(\tau^{\leq 0}\cat A, \tau^{\leq 0}\cat A)$:
\begin{equation} \label{eq:funct_disttria_tstruct}
    i_{\leq 0} \tau_{\leq 0} \xrightarrow{\varepsilon} \operatorname{id} \xrightarrow{\eta} i_{\geq 1} \tau_{\geq 1},
\end{equation}
where $\varepsilon$ and $\eta$ are respectively the counit and unit morphisms of the adjunctions $i_{\leq 0} \dashv \tau_{\leq 0}$ and $\tau_{\geq 1} \dashv i_{\geq 1}$.
\end{proposition}
\begin{proof}
We take the following distinguished triangle in the derived category $\dercomp(\tau^{\leq 0} \cat A, \tau^{\leq 0} \cat A)$:
\[
i_{\leq 0} \tau_{\leq 0} \xrightarrow{\varepsilon} \operatorname{id} \to \cone(\varepsilon)
\]
We first claim that $\cone(\varepsilon)$ is a quasi-functor $\tau^{\leq 0} \cat A \to \tau^{\leq 0} \cat A$. Indeed, the above distinguished triangle induces a distinguished triangle
\[
(i_{\leq 0} \tau_{\leq 0})_A \xrightarrow{\varepsilon} \tau^{\leq 0}\cat A(-,A) \to \cone(\varepsilon)_A
\]
in $\dercomp(\tau^{\leq 0}\cat A)$, for all $A \in \tau^{\leq 0} \cat A$. Now, we consider the following commutative diagram in $\dercomp(\tau^{\leq 0}\cat A)$:
\[
\begin{tikzcd}
(i_{\leq 0} \tau_{\leq 0})_A \arrow[r] & \tau^{\leq 0}\cat A(-,A) \\ \tau^{\leq 0}\cat A(-, i_{\leq 0} \Phi_{\tau_{\leq 0}}(A)) \arrow[r] \arrow[u, "\approx"] & \tau^{\leq 0}\cat A(-,A) \arrow[u, equal]
\end{tikzcd}
\]
By the Yoneda lemma, the morphism $\tau^{\leq 0}\cat A(-, i_{\leq 0} \Phi_{\tau_{\leq 0}}(A)) \to \tau^{\leq 0}\cat A(-,A)$ induces a morphism $i_{\leq 0} \Phi_{\tau_{\leq 0}}(A) \to A$ in $H^0(\cat A) = H^0(\tau^{\leq 0}\cat A)$, and we may take its cone, thus obtaining a distinguished triangle
\[
i_{\leq 0} \Phi_{\tau_{\leq 0}}(A) \to A \to \Phi_{\cone(\varepsilon)}(A)
\]
in $H^0(\cat A)$, and thus a distinguished triangle
\[
\tau^{\leq 0}\cat A(-, i_{\leq 0} \Phi_{\tau_{\leq 0}}(A)) \to \tau^{\leq 0}\cat A(-,A) \to \tau^{\leq 0}\cat A(-,\Phi_{\cone(\varepsilon)}(A))
\]
in $\dercomp(\tau^{\leq 0}\cat A)$. We conclude that there exists an isomorphism
\[
\tau^{\leq 0}\cat A(-,\Phi_{\cone(\varepsilon)}(A)) \xrightarrow{\approx} \cone(\varepsilon)_A,
\]
in $\dercomp(\tau^{\leq 0}\cat A)$, as we wanted.

Next, we observe that by Lemma \ref{lemma:qfunct_orthogonal} we have
\[
\dercomp(\tau^{\leq 0}\cat A, \tau^{\leq 0}\cat A)(i_{\leq 0} \tau_{\leq 0}, i_{\geq 1} \tau_{\geq 1})=0,
\]
since $\tau^{\leq 0}\cat A(A,B)\cong 0$ if $A \in \tau^{\leq 0}\cat A_{\leq 0}$ and $B \in \tau^{\leq 0}\cat A_{\geq 1}$. This implies that
\[
(i_{\leq 0} \tau_{\leq 0} \xrightarrow{\varepsilon} \operatorname{id} \xrightarrow{\eta} i_{\geq 1} \tau_{\geq 1}) = 0,
\]
and we can find a morphism
\[
f \colon \cone(\varepsilon) \to i_{\geq 1} \tau_{\geq 1},
\]
such that the composition $1 \to \cone(\varepsilon) \xrightarrow{f} i_{\geq 1} \tau_{\geq 1}$ is equal to $\eta \colon 1 \to i_{\geq 1} \tau_{\geq 1}$ in $\dercomp(\tau^{\leq 0}\cat A, \tau^{\leq 0}\cat A)$. Hence, for all $A \in \tau^{\leq 0}\cat A$, we have a commutative diagram in $\dercomp(\tau^{\leq 0}\cat A)$:
\[
\begin{tikzcd}
{i_{\leq 0}{\tau_{\leq 0}}_A} \arrow[r, "\varepsilon_A"] \arrow[d, equal] & {h_A} \arrow[r] \arrow[d, equal] & {\cone(\varepsilon)_A} \arrow[d, "f_A"] \\
{i_{\leq 0}{\tau_{\leq 0}}_A} \arrow[r]           & {h_A} \arrow[r, "\eta_A"]           & {i_{\geq 1}{\tau_{\geq 1}}_A.}         
\end{tikzcd}
\]
Thanks to the derived Yoneda lemma, this commutative diagram is uniquely induced by the following commutative diagram in $H^0(\cat A)$:
\[
\begin{tikzcd}
{i_{\leq 0}\Phi_{\tau_{\leq 0}}(A)} \arrow[r, "\tilde{\varepsilon}_A"] \arrow[d, equal] & {A} \arrow[r] \arrow[d, equal] & {\Phi_{\cone(\varepsilon)}(A)} \arrow[d, "\tilde{f}_A"] \\
{i_{\leq 0}\Phi_{\tau_{\leq 0}}(A)} \arrow[r]           & {A} \arrow[r, "\tilde{\eta}_A"]           & {i_{\geq 1}{\Phi_{\tau_{\geq 1}}}(A).}         
\end{tikzcd}
\]
The morphism $\tilde{f}_A$ is necessarily an isomorphism, since the rows are canonical (functorial) distinguished triangles associated to $A$ with respect to the t-structure on $H^0(\cat A)$. We deduce that $f_A$ is also an isomorphism in $\dercomp(\tau^{\leq 0}\cat A)$ for all $A$, hence $f \colon \cone(\varepsilon) \to i_{\geq 1}\tau_{\geq 1}$ is an isomorphism in $\dercomp(\tau^{\leq 0}\cat A, \tau^{\leq 0}\cat A)$, and we conclude.
\end{proof}

\subsection{t-structure on the dg-category of quasi-functors} \label{subsec:tstruct_quasifunctors}
The goal of this subsection is to endow the dg-category of quasi-functors $\RHom(\smallcat a,\cat B)$ with a natural t-structure, assuming that $\cat B$ is a pretriangulated dg-category with a t-structure and that $\smallcat a$ is small and concentrated in nonpositive degrees. This t-structure is in fact defined ``objectwise''
(see Remark \ref{remark:tstruct_quasifunctors_cohomologies}).

\begin{theorem} \label{thm:tstruct_quasifunctors}
Let $\cat B$ be a pretriangulated dg-category endowed with a t-structure $(\cat B_{\leq 0}, \cat B_{\geq 0})$, and let $\smallcat a$ be a small dg-category with cohomology concentrated in nonpositive degrees. Assume that $\smallcat a$ or $\cat B$ is h-flat. Then, the dg-category of quasi-functors $\hproj^\mathrm{rqr}(\smallcat a,\cat B)$ has a t-structure such that
\begin{align*}
    \hproj^\mathrm{rqr}(\smallcat a,\cat B)_{\leq 0} &= \hproj^\mathrm{rqr}(\smallcat a,\cat B_{\leq 0}), \\
    \hproj^\mathrm{rqr}(\smallcat a,\cat B)_{\geq 0} &= \hproj^\mathrm{rqr}(\smallcat a,\cat B_{\geq 0}).
\end{align*}
\end{theorem}
\begin{proof}
Upon taking a suitable resolution, we may assume that $\smallcat a$ is h-projective. Recalling Lemma \ref{lemma:truncations_image_qfun}, we want to check that the triangulated category 
\begin{align*}
H^0(\hproj^\mathrm{rqr}(\smallcat a,\cat B)) &\cong H^0(\hproj^\mathrm{rqr}(\smallcat a,\tau^{\leq 0}\cat B)) \\
&\cong \dercomp^\mathrm{rqr}(\smallcat a, \tau^{\leq 0}\cat B)
\end{align*}
has the t-structure described above, where $\dercomp^\mathrm{rqr}(\smallcat a, \tau^{\leq 0}\cat B)$ is the full category of $\dercomp(\smallcat a, \tau^{\leq 0}\cat B)$ spanned by the right quasi-representable bimodules.

First, $\dercomp^\mathrm{rqr}(\smallcat a, \tau^{\leq 0}\cat B_{\leq 0})$ is closed under nonnegative shifts, because $H^0(\cat B)_{\leq 0}$ has this property. Analogously, $\dercomp^\mathrm{rqr}(\smallcat a, \tau^{\leq 0}\cat B_{\geq 0})$ is closed under nonpositive shifts.

Next, let $F \colon \smallcat a \to \tau^{\leq 0}\cat B$ be a quasi-functor. From the distinguished triangle \eqref{eq:funct_disttria_tstruct} we get a distinguished triangle (in $\dercomp(\smallcat a, \tau^{\leq 0}\cat B)$) of quasi-functors $\smallcat a \to \tau^{\leq 0} \cat B$, by precomposing with $F$:
\begin{equation}
    i_{\leq 0} \tau_{\leq 0} F \to F \to i_{\geq 1} \tau_{\geq 1} F,
\end{equation}
and it is clear that the composition $i_{\leq 0} \tau_{\leq 0} F$ has quasi-essential image in $\tau^{\leq 0} \cat B_{\leq 0}$, and analogously $i_{\geq 1} \tau_{\geq 1} F$ has quasi-essential image in $\tau^{\leq 0} \cat B_{\geq 1}$. To conclude, we only need to check that given quasi-functors $F, G \colon \smallcat a \to \tau^{\leq 0} \cat B$ such that $F$ has quasi-essential image in $\tau^{\leq 0} \cat B_{\geq 0}$ and $G$ has quasi-essential image in $\tau^{\leq 0} \cat B_{\geq 1}$, then $\dercomp^\mathrm{rqr}(\smallcat a, \tau^{\leq 0}\cat B)(F,G) = 0$. This follows from Lemma \ref{lemma:qfunct_orthogonal}, because $\tau^{\leq 0}\cat B(A,B)\cong 0$ if $A \in \tau^{\leq 0}\cat B_{\leq 0}$ and $B \in \tau^{\leq 0}\cat B_{\geq 1}$.
\end{proof}

The t-structure on the dg-category of quasi-functors is well-behaved with respect to pre- and post-composition, namely:
\begin{proposition} \label{prop:tstruct_precomposition}
Let $F \colon \smallcat a' \to \smallcat a$ be a quasi-functor between small dg-categories with cohomology concentrated in nonpositive degrees, and let $\cat B$ be a pretriangulated dg-category with a t-structure. Assume that either $\smallcat a$ or $\cat B$ is h-flat. Then, the induced precomposition quasi-functor
\[
F^* \colon \hproj^\mathrm{rqr}(\smallcat a, \cat B) \to \hproj^\mathrm{rqr}(\smallcat a', \cat B)
\]
is t-exact.
\end{proposition}
\begin{proof}
Straightforward.
\end{proof}
\begin{proposition} \label{prop:tstruct_postcomposition}
Let $F \colon \cat B \to \cat B'$ a left t-exact (respectively right t-exact, t-exact) quasi-functor between dg-categories endowed with a t-structure, and let $\smallcat a$ be a small dg-category with cohomology concentrated in nonpositive degrees. Assume that either $\smallcat a$ or $\cat B$ is h-flat. Then, the induced postcomposition quasi-functor
\[
F_* \colon \hproj^\mathrm{rqr}(\smallcat a, \cat B) \to \hproj^\mathrm{rqr}(\smallcat a, \cat B')
\]
is left t-exact (respectively right t-exact, t-exact).
\end{proposition}
\begin{proof}
Straightforward.
\end{proof}

Let us now discuss the heart of the t-structure induced on the category of quasi-functors.
\begin{proposition} \label{prop:tstruct_quasifunctors_heart}
Let $\smallcat a$ be a small dg-category with cohomology concentrated in nonpositive degrees, and let $\cat B$ be a pretriangulated dg-category with a t-structure $(\cat B_{\leq 0}, \cat B_{\geq 0})$. Assume that $\smallcat a$ or $\cat B$ is h-flat. Then, there is an equivalence of abelian categories:
\begin{equation}
    H^0(\hproj^{\mathrm{rqr}}(\smallcat a, \cat B))^\heartsuit \cong \Fun(H^0(\smallcat a), H^0(\cat B)^\heartsuit ).
\end{equation}
\end{proposition}
\begin{proof}
Without loss of generality, we assume that $\smallcat a$ is h-projective. By definition, the heart $H^0(\hproj^{\mathrm{rqr}}(\smallcat a, \cat B))^\heartsuit$ is the subcategory of $H^0(\hproj^{\mathrm{rqr}}(\smallcat a, \cat B)$ of quasi-functors whose quasi-essential image is contained in $\cat B_{\leq 0} \cap \cat B_{\geq 0}$. As a dg-category, $\cat B_{\leq 0} \cap \cat B_{\geq 0}$ has cohomology concentrated in nonnegative degrees (see Remark \ref{remark:tstruct_dgcat_aisles_heart}), hence
\[
\tau^{\leq 0}(\cat B_{\leq 0} \cap \cat B_{\geq 0}) \cong H^0(\cat B)^\heartsuit.
\]
By Lemma \ref{lemma:truncations_image_qfun}, we therefore deduce that
\[
H^0(\hproj^{\mathrm{rqr}}(\smallcat a, \cat B))^\heartsuit \cong H^0(\hproj^{\mathrm{rqr}}(\smallcat a, H^0(\cat B)^\heartsuit)),
\]
and from Lemma \ref{prop:H0_equiv_iftarget_H0} we immediately conclude.
\end{proof}
\begin{remark} \label{remark:tstruct_quasifunctors_cohomologies}
From the proof of Theorem \ref{thm:tstruct_quasifunctors} it is clear that the truncation functors of the t-structure on $\hproj^\mathrm{rqr}(\smallcat a, \cat B)$ are given by composition with the truncations $\tau_{\leq 0}$ and $\tau_{\geq 0}$ of the t-structure of $\cat B$. Next, if we identify the heart of this t-structure with $\Fun(H^0(\smallcat a),H^0(\cat B)^\heartsuit)$ according to the above Proposition \ref{prop:tstruct_quasifunctors_heart}, it is clear that t-cohomologies with respect to $\hproj^\mathrm{rqr}(\smallcat a, \cat B)$ are computed objectwise:
\begin{equation}
    H^i_t(F)(A) = H^i_t(\Phi_F(A)),
\end{equation}
for $F \in  \hproj^\mathrm{rqr}(\smallcat a,\cat B)$, where here $\Phi_F$ is the induced functor $H^0(\smallcat a) \to H^0(\cat B)$ (cf. Remark \ref{remark:qfun_notations}) and where $H^i_t$ denotes the 
cohomological functor associated to the t-structure, respectively on $\hproj^\mathrm{rqr}(\smallcat a, \cat B)$ and on $\cat B$.
\end{remark}

% t-structures and duality
Finally, we discuss the behaviour of the t-structures on functor dg-categories with respect to the duality \eqref{eq:quasifunctors_duality}.
\begin{proposition} \label{prop:duality_quasifunctor_texact}
Let $\smallcat a$ be a small dg-category with cohomology concentrated in nonpositive degrees, and let $\cat B$ be a pretriangulated dg-category with a t-structure. Assume that either $\smallcat a$ or $\cat B$ is h-flat. Then, the quasi-equivalence \eqref{eq:quasifunctors_duality}
\[
\dual_{\smallcat a, \cat B} \colon \hproj^\mathrm{rqr}(\smallcat a, \cat B)  \xrightarrow{\sim} \opp{\hproj^\mathrm{rqr}(\opp{\smallcat a}, \opp{\cat B})}
\]
is t-exact.
\end{proposition}
\begin{proof}
Let $F \in \hproj^\mathrm{rqr}(\smallcat a, \cat B)_{\leq 0}$. By definition, $H^*(F)$ factors through $H^*(\cat B_{\leq 0})$. By Remark \ref{remark:dual_quasifun_cohomology}, we know that the dual $\dual(F)$ is such that $H^*(\dual F)\cong \opp{H^*(F)}$, so by hypothesis it factors through $\opp{H^*(\cat B_{\leq 0})}$. This is precisely $H^*((\opp{\cat B})_{\geq 0})$, considering the opposite t-structure. This means that $\dual(F) \in \hproj^\mathrm{rqr}(\opp{\smallcat a}, \opp{\cat B})_{\geq 0}$, hence (again considering the opposite t-structure) that
\[
\dual(F) \in (\opp{\hproj^\mathrm{rqr}(\opp{\smallcat a}, \opp{\cat B})})_{\leq 0},
\]
as claimed.
\end{proof}

% Change of rings
\section{Change of rings} \label{section:changeofrings}
In this part, we fix a commutative dg-ring $R$ and a commutative $R$-dg-algebra $R \to S$. In contrast to the previous sections, we don't need to assume that $R$ or $S$ are concentrated in nonpositive degrees for the specific purpose of discussing base change along $R \to S$. We denote respectively by $\kat{dgCat}(R)$ and $\kat{dgCat}(S)$ the categories of small $R$-linear and $S$-linear dg-categories. 
\subsection{Restriction of scalars}
Restriction of scalars gives a functor
\begin{equation}
    \begin{split}
        \kat{dgCat}(S) &\to \kat{dgCat}(R), \\
        \smallcat a & \mapsto \smallcat a_R.
    \end{split}
\end{equation}
This functor is moreover lax monoidal, indeed there is a natural $R$-linear functor
\begin{equation*}
    \smallcat a_R \otimes_R \smallcat b_R \to (\smallcat a \otimes_S \smallcat b)_R.
\end{equation*}
It is also immediate to see that it preserves quasi-equivalences, hence inducing a restriction of scalars functor
\begin{equation} \label{eq:restriction_hqe}
\begin{split}
    \Hqe(S) & \to \Hqe(R), \\
    \smallcat a & \mapsto \smallcat a_R,
\end{split}
\end{equation}
where $\Hqe(R)$ and $\Hqe(S)$ are respectively the homotopy categories of $R$-linear and $S$-linear dg-categories.

Notice moreover that we have a natural $R$-linear dg-functor
\begin{equation}
    \compdg(S)_R \to \compdg(R)
\end{equation}
induced by restriction of scalars. From this, we see that any $S$-linear dg-bimodule $F \in \compdg_S(\smallcat a,\smallcat b)$ naturally induces a $R$-linear dg-bimodule $F_R \in \compdg_R(\smallcat a_R, \smallcat b_R)$. This actually gives an $R$-linear dg-functor:
\begin{equation}
\begin{split}
    \compdg_S(\smallcat a,\smallcat b)_R & \to \compdg_R(\smallcat a_R, \smallcat b_R), \\
    F & \mapsto F_R.
\end{split}
\end{equation}

The above discussion can be held verbatim also for locally small dg-categories, avoiding the reference to the categories $\kat{dgCat}(R)$ or $\kat{dgCat}(S)$.

If $\cat A$ is a $S$-linear dg-category, we will sometimes abuse notation and write again $\cat A$ instead of $\cat A_R$ for its underlying $R$-linear dg-category.
\subsection{Coextension of scalars} \label{subsec:coextension}
The commutative dg-ring $S$ can obviously be viewed as a dg-category with a single object $\diamondsuit$ and $\Hom(\diamondsuit,\diamondsuit)=S$, by assumption concentrated in nonpositive degrees. Our goal is to find a right adjoint to the restriction of scalars $\Hqe(S) \to \Hqe(R)$. 
In practice, we will work in the (slightly) more general setting of locally $\mathbb U$-small dg-categories, which are still small with respect to the bigger universe $\mathbb V$. This will allow us to take the suitable (cofibrant, h-projective, h-flat) resolutions of the given dg-categories: they will still be $\mathbb U$-locally small up to quasi-equivalence.

\subsubsection{The $S$-linear structure on $\compdg_R(S,\cat B)$} 
In analogy with the classical result, the coextension of scalars of an $R$-linear dg-category $\cat B$ along $R \to S$ will be described by an $S$-linear category of $R$-linear quasi-functors $S \to \cat B$. Such quasi-functors are in particular dg-bimodules, namely objects of $\compdg_R(S,\cat B)$.
\begin{lemma} \label{lemma:bimodules_tensor_S_Slinear}
Let $\cat B$ be an $R$-linear dg-category. The dg-category $\compdg_R(S,\cat B)$ has a natural $S$-linear structure defined as follows. Given a morphism $\varphi \colon X \to Y$ in $\compdg_R(S,\cat B)$ and $s \in S$, we set
\begin{equation}
    (s\varphi)^B_{\diamondsuit} = Y_s^{1_B} \varphi^B_{\diamondsuit} = (-1)^{|\varphi||s|} \varphi^B_{\diamondsuit} X_s^{1_B}.
\end{equation}
Since $S$ is commutative, $s\varphi \colon X \to Y$ is a well defined natural transformation. In particular, we have
\begin{equation}
    (s 1_X)^B_{\diamondsuit} = X^{1_B}_s.
\end{equation}
\end{lemma}
\begin{proof}
Straightforward.
\end{proof}
\begin{remark} \label{remark:coextension_dgmS}
Let $F \in \compdg_R(S,\cat B)$. Then, the complex $F_{\diamondsuit}^B$ has a natural structure of $S$-dg-module, for all $B \in \cat B$. Indeed, set:
\begin{equation}
\begin{split}
    S \otimes_R F_{\diamondsuit}^B & \to F_{\diamondsuit}^B, \\
    s \otimes x & \mapsto F_s^{1_B}(x).
    \end{split}
\end{equation}
In particular, for any $f \colon B \to B'$ in $\cat B$ and $s \in S$, the map
\[
F_s^f \colon F_{\diamondsuit}^{B'} \to F_{\diamondsuit}^B
\]
is $S$-linear.
\end{remark}
\begin{proposition}
The $S$-linear dg-category $\compdg_R(S,\cat B)$ is tensored over $\compdg(S)$. Namely, for any $V \in \compdg(S)$ and $F \in \compdg_R(S,\cat B)$, there is an object
\begin{equation*}
    V \otimes_S F \in \compdg_R(S,\cat B)
\end{equation*}
together with a natural $S$-linear isomorphism
\begin{equation} \label{eq:coextension_tensor}
    \compdg_R(S,\cat B)(V \otimes_S F, G) \cong \compdg(S)(V, \compdg_R(S,\cat B)(F,G)).
\end{equation}
The tensor $V \otimes_S F$ is thus dg-functorial in both variables.

The $S$-linear dg-category $\compdg_R(S,\cat B)$ is also cotensored over $\compdg(S)$. Namely, for any $V \in \compdg(S)$ and $G \in \compdg_R(S,\cat B)$, there is an object
\begin{equation*}
    \Hom_S(V, G) \in \compdg_R(S,\cat B)
\end{equation*}
together with a natural $S$-linear isomorphism
\begin{equation} \label{eq:coextension_cotensor}
    \compdg_R(S,\cat B)(F, \Hom_S(V,G)) \cong \compdg(S)(V, \compdg_R(S,\cat B)(F,G)).
\end{equation}
The cotensor $\Hom_S(V, G)$ is thus dg-functorial in both variables.
\end{proposition}
\begin{proof}
By Remark \ref{remark:coextension_dgmS}, we know that for any object $F \in \compdg_R(S,\cat B)$, the complex $F_{\diamondsuit}^B$ has a natural $S$-linear structure. Hence, we may set:
\begin{align}
    (V \otimes_S F)_{\diamondsuit}^B &= V \otimes_S F_{\diamondsuit}^B, \\
    \Hom_S(V,F)_{\diamondsuit}^B &= \compdg(S)(V, F_{\diamondsuit}^B).
\end{align}
Now, the verification is straightforward.
\end{proof}
\begin{corollary} \label{coroll:cotensor_acyclic}
Assume $V \in \compdg(S)$ is h-projective and $F \in \compdg_R(S,\cat B)$ is acyclic. Then, $\Hom_S(V,F)$ is acyclic.
\end{corollary}
\begin{proof}
This follows directly from the characterization $\Hom_S(V,F)_{\diamondsuit}^B = \compdg(S)(V, F_{\diamondsuit}^B)$.
\end{proof}

\subsubsection{Statement of the main results}

Next, we assume that $\cat A$ is an $S$-linear dg-category and $\cat B$ is an $R$-linear dg-category, with $\cat A$ being h-projective and $\cat B$ being h-flat (respectively as an $S$-linear and an $R$-linear dg-category). Clearly, $\hproj_R(S,\cat B)$ and $\hproj_R^\mathrm{rqr}(S,\cat B)$ inherit the $S$-linear structure from $\compdg_R(S,\cat B)$, and will be viewed as $S$-linear dg-categories. If we take the restriction $\hproj_R(S,\cat B)_R$ we clearly fall back to the $R$-linear dg-category of h-projective $S$-$\cat B$-dg-bimodules. The main result we want to achieve is the following:
\begin{theorem} \label{thm:coext_univproperty}
Let $\cat A$ be an $S$-linear dg-category and $\cat B$ an $R$-linear dg-category, with $\cat A$ having h-projective complexes of morphisms and $\cat B$ having h-flat complexes of morphisms (respectively of $S$-modules and of $R$-modules). There is a natural quasi-equivalence of $R$-linear dg-categories
\begin{equation} \label{eq:coext_univproperty}
     \hproj_R^{\mathrm{rqr}}(\cat A_R, \cat B) \xrightarrow{\sim} \hproj_S^{\mathrm{rqr}}(\cat A, \hproj_R^\mathrm{rqr}(S,\cat B))_R .
\end{equation}
\end{theorem}
\begin{corollary} \label{coroll:coext_hqe_univproperty}
Let $\smallcat a \in \Hqe(S)$ and let $\smallcat b \in \Hqe(R)$. Let $Q(\smallcat b) \to \smallcat b$ be a (functorial) cofibrant replacement of $\smallcat b$, so that $Q(\smallcat b)$ is h-flat. There is a natural bijection
\begin{equation} \label{eq:coext_hqe_univproperty}
    \Hqe(R)(\smallcat a_R, \smallcat b) \cong \Hqe(S)(\smallcat a, \hproj_R^\mathrm{rqr}(S,Q(\smallcat b))),
\end{equation}
so that $\smallcat b \mapsto \hproj_R^\mathrm{rqr}(S,Q(\smallcat b))$ defines a right adjoint of the restriction of scalars $\Hqe(S) \to \Hqe(R)$.
\end{corollary}
\begin{remark} \label{remark:notation_coext_rings}
The $S$-linear dg-category $\hproj_R^\mathrm{rqr}(S,Q(\cat B)))$ is uniquely determined up to $S$-linear quasi-equivalence, and will be denoted by $\cat B_{(S)}$. The dg-category $\cat B_{(S)}$, or more precisely the mapping
\begin{equation}
\cat B \mapsto \cat B_{(S)}
\end{equation}
will be sometimes called \emph{coextension of scalars}.
\end{remark}

The proof of \eqref{eq:coext_univproperty} will be based on the analogous result in the ``strict dg-functorial setting'', namely, the following proposition.
\begin{proposition} \label{prop:coextension_scalars_dg_strict}
There is a natural adjoint equivalence of $R$-linear dg-categories:
\begin{equation} \label{eq:coextension_scalars_dg_strict}
    l \colon \compdg_R(\cat A_R, \cat B) \leftrightarrows \Fun_S(\cat A,\compdg_R(S,\cat B))_R : r,
\end{equation}
where
\begin{equation}
    \begin{split}
        l(G)(A)_{\diamondsuit}^B &= G_A^B, \\
        l(G)_s^{f} &= G^{f}_{s 1_A}, \quad s \in S.
    \end{split}
\end{equation}
and
\begin{equation}
    r(F)_A^B  = F(A)^B_{\diamondsuit} \in \compdg(R).
\end{equation}
\end{proposition}

\subsubsection{A closer look at $\Fun_S(\cat A,\compdg_R(S,\cat B))$}

We now start a deeper investigation of the $S$-linear dg-category $\Fun_S(\cat A,\compdg_R(S,\cat B))$. Being a dg-category of functors with values in $\compdg_R(S,\cat B)$, it has natural notions of acyclic objects, h-projective objects and quasi-isomorphisms.
\begin{definition} \label{def:Funhp_acyclics_qis}
A (closed, degree $0$) natural transformation $F \to G$ between functors $F, G \in \Fun_S(\cat A, \compdg_R(S,\cat B))$ is a \emph{quasi-isomorphism} if the components $F(A) \to G(A)$ are quasi-iso\-morphisms in $\compdg_R(S,\cat B)$ for all $A \in \cat A$. An element $F \in \Fun_S(\cat A, \compdg_R(S,\cat B))$ is \emph{acyclic} if $F(A)$ is acyclic in $\compdg_R(S,\cat B)$ for all objects $A \in \cat A$. An element $P \in \Fun_S(\cat A, \compdg_R(S,\cat B))$ is \emph{h-projective} if $\Hom(P,F)$ is acyclic for any acyclic $F$. We shall denote by $\Fun_S^\mathrm{hp}(\cat A, \compdg_R(S,\cat B))$ the full dg-subcategory of $\Fun_S(\cat A, \compdg_R(S,\cat B))$ spanned by h-projective objects.
\end{definition}
\begin{remark} \label{remark:Funhp_qis_invert}
The $S$-linear dg-category $\Fun_S(\cat A, \compdg_R(S,\cat B))$ is (strongly) pretriangulated, since $\compdg_R(S,\cat B)$ is. It is immediate to see that $\varphi \colon F \to G$ is a quasi-isomorphism if and only if its cone $\cone(\varphi)$ is acyclic in the sense of the above Definition \ref{def:Funhp_acyclics_qis}. Moreover, it also straightforward to check that $\cone(\varphi)$ is h-projective if both $F$ and $G$ are, again in the sense of Definition \ref{def:Funhp_acyclics_qis}. Hence, if $\varphi$ is a quasi-isomorphism between h-projective objects in $\Fun_S(\cat A, \compdg_R(S,\cat B))$, the cone $\cone(\varphi)$ is both acyclic and h-projective. This immediately implies that 
\[
\cone(\varphi) \cong 0 \quad  \text{in $H^0(\Fun_S(\cat A, \compdg_R(S,\cat B)))$},
\]
which in turn implies that $\varphi \colon F \to G$ is an isomorphism in $H^0(\Fun_S(\cat A, \compdg_R(S,\cat B)))$. This discussion can be summarized in the following slogan: \emph{quasi-isomorphisms in $\Fun_S^\mathrm{hp}(\cat A, \compdg_R(S,\cat B))$ are invertible in $H^0$.}
\end{remark}
Since by assumption the $S$-linear dg-category $\cat A$ is h-projective, it is natural to expect that objects in $\Fun_S^\mathrm{hp}(\cat A, \compdg_R(S,\cat B))$ are termwise h-projective.
\begin{lemma} \label{lemma:F(A)_hproj}
Recall that we are assuming $\cat A$ to be h-projective. Let $F \in \Fun_S(\cat A, \compdg_R(S,\cat B))$ be h-projective. Then, $F(A)$ is h-projective for all $A \in \cat A$. In particular, we may identify: 
\[
\Fun^{\mathrm{hp}}_S(\cat A, \hproj_R(S,\cat B)) = \Fun^{\mathrm{hp}}_S(\cat A, \compdg_R(S,\cat B)).
\]
\end{lemma}
\begin{proof}
Let $A \in \cat A$, and let $G \in \compdg_R(S,\cat B)$ be acyclic. We compute:
\begin{align*}
\compdg_R(S,\cat B)(F(A),G) & \cong \compdg(\cat A)(\cat A(-,A), \compdg_R(S,\cat B)(F(-),G) \qquad \text{(Yoneda)} \\
& \cong \Fun_S(\cat A, \compdg_R(S,\cat B))(F, \Hom_S(\cat A(-,A),G)). \qquad \text{(cf. \eqref{eq:coextension_cotensor})}
\end{align*}
Now, $\Hom_S(\cat A(-,A),G)$ is acyclic by Corollary \ref{coroll:cotensor_acyclic}, because $G$ is acyclic and $\cat A(A',A)$ is h-projective for all $A' \in \cat A$, and we conclude using that $F$ is h-projective.
\end{proof}

\subsubsection{$\hproj_S^\mathrm{rqr}(\cat A, \hproj_R^\mathrm{rqr}(S,\cat B))$ as a subcategory of $\Fun_S(\cat A,\compdg_R(S,\cat B))$}

We observe that we have full $R$-linear dg-subcategories:
\[
\hproj_R^{\mathrm{rqr}}(\cat A_R, \cat B) \subseteq \hproj_R(\cat A_R, \cat B) \subseteq \compdg_R(\cat A_R, \cat B),
\]
and on the other hand we would like to identify $\hproj_S^\mathrm{rqr}(\cat A, \hproj_R^\mathrm{rqr}(S,\cat B))$ with a full dg-subcategory of $\Fun_S(\cat A,\compdg_R(S,\cat B))$, at least up to ($S$-linear) quasi-equivalence. This is first achieved in the following lemma, where we actually prove something more than that.

\begin{lemma} \label{lemma:rqr-to-functors}
There are $S$-linear quasi-fully faithful dg-functors
\begin{align*}
    L \colon \hproj_S^\mathrm{rqr}(\cat A, \hproj_R^\mathrm{rqr}(S,\cat B)) & \hookrightarrow \Fun_S(\cat A,\compdg_R(S,\cat B)), \\
    L' \colon \hproj_S^\mathrm{rqr}(\cat A, \hproj_R(S,\cat B)) & \hookrightarrow \Fun_S(\cat A,\compdg_R(S,\cat B)),
\end{align*}
defined as follows:
\begin{equation} \label{eq:ladj_restriction_ringchange}
\begin{split}
        F \mapsto L(F)(A) &= \int^{X}  F_A^X \otimes_S X, \quad \textrm{$(X \in \hproj^\mathrm{rqr}_R(S,\cat B))$,} \\
        F \mapsto L'(F)(A) &= \int^{X}  F_A^X \otimes_S X, \quad \textrm{$(X \in \hproj_R(S,\cat B))$.}
\end{split}    
\end{equation}
where $F_A^X \otimes_S X$ is the tensor in $\compdg_R(S,\cat B)$, see \eqref{eq:coextension_tensor}.
\end{lemma}
\begin{remark} \label{remark:rqr-to-functors}
It is worth mentioning where the definition of the above dg-functors $L$ and $L'$ in \eqref{eq:ladj_restriction_ringchange} come from. First, we have a dg-functor
\begin{equation*}
\begin{split}
T \colon \Fun_S(\cat A, & \compdg_R(S,\cat B)) \to \compdg_S(\cat A, \hproj^\mathrm{rqr}_R(S,\cat B)), \\
G &\mapsto T(G)_A^X = \compdg_R(S,\cat B)(X,G(A))
\end{split}
\end{equation*}
We look for a left adjoint to $T$, by computing as follows:
\begin{align}
    \compdg_S(\cat A, \hproj^\mathrm{rqr}_R(S,\cat B))(F, T(G)) &\cong \int_{A,X} \compdg(S)(F_A^X, \compdg_R(S,\cat B)(X,G(A))) \nonumber \\
    &\cong \int_{A,X} \compdg_R(S,\cat B)(F_A^X \otimes_S X, G(A)) \quad \text{(by \eqref{eq:coextension_tensor})} \nonumber \\
    & \cong \int_A \compdg_R(S,\cat B)(\int^X F_A^X \otimes_S X, G(A)) \nonumber \\
    & \cong \Fun_S(\cat A, \compdg_R(S,\cat B))(\int^X F^X \otimes_S X, G) \nonumber \\
    & = \Fun_S(\cat A, \compdg_R(S,\cat B))(L(F), G). \label{eq:rqr-to-functors_adj_1}
\end{align}
We may also define
\[
T' \colon \Fun_S(\cat A, \compdg_R(S,\cat B)) \to \compdg_S(\cat A, \hproj_R(S,\cat B))
\]
with a similar formula:
\[
G \mapsto T(G)_A^X = \compdg_R(S,\cat B)(X,G(A)).
\]
A similar computation as above shows that
\begin{equation}
 \compdg_S(\cat A, \hproj_R(S,\cat B))(F, T'(G)) \cong \Fun_S(\cat A, \compdg_R(S,\cat B))(L'(F), G). \label{eq:rqr-to-functors_adj_2}
\end{equation}
Clearly, isomorphisms \eqref{eq:rqr-to-functors_adj_1} and \eqref{eq:rqr-to-functors_adj_2} yield natural (closed, degree $0$) morphisms
\begin{align}
    \eta_F \colon F \to TL(F), \label{eq:rqr-to-functors_unit_1} \\
    \eta'_{F'} \colon F' \to T'L'(F'), \label{eq:rqr-to-functors_unit_2}
\end{align}
respectively in $\compdg_S(\cat A, \hproj^\mathrm{rqr}_R(S,\cat B))$ and $\compdg_S(\cat A, \hproj_R(S,\cat B))$, and natural (closed, degree $0$) morphisms
\begin{align}
    \varepsilon_G \colon LT(G) \to G, \label{eq:rqr-to-functors_counit_1} \\
    \varepsilon'_{G'} \colon L'T'(G') \to G', \label{eq:rqr-to-functors_counit_2}
\end{align}
in $\Fun_S(\cat A, \compdg_R(S,\cat B))$.

We also remark that the coend which defines $L$ (and respectively $L'$) exists once we work in the ``enlarged universe'' $\mathbb V$. See \S \ref{subsec:sizes} and Remark \ref{remark:coend_size}.
\end{remark}
\begin{remark} \label{remark:leftadj_coend_representable}
Let $F \in \hproj_S^\mathrm{rqr}(\cat A, \hproj_R^\mathrm{rqr}(S,\cat B))$. By assumption, $\cat A$ is h-projective, hence \cite[Lemma 3.4]{canonaco-stellari-internalhoms} for a fixed $A \in \cat A$ the dg-module $F_A$ is also h-projective and we have an isomorphism in $H^0(\compdg_S(\hproj_R^\mathrm{rqr}(S,\cat B)))$:
\[
F_A \cong \Hom(-,\Phi_F(A)),
\]
for some $\Phi_F(A) \in \hproj_R^\mathrm{rqr}(S,\cat B) \subseteq \compdg_R(S,\cat B)$. We also have isomorphisms in $H^0(\compdg_R(S,\cat B))$:
\begin{align*}
    \Hom(\int^X F_A^X \otimes_S X, -) &\cong \int_X \Hom(F_A^X \otimes_S X, -) \\
    &\cong \int_X \compdg(S)(F_A^X, \Hom(X,-)) \qquad \text{(recall \eqref{eq:coextension_tensor})}  \\
    &\cong \int_X \compdg(S)(\Hom(X,\Phi_F(A)), \Hom(X,-)) \\
    &\cong \Hom(\Phi_F(A), -). \qquad \text{(Yoneda)}
\end{align*}
We conclude that, for a fixed $A \in \cat A$, we have an isomorphism in $H^0(\compdg_R(S,\cat B))$:
\begin{equation}
L(F)(A) \cong \int^X F_A^X \otimes_S X \cong \Phi_F(A), \label{eq:leftadj_coend_representable_1}
\end{equation}

With an analogous (actually, nearly identical) argument, we may also show the following claim. Let $F \in \hproj_S^\mathrm{rqr}(\cat A, \hproj_R(S,\cat B))$, so that we have an isomorphism in $H^0((\compdg_S(\hproj_R(S,\cat B)))$:
\[
F_A \cong \Hom(-,\Phi_F(A))
\]
for some $\Phi_F(A) \in \hproj_R(S,\cat B) \subseteq \compdg_R(S,\cat B)$. Then, we conclude that we have an isomorphism in $H^0(\compdg_R(S,\cat B))$:
\begin{equation}
L'(F)(A) \cong \int^X F_A^X \otimes_S X \cong \Phi_F(A), \label{eq:leftadj_coend_representable_2}
\end{equation}
for a fixed $A \in \cat A$.
\end{remark}
\begin{proof}[Proof of Lemma \ref{lemma:rqr-to-functors}]
We show that $L$ is quasi-fully faithful; the same claim for $L'$ is proven in a completely analogous way.  To do so, we check that the composition
\begin{equation*}
    \Hom(F_1, F_2) \xrightarrow{{\eta_{F_2}}_*} \Hom(F_1, TL(F_2)) \cong \Hom(L(F_1),L(F_2))
\end{equation*}
is a quasi-isomorphism, using the notations of the above Remark \ref{remark:rqr-to-functors}. Since we are assuming $F_1$ to be h-projective, to check that $\Hom(F_1, F_2) \xrightarrow{{\eta_{F_2}}_*} \Hom(F_1, TL(F_2))$ is a quasi-isomorphism it is enough to see that the unit map \eqref{eq:rqr-to-functors_unit_1}
\begin{equation*}
    \eta_{F} \colon F \to TL(F)
\end{equation*}
is a quasi-isomorphism for all $F \in \hproj_S^\mathrm{rqr}(\cat A, \hproj_R^\mathrm{rqr}(S,\cat B))$. If $A \in \cat A$, we get from the above Remark \ref{remark:leftadj_coend_representable} the following isomorphisms, respectively in $H^0((\compdg_S(\hproj_R(S,\cat B)))$ and $H^0(\compdg_R(S,\cat B))$:
\begin{align*}
F_A &\cong \Hom(-,\Phi_F(A)), \\
L(F)(A) &\cong \int^X F_A^X \otimes_S X \cong \Phi_F(A).
\end{align*}
Moreover, we have a commutative diagram:
\begin{equation*}
    \begin{tikzcd}
F_A^X \arrow[r, "\eta_F"] \arrow[rdd, "\approx"'] & TL(F)_A^X \arrow[d, equal]                    \\
                                        & {\Hom(X,L(F)(A))} \arrow[d, "\approx"] \\
                                        & {\Hom(X,\Phi_F(A))},                  
\end{tikzcd}
\end{equation*}
We conclude that $\eta_F$ is a quasi-isomorphism, as claimed. 
\end{proof}
In order to prove Theorem \ref{thm:coext_univproperty}, we need to investigate the behaviour of the involved dg-functors with respect to h-projective objects.
\begin{lemma} \label{lemma:changeofrings_hproj_preservation}
Consider the $S$-linear dg-functors
\begin{align*}
    L \colon \hproj_S^\mathrm{rqr}(\cat A, \hproj_R^\mathrm{rqr}(S,\cat B)) & \hookrightarrow \Fun_S(\cat A,\compdg_R(S,\cat B)), \\
    L' \colon \hproj_S^\mathrm{rqr}(\cat A, \hproj_R(S,\cat B)) & \hookrightarrow \Fun_S(\cat A,\compdg_R(S,\cat B)),
\end{align*}
of Lemma \ref{lemma:rqr-to-functors}. They both have essential image in $\Fun_S^\mathrm{hp}(\cat A, \compdg_R(S,\cat B))$.
\end{lemma}
\begin{proof}
We recall the isomorphisms \eqref{eq:rqr-to-functors_adj_1} and \eqref{eq:rqr-to-functors_adj_2}:
\begin{align*}
   \compdg_S(\cat A, \hproj_R^\mathrm{rqr}(S,\cat B))(F,T(G))  & \cong \Fun_S(\cat A,\compdg_R(S,\cat B))(L(F),G), \\
    \compdg_S(\cat A, \hproj_R(S,\cat B))(F', T'(G')) & \cong \Fun_S(\cat A, \compdg_R(S,\cat B))(L'(F'), G').
\end{align*}
If $F \in \hproj_S^\mathrm{rqr}(\cat A, \hproj_R^\mathrm{rqr}(S,\cat B))$ (resp. $F \in \hproj_S^\mathrm{rqr}(\cat A, \hproj_R^\mathrm{rqr}(S,\cat B))$), then $L(F)$ (resp. $L'(F)$) is h-projective in $\Fun_S(\cat A, \compdg_R(S,\cat B))$ (recall Definition \ref{def:Funhp_acyclics_qis}). This holds because of the above isomorphisms, and the fact that $T(Y)$ (resp. $T'(Y)$) is an acyclic object in $\compdg_S(\cat A, \hproj_R(S,\cat B))$ (resp. in $\compdg_S(\cat A, \hproj_R^\mathrm{rqr}(S,\cat B))$), if $Y \in \Fun_S(\cat A, \compdg_R(S,\cat B))$ is acyclic in the sense of Definition \ref{def:Funhp_acyclics_qis}. Indeed:
\[
T(Y)_A^X = \compdg_R(S,\cat B)(X,Y(A)),
\]
is acyclic since $X$ is h-projective and $Y(A)$ is acyclic, and a similar argument proves the claim for $T'(Y)$.
\end{proof}
We can now characterize the actual quasi-essential images of $L$ and $L'$.
\begin{lemma} \label{lemma:rqr-to-functors_quasiequivalences}
The $S$-linear dg-functors $L$ and $L'$ of Lemma \ref{lemma:rqr-to-functors} induce quasi-equivalences
\begin{align}
    L \colon \hproj_S^\mathrm{rqr}(\cat A, \hproj_R^\mathrm{rqr}(S,\cat B)) & \xrightarrow{\sim} \Fun_S^\mathrm{hp}(\cat A,\hproj_R^\mathrm{rqr}(S,\cat B)), \\
    L' \colon \hproj_S^\mathrm{rqr}(\cat A, \hproj_R(S,\cat B)) & \xrightarrow{\sim} \Fun_S^\mathrm{hp}(\cat A,\hproj_R(S,\cat B)).
\end{align}
\end{lemma}
\begin{proof}
We already know from Lemma \ref{lemma:rqr-to-functors} that $L$ and $L'$ are quasi-fully faithful. Moreover, since by assumption $\cat A$ is h-projective, we may identify
\[
\Fun^{\mathrm{hp}}_S(\cat A, \hproj_R(S,\cat B)) = \Fun^{\mathrm{hp}}_S(\cat A, \compdg_R(S,\cat B)),
\]
see Lemma \ref{lemma:F(A)_hproj}. Hence, the quasi-essential image of $L'$ lies in $\Fun^{\mathrm{hp}}_S(\cat A, \hproj_R(S,\cat B))$. As for $L$, we recall the isomorphism \eqref{eq:leftadj_coend_representable_1} in in $H^0(\compdg_R(S,\cat B))$ from Remark \ref{remark:leftadj_coend_representable}:
\[
L(F)(A) \cong \Phi_F(A),
\]
which holds for $F \in \hproj_S^\mathrm{rqr}(\cat A, \hproj_R^\mathrm{rqr}(S,\cat B))$ and for a fixed $A \in \cat A$. The object $\Phi_F(A)$ lies in $\hproj_R^\mathrm{rqr}(S,\cat B)$, hence the same is true for $L(F)(A)$.

We now prove that $H^0(L)$ is essentially surjective. Let $G \in \Fun_S^{\mathrm{hp}}(\cat A, \hproj_R^\mathrm{rqr}(S,\cat B))$. We are going to prove that
\begin{equation} \label{eq:hproj-fun-res}
    L(Q(T(G)) \to LT(G) \xrightarrow{\varepsilon_G} G
\end{equation}
is a quasi-isomorphism, where $Q(T(G)) \to T(G)$ is an h-projective resolution of $T(G)$ and $\varepsilon_G$ is the natural morphism \eqref{eq:rqr-to-functors_counit_1}. Both $G$ and $L(Q(T(G)))$ are h-projective (Lemma \ref{lemma:changeofrings_hproj_preservation}), hence this quasi-isomorphism will be an isomorphism in $H^0(\Fun_S^{\mathrm{hp}}(\cat A, \hproj_R^\mathrm{rqr}(S,\cat B)))$ (Remark \ref{remark:Funhp_qis_invert}), as we want. Let $A \in \cat A$; since by hypothesis $G(A) \in \hproj_R^\mathrm{rqr}(S,\cat B)$, we have that
\begin{equation*}
    T(G)_A = \hproj_R^{\mathrm{rqr}}(S,\cat B)(-,G(A))
\end{equation*}
is a representable dg-module. Next, we have a quasi-isomorphism
\begin{equation} \label{eq:hproj-fun-res-A}
    Q(T(G))_A \to T(G)_A = \hproj_R^{\mathrm{rqr}}(S,\cat B)(-,G(A)). \tag{$\ast$}
\end{equation}
Since $\cat A$ is h-projective, $Q(T(G))_A$ is an h-projective dg-module \cite[Lemma 3.4]{canonaco-stellari-internalhoms}, and the same is true for the representable dg-module $\hproj_R^{\mathrm{rqr}}(S,\cat B)(-,G(A))$. Hence, \eqref{eq:hproj-fun-res-A} is a homotopy equivalence. Next, we have a commutative diagram
\begin{equation*}
 \begin{tikzcd}
L(Q(T(G)))(A) = \displaystyle\int^X Q(T(G))_A^X \otimes_S X \arrow[r, "\eqref{eq:hproj-fun-res}_A" ] \arrow[d, "\approx"'] & G(A) \arrow[d, equal]                    \\
       \displaystyle\int^X \hproj_R^{\mathrm{rqr}}(S,B)(X,G(A)) \otimes_S X       \arrow[r,"\sim"]   & G(A),
\end{tikzcd}
\end{equation*}
where the isomorphism of the bottom row follows from a similar argument as in Remark \ref{remark:leftadj_coend_representable}, and the left vertical arrow is a homotopy equivalence induced by \eqref{eq:hproj-fun-res-A}.  We conclude that indeed $L(Q(T(G)) \to G$ is a quasi-isomorphism, hence a homotopy equivalence, and we are done.

To show that $H^0(L')$ is essentially surjective, we can argue in the same way as above, replacing $L$ with $L'$ and $T$ with $T'$.
\end{proof}

\subsubsection{The proof of Theorem \ref{thm:coext_univproperty}}
We now have the tools to prove our main result. For our convenience, we recall the adjoint ($R$-linear) equivalence of Proposition \ref{prop:coextension_scalars_dg_strict}:
\begin{equation}
    l \colon \compdg_R(\cat A_R, \cat B) \leftrightarrows \Fun_S(\cat A,\compdg_R(S,\cat B))_R : r,
\end{equation}
where
\begin{equation}
    \begin{split}
        l(G)(A)_{\diamondsuit}^B &= G_A^B, \\
        l(G)_s^{f} &= G^{f}_{s 1_A}, \quad s \in S.
    \end{split}
\end{equation}
and
\begin{equation}
    r(F)_A^B  = F(A)^B_{\diamondsuit} \in \compdg(R).
\end{equation}
We also recall that both $\cat A$ is assumed to be an $S$-linear h-projective dg-category, and $\cat B$ is assumed to be an $R$-linear h-flat dg-category.

We now prove a lemma on preservation of h-projective objects.
\begin{lemma} \label{lemma:coextension_scalars_hproj}
Both dg-functors $r$ and $l$ preserve acyclic objects. Hence, they also preserve h-projective objects, and they induce an adjoint equivalence
\begin{equation} \label{eq:coextension_scalars_hproj}
    l \colon \hproj_R(\cat A_R, \cat B) \leftrightarrows \Fun^\mathrm{hp}_S(\cat A,\hproj_R(S,\cat B))_R : r.
\end{equation}
\end{lemma}
\begin{proof}
Recall from Lemma \ref{lemma:F(A)_hproj} that we may identify
\[
\Fun^{\mathrm{hp}}_S(\cat A, \hproj_R(S,\cat B)) = \Fun^{\mathrm{hp}}_S(\cat A, \compdg_R(S,\cat B)).
\]
Recalling Definition \ref{def:Funhp_acyclics_qis}, it is immediate from the definition that $r$ and $l$ preserve acyclic objects.\footnote{By the way, being acyclic does not depend on the base dg-ring $R$ or $S$.} Then, since $r$ and $l$ are mutually inverse dg-equivalences, we have ($R$-linear) isomorphisms
\begin{align*}
    \Hom(l(G),F) & \cong \Hom(G,r(F)), \\
    \Hom(r(F),G) & \cong \Hom(F,l(G)).
\end{align*}
From this, we easily conclude that both $r$ and $l$ preserve h-projective objects.
\end{proof}

We now prove:
\begin{lemma} \label{lemma:coextension_scalars_hprojrqr}
The adjoint equivalence \eqref{eq:coextension_scalars_hproj} of Lemma \ref{lemma:coextension_scalars_hproj} restricts to an $R$-linear adjoint equivalence
\begin{equation} \label{eq:coextension_scalars_hprojrqr}
    l \colon \hproj_R^\mathrm{rqr}(\cat A_R, \cat B) \leftrightarrows \Fun^\mathrm{hp}_S(\cat A,\hproj_R^\mathrm{rqr}(S,\cat B))_R : r
\end{equation}
\end{lemma}
\begin{proof}
Let $G \in \hproj_R^\mathrm{rqr}(\cat A_R, \cat B)$. Hence, for $A \in \cat A$, we have a quasi-isomorphism
\[
\cat B(-,\Phi_G(A)) \xrightarrow{\approx} l(G)(A)_{\diamondsuit} = G_A, 
\]
hence we conclude that $l(G) \in \Fun^\mathrm{hp}_S(\cat A,\hproj_R^\mathrm{rqr}(S,\cat B))$.

On the other hand, let $F \in \Fun^\mathrm{hp}_S(\cat A,\hproj_R^\mathrm{rqr}(S,\cat B))$. Then, for $A \in \Ob(\cat A_R)=\Ob(\cat A)$, we have a quasi-isomorphism
\[
\cat B(-, \Phi_{F(A)}(\diamondsuit)) \xrightarrow{\approx} r(F)_A = F(A)_{\diamondsuit},
\]
and we conclude that $r(F) \in \hproj_R^\mathrm{rqr}(\cat A_R, \cat B)$. The result now follows.
\end{proof}
Combining the above Lemma \ref{lemma:coextension_scalars_hproj} and \ref{lemma:coextension_scalars_hprojrqr} with Lemma \ref{lemma:rqr-to-functors_quasiequivalences}, we can directly deduce the following result, which includes Theorem \ref{thm:coext_univproperty}:
\begin{proposition} \label{prop:coextension_scalars_hproj_hprojrqr}
There are natural $R$-linear quasi-equivalences:
\begin{align}
  \hproj_R^{\mathrm{rqr}}(\cat A_R, \cat B)  & \xrightarrow{\sim} \hproj_S^{\mathrm{rqr}}(\cat A, \hproj_R^\mathrm{rqr}(S,\cat B))_R, \nonumber \\
 \hproj_R(\cat A_R, \cat B)   & \xrightarrow{\sim} \hproj_S^{\mathrm{rqr}}(\cat A, \hproj_R(S,\cat B))_R. \label{eq:coextension_scalars_hproj_universal}
\end{align}
\end{proposition}
\subsubsection{Coextension of scalars of $\hproj(\smallcat b)$}

% Description of the change of rings for dg-categories of the form h-proj(A)
Let $\smallcat b$ be a small $R$-linear dg-category. Let us have a closer look at the quasi-equivalence \eqref{eq:coextension_scalars_hproj_universal}. First, from \cite[Corollary 4.2]{canonaco-stellari-internalhoms} we know in general that 
\[
\hproj_R(\opp{\cat C} \lotimes \smallcat b) \cong \RHom_R(\cat C, \hproj(\smallcat b)).
\]
where $\RHom_R(-,-)$ is the internal $\Hom$ in the homotopy category of $R$-linear 
dg-categories $\Hqe(R)$, and $\cat C$ is an $R$-linear dg-categories. In particular, if either $\cat C$ or $\smallcat b$ is h-flat then we have 
\[
\hproj_R(\cat C,\smallcat b) \cong \RHom(\cat C, \hproj_R(\smallcat b)).
\]
So, a posteriori we know that we have a quasi-equivalence
\begin{equation} \label{eq:hprojrqr_hproj}
\hproj_R(\cat C, \smallcat b) \cong \hproj^\mathrm{rqr}_R(\cat C, Q \hproj_R(\smallcat b)),
\end{equation}
where $Q \hproj_R(\smallcat b) \to \hproj_R(\smallcat b)$ is an h-flat resolution. Combining these observations with \eqref{eq:coextension_scalars_hproj_universal}, we get a natural $R$-linear quasi-equivalence:
\begin{equation}
    \hproj_R^\mathrm{rqr}(\cat A_R, Q \hproj_R(\smallcat b))   \xrightarrow{\sim} \hproj_S^{\mathrm{rqr}}(\cat A, \hproj_R(S,\smallcat b))_R.
\end{equation}
So, $\hproj_R^\mathrm{rqr}(S, Q\hproj(\smallcat b))$ satisfies the universal property of Theorem \ref{thm:coext_univproperty}, and we obtain an $S$-linear quasi-equivalence
\begin{equation} \label{eq:hprojrqr_hproj_coextension}
    \hproj_R(S,\smallcat b) \cong \hproj_R^\mathrm{rqr}(S, Q\hproj(\smallcat b)),
\end{equation}
and recalling Remark \ref{remark:notation_coext_rings} we may then identify 
\begin{equation}
    \hproj_R(\smallcat b)_{(S)} = \hproj_R(S,\smallcat b)
\end{equation}
as $S$-linear dg-categories, up to quasi-equivalence (assuming that $\smallcat b$ is h-flat as an $R$-linear dg-category).

It is useful to have yet a closer look to quasi-equivalences \eqref{eq:hprojrqr_hproj} and \eqref{eq:hprojrqr_hproj_coextension}. The first one \eqref{eq:hprojrqr_hproj} is very general, and up to suitable resolutions it is described as in the following result.
\begin{proposition} \label{prop:hprojrqr_hproj}
Let $\cat C$ be an h-projective ($R$-linear) dg-category, and let $\smallcat b$ be any small ($R$-linear) dg-category. Consider the dg-functors
\begin{equation} \label{eq:hprojrqr_hproj_functor1}
\begin{split}
\Phi \colon \compdg_R(\cat C, \smallcat b) & \to \compdg_R(\cat C, \hproj_R(\smallcat b)),\\
F & \mapsto \Phi(F)_C^X = \compdg_R(\smallcat b)(X,F_C)
\end{split}
\end{equation}
and
\begin{equation} \label{eq:hprojrqr_hproj_functor2}
\begin{split}
    \Psi \colon \compdg_R(\cat C, \hproj_R(\smallcat b)) & \to \compdg_R(\cat C, \smallcat b), \\
    G &\mapsto \Psi(G)_C^B = G_C^{h(B)},
\end{split}
\end{equation}
where $h \colon \smallcat b \to \hproj_R(\smallcat b)$ is the Yoneda embedding. These dg-functors are adjoint:
\[
\Psi \dashv \Phi,
\]
and they induce mutually inverse quasi-equivalences:
\begin{equation} \label{eq:hprojrqr_hproj_functors12}
    \Psi \colon \hproj_R^\mathrm{rqr}(\cat C, \hproj_R(\smallcat b)) \leftrightarrows \hproj_R(\cat C, \smallcat b)  : \Phi. 
\end{equation}
\end{proposition}
\begin{proof}
We first notice that we have an isomorphism
\begin{equation} \label{eq:hprojrqr_hproj_counit}
    \varepsilon_F \colon \Psi(\Phi(F))_C^B = \compdg_R(\smallcat b)(h(B), F_C) \xrightarrow{\sim} F_C^B,
\end{equation}
naturally in $C \in \cat C$, $B \in \smallcat b$ and $F \in \compdg_R(\cat C,\smallcat b)$, thanks to the Yoneda lemma. Moreover, for $G \in \compdg_R(\cat C, \hproj_R(\smallcat b))$, the action
\[
G_C^X \otimes_R \hproj_R(\smallcat b)(h(B),X) \to G_C^{h(B)} = \Psi(G)_C^B
\]
yields a natural morphism
\begin{equation} \label{eq:hprojrqr_hproj_unit}
  \eta_G \colon  G_C^X \to \compdg_R(\smallcat b)(X, G_C^{h(-)}) = \Phi(\Psi(G))_C^X,
\end{equation}
again recalling the Yoneda lemma: $\hproj_R(\smallcat b)(h(-),X) \cong X$. A direct inspection shows that the triangle identities hold:
\begin{align*}
    (\Psi \xrightarrow{\Psi \circ \eta} \Psi \Phi \Psi \xrightarrow{\varepsilon \circ \Psi} \Psi) = 1_{\Psi}, \\
    (\Phi \xrightarrow{\eta \circ \Phi} \Phi \Psi \Phi \xrightarrow{\Phi \circ \varepsilon} \Phi) = 1_{\Phi}. 
\end{align*}
This proves that $\Psi$ is left adjoint to $\Phi$.

Moreover, we observe that $\Phi$ preserves acyclic objects. Indeed, if $F$ is acyclic, then
\[
\Phi(F)_C^X = \compdg_R(\smallcat b)(X,F_C)
\]
is also acyclic, since $X$ is h-projective and $F_C$ is acyclic. From this, we deduce that $\Psi$ preserves h-projective objects, being left adjoint to $\Phi$: this follows directly from the isomorphism
\[
\Hom(\Psi(G),F) \cong \Hom(G,\Phi(F)).
\]

Hence, $\Psi$ induces a dg-functor
\[
\Psi \colon \hproj_R^\mathrm{rqr}(\cat C, \hproj_R(\smallcat b)) \to \hproj_R(\cat C,\smallcat b).
\]
We first show that $\Psi$ is quasi-fully faithful. On hom-complexes, $\Psi$ is defined by the composition
\[
\Hom(G',G) \xrightarrow{(\eta_{G})_*} \Hom(G', \Phi(\Psi(G))) \cong \Hom(\Psi(G), \Psi(G')).
\]
We want to show that $(\eta_{G})_*$ is a quasi-isomorphism. Since $G$ is taken to be h-projective, this holds if $\eta_{G} \colon G \to \Phi(\Psi(G))$ is a quasi-isomorphism. By hypothesis, $G$ is right quasi-representable; hence, we obtain a commutative diagram:
\begin{equation*}
    \begin{tikzcd}
G_C^X \arrow[r, "\eta_G"]                                     & {\compdg_R(\smallcat b)(X, G_C^{h(-)})}                                 \\
{\hproj_R(\smallcat b)(X,\Phi_G(C))} \arrow[r, "\sim"] \arrow[u, "\sim"] & {\compdg_R(\smallcat b)(X, \hproj_R(\smallcat b)(h(-),\Phi_G(C))).} \arrow[u, "\sim"]
\end{tikzcd}
\end{equation*}
The left vertical arrow is a quasi-isomorphism thanks to the right quasi-representability of $G$, where $\Phi_G(C) \in \hproj_R(\smallcat b)$. The right vertical arrow is a quasi-isomorphism, since $X$ is h-projective and hence $\compdg_R(\smallcat b)(X,-)$ preserves quasi-isomorphisms. Finally, the lower horizontal arrow is an isomorphism thanks to the Yoneda lemma. We conclude that $\eta_G$ is a quasi-isomorphism, as claimed.

We now show that $\Psi$ is quasi-essentially surjective. Let $F \in \hproj_R(\cat C,\smallcat b)$. Since $\cat C$ is h-projective, we deduce \cite[Lemma 3.4]{canonaco-stellari-internalhoms} that $F_C$ lies in $\hproj_R(\smallcat b)$. Hence:
\[
\Phi(F)_C = \hproj_R(\smallcat b)(-, F_C)
\]
is representable for all $C \in \cat C$, and $\Phi(F)$ is in particular right quasi-representable. Next, we take an h-projective resolution $Q(\Phi(F)) \to F$, and we consider the composition
\[
\Psi(Q(\Phi(F))) \to \Psi(\Phi(F)) \xrightarrow{\varepsilon_F} F.
\]
The morphsm $\varepsilon_F$ is an isomorphism, but it is immediate to see that $\Psi$ preserves quasi-isomorphisms. Recalling that it also preserves h-projective objects, we conclude that $\Psi(Q(\Phi(F))) \to F$ is a quasi-isomorphism between h-projective objects, hence it is an isomorphism in $H^0(\hproj_R(\cat C, \smallcat b))$, thus proving our claim.
\end{proof}
We can now obtain a more precise description of the quasi-equivalence \eqref{eq:hprojrqr_hproj_coextension}. 
\begin{proposition} \label{prop:ring_coext_hprojA}
Let $\smallcat b$ be a small h-flat $R$-linear dg-category, and let $q \colon Q \hproj_R(\smallcat b) \to \hproj_R(\smallcat b)$ be an $R$-linear h-flat resolution. There are natural $S$-linear quasi-equivalences:
\begin{equation} \label{eq:ring_coext_hprojA}
 \hproj_R^\mathrm{rqr}(S, Q\hproj(\smallcat b)) \leftrightarrows \hproj_R(S,\smallcat b).
\end{equation}
If viewed as $R$-linear quasi-equivalences between the underlying $R$-linear dg-categories, they can be identified (up to composing with suitable resolutions, see also \S \ref{ssec:qf}) with the quasi-equivalences $\Psi$ and $\Phi$ of the above Proposition \ref{prop:hprojrqr_hproj}.
\end{proposition}
\begin{proof}
We can directly observe that the dg-functor
\begin{align*}
\Phi' \colon \compdg_R(S,\smallcat b) &\to \compdg_R(S, Q \hproj_R(\smallcat b)), \\
F & \mapsto \Phi(F)_{\diamondsuit}^X = \compdg_R(\smallcat b)(q(X),F_{\diamondsuit}).
\end{align*}
is indeed $S$-linear. Moreover, $\Phi'(F)$ is right quasi-representable for all $F$. Indeed, we may take an h-projective resolution $Q_{\smallcat b}(F_{\diamondsuit}) \to F_{\diamondsuit})$, and since $q(X)$ is h-projective for all $X \in Q\hproj_R(\smallcat b)$, we have quasi-isomorphisms
\begin{align*}
\Phi'(F)_{\diamondsuit} &= \compdg(\smallcat b)(q(-), F_{\diamondsuit}) \\
& \cong \hproj_R(\smallcat b)(q(-), Q_{\smallcat b}(F_{\diamondsuit})) \\
& \cong Q\hproj_R(\smallcat b)(-,F'),
\end{align*}
where $F' \in Q\hproj(\smallcat b)$ is some object such that $q(F') \cong Q_{\smallcat b}(F_{\diamondsuit})$ in $H^0(\hproj_R(\smallcat b))$. Hence, $\Phi'$ induces an $S$-linear quasi-functor
\[
\Phi' \colon \hproj_R(S, \smallcat b) \to \hproj_R^\mathrm{rqr}(S, Q\hproj_R(\smallcat b)).
\]
If we take the induced $R$-linear quasi-functor between the underlying $R$-linear dg-categories and we compose with the suitable resolutions, it is clear that we can identify $\Phi'$ with
\[
\Phi \colon \hproj_R(Q(S), \smallcat b) \to \hproj_R^\mathrm{rqr}(Q(S),\hproj_R(\smallcat b))
\]
of Proposition \ref{prop:hprojrqr_hproj}, where $Q(S) \to S$ is an h-projective resolution of $S$ viewed as an $R$-linear dg-category. We conclude that $\Phi'$ is indeed an $S$-linear quasi-equivalence, with quasi-inverse induced by $\Psi$.

Another technique to prove the claim is as follows. We consider the following diagram of $R$-linear quasi-functors:
\begin{equation*}
    \begin{tikzcd}
{\hproj_R^{\mathrm{rqr}}(\cat A_R, Q \hproj_R(\smallcat b))} \arrow[r, leftrightarrow] \arrow[d, leftrightarrow] & {\hproj_S^{\mathrm{rqr}}(\cat A, \hproj_R^\mathrm{rqr}(S, Q \hproj_R(\smallcat b)))_R} \arrow[d, dotted, leftrightarrow] \\
{\hproj_R(\cat A_R, \smallcat b)} \arrow[r, leftrightarrow]                   & {\hproj_S^{\mathrm{rqr}}(\cat A, \hproj_R(S,\smallcat b))_R.}          
\end{tikzcd}
\end{equation*}
The non-dotted arrows are quasi-equivalences. The horizontal arrows are given by Proposition \ref{prop:coextension_scalars_hproj_hprojrqr}; the left vertical arrows are induced by $\Phi$ and $\Psi$ (up to composing with suitable resolutions); the right vertical dotted arrows are the ones which make the diagram commute, and by a Yoneda argument we see that they are induced by $S$-linear quasi-equivalences
\[
\hproj_R^\mathrm{rqr}(S, Q\hproj(\smallcat b)) \leftrightarrows \hproj_R(S,\smallcat b).
\]
If we take $\cat A=S$ in the above diagram, we see that the horizontal arrows can be identified with identities, upon identifying
\begin{align*}
    \hproj_S^{\mathrm{rqr}}(S, \hproj_R^\mathrm{rqr}(S, Q \hproj_R(\smallcat b)))_R & \cong \hproj_R^\mathrm{rqr}(S, Q\hproj_R(\smallcat b))_R, \\
    \hproj_S^{\mathrm{rqr}}(S, \hproj_R(S,\smallcat b))_R &\cong \hproj_R(S,\smallcat b)_R.
\end{align*}
We conclude that, in that case, the dotted vertical arrows can be identified with the left vertical arrows, namely, the quasi-equivalences induced by $\Phi$ and $\Psi$. This is precisely our claim.
\end{proof}

% Extensions of scalars
\subsection{Extension of scalars} \label{subsection:scalar_leftextension}
The restriction of scalars functor
\begin{equation*}
    \begin{split}
        \Hqe(S) & \to \Hqe(R), \\
        \smallcat b &\mapsto \smallcat b_R
    \end{split}
\end{equation*}
has also a left adjoint. This is obtained by \emph{extension of scalars} along $R \to S$. We shall define it more in general for locally small dg-categories.
\begin{definition}
Let $\cat A$ be an $R$-linear dg-category. We define an $S$-linear dg-category $S \otimes_R \cat A$ as follows:
\begin{equation}
\begin{split}
    \Ob(S \otimes_R \cat A) &= \Ob(S) \times \Ob(\cat A) = \{\diamondsuit\} \times \Ob(\cat A) = \Ob(\cat A), \\
    (S \otimes_R \cat A)(A,B) &= S \otimes_R \cat A(A,B),
\end{split}
\end{equation}
where the latter hom-complex has a natural $S$-linear structure.
\end{definition}
We first recall the basic (non derived) extension of scalars adjunction.
\begin{proposition} \label{prop:scalars_leftextension_nonderived}
Let $\cat A$ be an $R$-linear dg-category, and let $\cat B$ be an $S$-linear dg-category. There is a natural adjoint dg-equivalence of $R$-linear dg-categories:
\begin{equation}
    l \dashv r \colon \compdg_R(\cat A, \cat B_R) \leftrightarrows \compdg_S(S \otimes_R \cat A, \cat B)_R
\end{equation}

The dg-functor $l$ is defined as follows: for $F \in \compdg_R(\cat A, \cat B_R)$, we set:
\begin{align*}
    l(F)_{(\diamondsuit, A)}^B  &= F_A^B, \\
    l(F)_{(s \otimes f)}^g &= (-1)^{|f||s|} F^{sg}_f,
\end{align*}
where $F_A^B$ has an $S$-linear structure given by
\begin{equation*}
    s \cdot x = F^{s \cdot 1_B}_{1_A}(x),
\end{equation*}
for $x \in F_A^B$.

The dg-functor $r$ is defined by restriction:
\begin{equation*}
    r(F)_A^B = F_{(\diamondsuit, A)}^B.
\end{equation*}
\end{proposition}
\begin{lemma}
Assume that $\cat A$ is an h-flat $R$-linear dg-category and that $\cat B$ is a small h-flat $S$-linear dg-category. The above dg-functors $r$ and $l$ preserve acyclics, so they both preserve h-projectives and they induce an adjoint $R$-linear equivalence
\begin{equation*}
    l \colon \hproj_R(\cat A, \cat B_R) \leftrightarrows \hproj_S(S \otimes_R \cat A, \cat B)_R : r 
\end{equation*}
\end{lemma}
\begin{proof}
The fact that $r$ and $l$ preserve acyclics is immediate from the definition. Since they define an adjoint equivalence, they are each left adjoint to the other. Hence, they preserve h-projectives (this an argument we have already used before, for example in Proposition \ref{prop:hprojrqr_hproj}).
\end{proof}
Finally, we prove:
\begin{proposition} \label{prop:scalar_leftextension}
Assume that $\cat A$ is an h-flat $R$-linear dg-category and that $\cat B$ is a small h-flat $S$-linear dg-category. The above dg-functors $r$ and $l$ induce an $R$-linear adjoint equivalence
\begin{equation}
    l \colon \hproj^{\mathrm{rqr}}_R(\cat A, \cat B_R) \xrightarrow{\sim} \hproj_S^{\mathrm{rqr}}(S \otimes_R \cat A, \cat B)_R : r.
\end{equation}
\end{proposition}
\begin{proof}
We only need to show that both $l$ and $r$ preserve right quasi-representable bimodules. First, let $F \in \hproj^{\mathrm{rqr}}_R(\cat A, \cat B_R)$. We show that $l(F) \in \hproj_S^{\mathrm{rqr}}(S \otimes_R \cat A, \cat B)_R$. By hypothesis, we have for all $A \in \cat A$ an $R$-linear quasi-isomorphism of right $\cat B_R$-dg-modules
\begin{equation} \label{tag:rqr_scalarextensions}
\cat B_R(-,\Phi_F(A)) \xrightarrow{\approx} F_A, \tag{$\star$}
\end{equation}
which is determined by an element $e \in F^{\Phi_F(A)}_A$. We also notice that $F^{\Phi_F(A)}_A = l(F)_{(\diamondsuit, A)}^{\Phi_F(A)}$, so the element $e$ determines an $S$-linear morphism
\[
\cat B(-, \Phi_F(A)) \to l(F)_{(\diamondsuit, A)},
\]
which is clearly a lift of \eqref{tag:rqr_scalarextensions}, hence it is a quasi-isomorphism as well. This proves that $l(F)$ is right quasi-representable.

Next, let $G \in \hproj_S^{\mathrm{rqr}}(S \otimes_R \cat A, \cat B)_R$. By hypothesis, we have for all $A \in \cat A$ an $S$-linear quasi-isomorphism of right $\cat B$-dg-modules
\[
\cat B(-, \Phi_G(\diamondsuit, A)) \xrightarrow{\approx} G_{(\diamondsuit, A)}.
\]
By definition, $r(G)_A=G_{(\diamondsuit,A)}$, so restricting scalars we obtain an $R$-linear quasi-isomorphism of right $\cat B_R$-dg-modules
\[
\cat B_R(-,\Phi_G(\diamondsuit,A)) \xrightarrow{\sim} r(G)_A,
\]
and we conclude that $r(G)$ is right quasi-representable, as we wanted.
\end{proof}
\begin{corollary}
Let $\smallcat a$ be a small $R$-linear dg-category and let $\smallcat b$ be a small $S$-linear dg-category. We denote by $S \lotimes_R \smallcat a$ the $S$-linear dg-category $S \otimes_R Q(\smallcat a)$, where $Q(\smallcat a)\to \smallcat a$ is an h-flat resolution of $\smallcat a$ (the choice of resolution does not matter up to isomorphism in $\Hqe(S)$). Then, there is a natural bijection
\begin{equation}
    \Hqe(R)(\smallcat a, \smallcat b_R) \cong \Hqe(S)(S \lotimes_R \cat A, \smallcat b),
\end{equation}
so that the functor
\begin{equation}
\begin{split}
\Hqe(R) & \to \Hqe(S), \\
\smallcat a & \mapsto S \lotimes \smallcat a
\end{split}
\end{equation}
is left adjoint to the restriction of scalars $\Hqe(S) \to \Hqe(R)$.
\end{corollary}

% t-structures and change of rings
\subsection{Change of rings and t-structures} \label{subsection:changeofrings_tstructures}
Working with t-structures, we resume the assumption that the dg-rings $R$ and $S$ are strictly concentrated in nonpositive degrees. 

The coextension of scalars is well behaved with respect with t-structures. If $\cat B$ is an h-flat pretriangulated $R$-linear dg-category with a t-structure $(\cat B_{\leq 0}, \cat B_{\geq 0})$, we endow the $S$-linear dg-category
\[
\cat B_{(S)} = \hproj_R^\mathrm{rqr}(S,\cat B)
\]
with the natural t-structure of Theorem \ref{thm:tstruct_quasifunctors}. We deduce from Proposition \ref{prop:tstruct_quasifunctors_heart}  an $H^0(S)$-linear equivalence:
\footnote{{To be completely rigorous, we should indeed check that the equivalence supplied by Proposition \ref{prop:tstruct_quasifunctors_heart} is $H^0(S)$-linear, if we endow $\Fun_{H^0(R)}(H^0(S),H^0(\cat B)^\heartsuit)$ with the natural $H^0(S)$-linear structure (cf. \cite[\S 4]{lowen-vandenbergh-deformations-abelian}). But this is straightforward.}}
\[
H^0(\cat B_{(S)})^\heartsuit \cong \Fun_{H^0(R)}(H^0(S),H^0(\cat B)^\heartsuit).
\]
We remark that the category $\Fun_{H^0(R)}(H^0(S),H^0(\cat B)^\heartsuit)$ is precisely the category of $H^0(S)$-linear objects $H^0(\cat B)^\heartsuit_{(H^0(S))}$ described in \cite[\S 4]{lowen-vandenbergh-deformations-abelian}. We may then write:
\begin{equation}
    H^0(\cat B_{(S)})^\heartsuit \cong H^0(\cat B)^\heartsuit_{(H^0(S))}.
\end{equation}
Proposition \ref{prop:tstruct_postcomposition} immediately yields the following result, which tells us that the coextension of scalars is actually stable under \emph{t-exact} quasi-equivalences.
 \begin{corollary} \label{coroll:tstruct_changeofrings}
Let $\cat B$ and $\cat B'$ be pretriangulated h-flat $R$-linear dg-categories endowed with t-structures, and let $F \colon \cat B \to \cat B'$ be a quasi-functor. If $F$ is left t-exact (right t-exact, t-exact), the induced quasi-functor
\begin{equation}
    F_* \colon \cat B_{(S)} \to \cat B'_{(S)}
\end{equation}
is left t-exact (right t-exact, t-exact). If $F$ is a t-exact quasi-equivalence, the same is true for $F_*$.
\end{corollary}

Let us now consider the change of rings for the dg-category of dg-modules. From Proposition \ref{prop:ring_coext_hprojA}, we know that we have a natural quasi-equivalence
\[
\Psi \colon \hproj_R^\mathrm{rqr}(S, Q\hproj(\smallcat b)) \xrightarrow{\sim} \hproj_R(S,\smallcat b).
\]
for a given h-flat small dg-category $\smallcat b$. In addition, if $\smallcat b$ has cohomology concentrated in nonpositive degrees, we know that $\hproj(\smallcat b)$ has the natural t-structure; we can transport this t-structure on $Q\hproj(\smallcat b)$ via the h-flat resolution $Q \hproj(\smallcat b) \to \hproj(\smallcat b)$. We endow $\hproj_R^\mathrm{rqr}(S, Q\hproj(\smallcat b))$ with the t-structure of Theorem \ref{thm:tstruct_quasifunctors} and $\hproj_R(S,\smallcat b) = \hproj_R(\smallcat b \otimes_R \opp{S})$ with the canonical t-structure (both $S$ and $\smallcat b$ have cohomology concentrated in nonpositive degrees, so the same is true for the tensor product, see Lemma \ref{lemma:tstruct_tensorprod}). Then, we can check:
\begin{lemma} \label{lemma:tstruct_hproj_coextension}
Let $\smallcat b$ be an h-flat small dg-category. The $S$-linear quasi-equivalences \eqref{eq:ring_coext_hprojA} of Proposition \ref{prop:ring_coext_hprojA},
\[
 \hproj_R^\mathrm{rqr}(S, Q\hproj(\smallcat b)) \leftrightarrows \hproj_R(S,\smallcat b),
\]
are t-exact.
\end{lemma}
\begin{proof}
t-exactness does not depend on the base dg-ring, so we may take an h-flat resolution $Q(S) \to S$ of $S$ viewed as an $R$-linear dg-category, and observe that we have t-exact ($R$-linear) quasi-equivalences
\begin{align*}
    \hproj_R^\mathrm{rqr}(S, Q\hproj(\smallcat b)) & \cong \hproj_R^\mathrm{rqr}(Q(S),\hproj(\smallcat b)), \\
    \hproj_R(S,\smallcat b) &\cong \hproj_R(Q(S),\smallcat b).
\end{align*}
The first one is obtained by pre/post-composition with the h-flat resolutions, cf. Proposition \ref{prop:tstruct_precomposition} and Proposition \ref{prop:tstruct_postcomposition}; the second one is obtained by restricting along $Q(S) \otimes \opp{\smallcat b} \to S \otimes \smallcat b$ (restriction of dg-modules is always t-exact). Thanks to the above t-exact quasi-equivalences and recalling the proof of Proposition \ref{prop:ring_coext_hprojA}, it is then enough to prove that the adjoint quasi-equivalences \eqref{eq:hprojrqr_hproj_functors12} of Proposition \ref{prop:hprojrqr_hproj},
\[ 
    \Psi \colon \hproj_R^\mathrm{rqr}(Q(S), \hproj_R(\smallcat b)) \leftrightarrows \hproj_R(Q(S), \smallcat b)  : \Phi,
\]
are t-exact. Of course, it is enough to prove that $\Psi$ is t-exact. If $G \in \hproj_R^\mathrm{rqr}(Q(S), \hproj_R(\smallcat b))$, then by definition
\[
\Psi(G)_{\diamondsuit}^B = G^{h(B)}_{\diamondsuit}.
\]
We then have:
\begin{align*}
H^k(\Psi(G)_{\diamondsuit}^B) & \cong H^k(\hproj(\smallcat b)(h(B), \Phi_G(\diamondsuit))) \\
& \cong H^k(\Phi_G(\diamondsuit)^B). \qquad \text{(Yoneda)}
\end{align*}
From this, recalling how the aisles of $\hproj_R^\mathrm{rqr}(Q(S), \hproj_R(\smallcat b))$ and $\hproj_R(Q(S), \smallcat b)$ are defined (cf. Proposition \ref{prop:dercomp_naturaltstruct} and Theorem \ref{thm:tstruct_quasifunctors}), we immediately conclude.
\end{proof}

% Homotopically finitely presented objects 

\subsubsection{Homotopically finitely presented objects} \label{subsubsec:hfp_obj}

We now put some finiteness conditions on $R$ and $S$.\footnote{We assume the reader to be familiar with the basic notions about coherent rings, finitely generated and finitely presented modules.}
\begin{definition} \label{def:homotopically_fp}
Let $B$ be a commutative dg-ring strictly concentrated in nonpositive degrees.
\begin{itemize}
\item We say that $B$ is \emph{homotopically coherent} if $H^0(B)$ is a coherent commutative ring, and if $H^i(R)$ is a finitely presented $H^0(R)$-module for all $i \in \mathbb Z$:
\[
H^i(B) \in \operatorname{mod}(H^0(B)).
\]

\item Let $M$ be a $B$-dg-module. We say that $M$ is \emph{homotopically finitely presented (hfp)} if $H^i(M)$ is a finitely presented $H^0(B)$-module for all $i \in \mathbb Z$:
\[
H^i(M) \in \operatorname{mod}(H^0(B)).
\]

We denote by $\hfp(B)$ the full dg-subcategory of $\compdg(B)$ spanned by the homotopically finitely presented dg-modules.
\end{itemize}
\end{definition}

In this part we assume that our dg-rings $R$ and $S$ are both homotopically coherent, and moreover that $S$ is homotopically finitely presented as an $R$-dg-module: $S \in \hfp(R)$.
\begin{definition}
  \label{def:hfp}
Let $\cat B$ be a pretriangulated dg-category with a t-structure $(\cat B_{\leq 0}, \cat B_{\geq 0})$. We define the full dg-subcategory $\hfp(\cat B)$ of \emph{homotopically finitely presented objects} of $\cat B$ as follows:
\begin{equation}
    \hfp(\cat B) = \{ B \in \cat B : H^n_t(B) \in \fp(H^0(\cat B)^\heartsuit) \quad \forall\, n \in \mathbb Z \},
\end{equation}
where $\fp(H^0(\cat B)^\heartsuit)$ is the full subcategory of $H^0(\cat B)^\heartsuit$ of finitely presented objects, cf. \cite[\S 2.2]{lowen-vandenbergh-deformations-abelian}.

Moreover, we set:
\begin{equation}
    \hfpbd(\cat B) = \{ B \in \hfp(\cat B) : \tau_{\leq -M} B = 0 \text{ and }  \tau_{\geq M} B = 0, \ \text{for some  $M > 0$}\}.
\end{equation}
\end{definition}
\begin{proposition} \label{prop:hfp_changeofrings}
Let $\cat B$ be an h-flat pretriangulated $R$-linear dg-category with a t-structure $(\cat B_{\leq 0}, \cat B_{\geq 0})$. Assume that filtered colimits in the heart $H^0(\cat B)^\heartsuit$ are exact. The natural ($S$-linear) inclusion functor
\[
\hfp(\hproj_R^\mathrm{rqr}(S, \cat B)) \hookrightarrow \hproj_R^\mathrm{rqr}(S, \cat B)
\]
factors through a ($S$-linear) quasi-equivalence
\begin{equation} \label{eq:hfp_changeofrings}
    \hfp(\hproj_R^\mathrm{rqr}(S, \cat B)) \xrightarrow{\sim} \hproj_R^\mathrm{rqr}(S, \hfp(\cat B)).
\end{equation}
\end{proposition}
\begin{proof}
We may identify $\hproj_R^\mathrm{rqr}(S, \hfp(\cat B))$ with the full dg-subcategory of $\hproj_R^\mathrm{rqr}(S,\cat B)$ of quasi-functors whose essential image lies in $\hfp(\cat B)$. Hence, we only have to check that the quasi-essential image of $\hfp(\hproj_R^\mathrm{rqr}(S, \cat B)) \hookrightarrow \hproj_R^\mathrm{rqr}(S, \cat B)$ is precisely $\hproj_R^\mathrm{rqr}(S, \hfp(\cat B))$.

Let $F \in \hfp(\hproj_R^\mathrm{rqr}(S, \cat B))$. We have a quasi-isomorphism of right $\cat B$-dg-modules:
\[
F_{\diamondsuit} \cong \cat B(-,\Phi_F(\diamondsuit)),
\]
for some $\Phi_F(\diamondsuit) \in \cat B$. According to Proposition \ref{prop:tstruct_quasifunctors_heart} and the discussion at the beginning of this \S \ref{subsection:changeofrings_tstructures}, we identify
\[
H^0(\hproj_R^\mathrm{rqr}(S, \cat B))^\heartsuit \cong \Fun_{H^0(R)}(H^0(S),H^0(\cat B)^\heartsuit).
\]
as $H^0(S)$-linear categories. Thanks to \cite[Proposition 4.6]{lowen-vandenbergh-deformations-abelian}, we may further identify
\[
\fp(\Fun_{H^0(R)}(H^0(S), H^0(\cat B)^\heartsuit)) \cong \Fun_{H^0(R)}(H^0(S), \fp(H^0(\cat B)^\heartsuit))
\]
Then, for all $i$, we have by assumption that $H^i_t(F) \in \Fun_{H^0(R)}(H^0(S), \fp(H^0(\cat B)^\heartsuit))$. Thanks to the above identifications and recalling how the t-structure and the t-cohomologies on $\hproj_R^\mathrm{rqr}(S,\cat B)$ are defined (Remark \ref{remark:tstruct_quasifunctors_cohomologies}), we conclude that
\[
H^i_t(\Phi_F(\diamondsuit)) = H^i_t(F)(\diamondsuit) \in \fp(H^0(\cat B))^\heartsuit,
\]
This proves that the quasi-essential image of $F$ lies in $\hfp(\cat B)$. 

On the other hand, let $F \in \hproj_R^\mathrm{rqr}(S, \hfp(\cat B))$. Viewing $F$ as a quasi-functor $S \to \cat B$ with quasi-essential image in $\hfp(\cat B)$, we have a quasi-isomorphism of right $\cat B$-dg-modules
\[
F_{\diamondsuit} \cong \cat B(-,\Phi_F(\diamondsuit)),
\]
where $\Phi_F(\diamondsuit) \in \hfp(\cat B)$.
We want to prove that
\[
H^i_t(F) \in \fp(\Fun_{H^0(R)}(H^0(S), H^0(\cat B)^\heartsuit)) \cong \Fun_{H^0(R)}(H^0(S), \fp(H^0(\cat B)^\heartsuit),
\]
where we used \cite[Proposition 4.6]{lowen-vandenbergh-deformations-abelian} again. This amounts to checking that
\[
H^i_t(F)(\diamondsuit) = H^i_t(\Phi_F(\diamondsuit)) \in \fp(H^0(\cat B))^\heartsuit,
\]
but this holds by assumption.
\end{proof}

\begin{corollary} \label{coroll:hfpbd_changeofrings}
Let $\cat B$ be an h-flat pretriangulated $R$-linear dg-category with a t-structure $(\cat B_{\leq 0}, \cat B_{\geq 0})$. Assume that filtered colimits in the heart $H^0(\cat B)^\heartsuit$ are exact. The ($S$-linear) quasi-equivalence \eqref{eq:hfp_changeofrings} of Proposition \ref{prop:hfp_changeofrings} restricts to an $S$-linear quasi-equivalence
\begin{equation} \label{eq:hfpbd_changeofrings}
    \hfpbd(\hproj_R^\mathrm{rqr}(S, \cat B)) \xrightarrow{\sim} \hproj_R^\mathrm{rqr}(S, \hfpbd(\cat B)).
\end{equation}
\end{corollary}
\begin{proof}
The result follows by recalling that the truncation functors of the t-structure $\hproj_R^\mathrm{rqr}(S, \cat B)$ are defined ``objectwise'' using the truncation functors of $\cat B$ (cf Remark \ref{remark:tstruct_quasifunctors_cohomologies}).
\end{proof}
\begin{remark} \label{remark:hfp_changeofrings_inverse}
From the proof of Proposition \ref{prop:hfp_changeofrings}, we see that a quasi inverse of \eqref{eq:hfp_changeofrings} is given by the natural quasi-functor
\begin{equation}
\hproj_R^\mathrm{rqr}(S, \hfp(\cat B)) \xrightarrow{\sim} \hfp(\hproj_R^\mathrm{rqr}(S, \cat B))
\end{equation}
induced by postcomposition with the inclusion $\hfp(\cat B) \hookrightarrow \cat B$.

Analogously, a quasi-inverse of \eqref{eq:hfpbd_changeofrings} is given by the quasi-functor
\begin{equation}
\hproj_R^\mathrm{rqr}(S, \hfpbd(\cat B)) \xrightarrow{\sim} \hfpbd(\hproj_R^\mathrm{rqr}(S, \cat B))
\end{equation}
induced by postcomposition with the inclusion $\hfpbd(\cat B) \hookrightarrow \cat B$.
\end{remark}
In some interesting cases, the t-structure on a dg-category $\cat B$ restricts to t-structures on $\hfp(\cat B)$ and $\hfpbd(\cat B)$, and the same is true for the coextension of scalars dg-category.
\begin{lemma} \label{lemma:hfp_tstruct_restrict}
Let $\cat B$ be an h-flat $R$-linear pretriangulated dg-category endowed with a t-structure, such that filtered colimits in the heart $H^0(\cat B)^\heartsuit$ are exact. Moreover, assume that the full dg-subcategory $\hfp(\cat B)$ is pretriangulated. Then:
\begin{enumerate}
    \item The t-structure on $\cat B$ restricts to t-structures on $\hfp(\cat B)$ and $\hfpbd(\cat B)$ so that the inclusions 
    \[
    \hfpbd(\cat B) \hookrightarrow \hfp(\cat B) \hookrightarrow \cat B
    \]
    are t-exact. \label{enum:t-struct_hfp_pretr}
    \item \label{enum:hfp_functoriality_texact} If $F \colon \cat B \to \cat B'$ is a t-exact quasi-equivalence between dg-categories both satisfying the above assumptions, then $F$ restricts to t-exact quasi-equivalences
    \begin{align*}
    \hfp(F) \colon \hfp(\cat B) & \to \hfp(\cat B'), \\
    \hfpbd(F) \colon \hfpbd(\cat B) & \to \hfpbd(\cat B').
    \end{align*}
    \item  \label{enum:hfp_changeofrings_texact} The coextension of scalars dg-category $\hproj_R^\mathrm{rqr}(S, \hfp(\cat B))$ is pretriangulated and is endowed with the t-structure of Theorem \ref{thm:tstruct_quasifunctors}. The inclusions $\hfp(\cat B) \hookrightarrow \cat B$ and $\hfpbd(\cat B) \hookrightarrow \cat B$ induce t-exact $S$-linear quasi functors
    \begin{align*}
    \hproj_R^\mathrm{rqr}(S, \hfp(\cat B)) & \to \hproj_R^\mathrm{rqr}(S, \cat B), \\
    \hproj_R^\mathrm{rqr}(S, \hfpbd(\cat B)) & \to \hproj_R^\mathrm{rqr}(S, \cat B).
    \end{align*}
    which give t-exact $S$-linear quasi-equivalences
    \begin{align*}
    \hproj_R^\mathrm{rqr}(S, \hfp(\cat B)) & \xrightarrow{\sim} \hfp(\hproj_R^\mathrm{rqr}(S, \cat B)), \\
    \hproj_R^\mathrm{rqr}(S, \hfpbd(\cat B)) & \xrightarrow{\sim} \hfpbd(\hproj_R^\mathrm{rqr}(S, \cat B)).
    \end{align*}
\end{enumerate} 
\end{lemma}
\begin{proof}
For \eqref{enum:t-struct_hfp_pretr}, let $X \in \hfp(\cat B)$, and consider $\tau_{\leq 0} X$. If $k>0$ we have $H^k(\tau_{\leq 0} X) = 0$ and if $k\leq 0$ we have $H^k(\tau_{\leq 0} X)= H^k(X)$. In any case, $H^k(\tau_{\leq 0} X) \in \fp(H^0(\cat B)^\heartsuit)$ for all $k$, and the same is true for $\tau_{\geq 0} X$ for similar reasons. This proves that the t-structure on $\cat B$ restricts to $\hfp(\cat B)$, and the inclusion $\hfp(\cat B) \hookrightarrow \cat B$ is t-exact. Moreover, the full subcategory $\hfpbd(\cat B)$ of $\hfp(\cat B)$ is clearly pretriangulated, and the t-structure of $\hfp(\cat B)$ is easily seen to restrict to $\hfpbd(\cat B)$.

For part \eqref{enum:hfp_functoriality_texact}, we first observe that $F$ induces an exact equivalence $F^\heartsuit \colon H^0(\cat B)^\heartsuit \xrightarrow{\sim} H^0(\cat B')^\heartsuit$. Hence, it also induces an equivalence $\fp(H^0(\cat B)^\heartsuit ) \xrightarrow{\sim} \fp(H^0(\cat B')^\heartsuit)$. The claim now follows using that $F$ commutes with truncations and t-cohomologies, being t-exact.

Finally, part \eqref{enum:hfp_changeofrings_texact} follows from Proposition \ref{prop:tstruct_postcomposition}, Proposition \ref{prop:hfp_changeofrings}, Corollary \ref{coroll:hfpbd_changeofrings} and Remark \ref{remark:hfp_changeofrings_inverse}.
\end{proof}

\begin{example} \label{example:hlc_hfp}
Let $\smallcat q$ be a small homotopically locally coherent $R$-linear dg-category (see \cite[Definition 5.8]{genovese-lowen-vdb-dginj}), which we assume to be h-flat. We know from \cite[Theorem 5.9]{genovese-lowen-vdb-dginj} that the natural t-structure on $\hproj(\smallcat q)$ (whose heart $\Mod(H^0(\smallcat q))$ certainly has exact filtered colimits) restricts to the full pretriangulated dg-subcategory $\hproj(\smallcat q)^\mathrm{hfp} = \hfp(\hproj(\smallcat q))$. From Lemma \ref{lemma:tstruct_hproj_coextension} we get a t-exact $S$-linear quasi-equivalence:
\[
\hproj_R^\mathrm{rqr}(S, Q \hproj(\smallcat q)) \cong \hproj_R(\smallcat q \otimes_R  \opp{S}).
\]
Thanks to Lemma \ref{lemma:hfp_tstruct_restrict} \eqref{enum:hfp_functoriality_texact}, there is an induced $S$-linear quasi-equivalence:
\[
\hfp(\hproj_R^\mathrm{rqr}(S, Q \hproj(\smallcat q)) \cong \hproj_R(\smallcat q \otimes_R  \opp{S})^\mathrm{hfp}.
\]
Now, we may combine this with Lemma \ref{lemma:hfp_tstruct_restrict} \eqref{enum:hfp_changeofrings_texact} and deduce $S$-linear quasi-equivalences:
\begin{equation}
\begin{split}
    \hproj_R^\mathrm{rqr}(S,Q\hproj(\smallcat q)^\mathrm{hfp}) & \cong \hfp(\hproj_R^\mathrm{rqr}(S, Q \hproj(\smallcat q)) \\
    & \cong \hproj_R(\smallcat q \otimes_R  \opp{S})^\mathrm{hfp}.
\end{split}
\end{equation}
\end{example}

\subsubsection{Coextension of scalars and duality} 
We now study the compatibility of the coextension of scalars with the duality of quasi-functors \eqref{eq:quasifunctors_duality}.
\begin{proposition} \label{prop:quasifunctors_duality_changeofrings_tstruct}
Let $\cat B$ be an $R$-linear dg-category with h-flat complexes of morphisms. The duality \eqref{eq:quasifunctors_duality} induces an $S$-linear quasi-equivalence
\begin{equation} \label{eq:duality_compatible_coextension}
    \dual = \dual_{S,\cat B} \colon \hproj_R^\mathrm{rqr}(S,\cat B) \xrightarrow{\sim} \opp{\hproj_R^\mathrm{rqr}(S,\opp{\cat B})}.
\end{equation}
If $\cat B$ is pretriangulated and has a t-structure, then the above \eqref{eq:duality_compatible_coextension} is also t-exact.
\end{proposition}
\begin{proof}
First of all, notice that $\opp{S}=S$ (strictly) by the assumption that $S$ is strictly commutative. Then, a tedious but straightforward check tells us that the duality dg-functor
\[
\dual \colon \compdg_R(S,\cat B) \to \opp{\compdg_R(\opp{S},\opp{\cat B})}
\]
is $S$-linear. This induces an $S$-linear quasi-functor
\[
\dual \colon \hproj_R^\mathrm{rqr}(S,\cat B) \to \opp{\hproj_R(S,\opp{\cat B})},
\]
so to conclude we have to show that it is quasi-fully faithful and its quasi-essential image is precisely $\opp{\hproj_R^\mathrm{rqr}(\opp{S},\opp{\cat B})}$. This claim does not depend on $S$-linearity, so we may restrict back to the underlying $R$-linear dg-categories and quasi-functors. The result is now a direct consequence of Proposition \ref{prop:duality_quasifunctors}. If moreover $\cat B$ is pretriangulated and has a t-structure, we deduce t-exactness from Proposition \ref{prop:duality_quasifunctor_texact}.
\end{proof}

\subsubsection{A tricky remark}
We conclude this part with a ``tricky'' observation. Let $\cat A$ be an $R$-linear dg-category. First, we recall from \S \ref{subsection:scalar_leftextension} that
\begin{equation}
    S \otimes_R \cat A
\end{equation}
is naturally $S$-linear. So, in principle one has the dg-category $\compdg_R(\opp{S} \otimes_R \cat A)$ of $R$-linear dg-bimodules (which is naturally $S$-linear, see Lemma \ref{lemma:bimodules_tensor_S_Slinear}) and also the $S$-linear dg-category $\compdg_S(\opp{S} \otimes_R \cat A)$ of $S$-linear dg-bimodules.
\begin{proposition} \label{prop:comparison_dgmodules_Slinear_Rlinear}
There is an $S$-linear equivalence of dg-categories:
\begin{equation} \label{eq:dgm_Slinear_versusRlinear}
    \compdg_S(\opp{S} \otimes_R \cat A) \xrightarrow{\sim} \compdg_R(\opp{S} \otimes_R \cat A).
\end{equation}

Moreover, assume that $\cat A$ is h-flat. Then, the above equivalence induces an $S$-linear equivalence
\begin{equation} \label{eq:hproj_Slinear_versusRlinear}
    \hproj_S(\opp{S} \otimes_R \cat A) \xrightarrow{\sim} \hproj_R(\opp{S} \otimes_R \cat A).
\end{equation}
Finally, if $\cat A$ has cohomology concentrated in nonpositive degrees, the above equivalence is t-exact with respect to the canonical t-structures.
\end{proposition}
\begin{proof} 
The dg-functor \eqref{eq:dgm_Slinear_versusRlinear} is defined as follows. Let $M \in \compdg_S(\opp{S} \otimes_R \cat A)$, which is an $S$-linear dg-functor
\[
M \colon S \otimes_R \opp{\cat A} \to \compdg(S)
\]
We map $M$ to the object $\widehat{M} \in \compdg_R(\opp{S} \otimes_R \cat A)$ defined by
\[
\widehat{M} \colon S \otimes_R \opp{\cat A} \to \compdg(S) \xrightarrow{\Res} \compdg(R).
\]
On the other hand, let $N \in \compdg_R(\opp{S} \otimes_R \cat A)$. We define $\widetilde{N} \in \compdg_S(\opp{S} \otimes_R \cat A)$ as
\[
\widetilde{N}_{\diamondsuit}^A = N_{\diamondsuit}^A,
\]
where we endow $N_{\diamondsuit}^A$ with the $S$-linear structure of Remark \ref{remark:coextension_dgmS}. The mapping $N \mapsto \widehat{N}$ then defines an ($S$-linear) inverse to \eqref{eq:dgm_Slinear_versusRlinear}. 

Now, it is clear by the definitions that both $M \mapsto \widehat{M}$ and $N \mapsto \widetilde{N}$ preserve acyclic objects. Being inverses to each other, they also preserve h-projective morphisms (this is an argument we have already used, for example in Proposition \ref{prop:hprojrqr_hproj}). Hence, they restrict and give the equivalence \eqref{eq:hproj_Slinear_versusRlinear} and its inverse.
\end{proof}

% Derived deformations

\section{Derived deformations} \label{section:derived_deformations}
\subsection{Recollections and notations} \label{subsection:recollections_before_deformations}
We fix a morphism of dg-rings $R \to S$, which we assume to be strictly commutative ($R=\opp{R}$ and $S=\opp{S}$) and strictly concentrated in nonpositive degrees. We are going to ease the notation a little bit, according to the following prescriptions:
\begin{itemize}
    \item For a given $R$-linear dg-category $\cat A$, we shall denote by $\dercompdg(\cat A)$ the \emph{dg-derived category of $\cat A$}, namely the dg-category $\hproj(\cat A)$. It is a dg-enhancement of the derived category $\dercomp(\cat A)$, in sense that 
    \[
    H^0(\dercompdg(\cat A)) \cong \dercomp(\cat A).
    \]
    If we want to stress $R$-linearity, we write $\dercompdg_{,R}(\cat A)$ (and analogously in the $S$-linear case).
    \item Given $R$-linear dg-categories $\cat A$ and $\cat B$, we denote by $\cat A \lotimes \cat B$ the \emph{derived tensor product} of $\cat A$ and $\cat B$, which is obtained by taking the ordinary tensor product $\cat A \otimes \cat B$, upon replacing $\cat A$ or $\cat B$ with a dg-category with h-flat complexes of morphisms if necessary. Moreover, we denote by  $\RHom(\cat A, \cat B)$ the ($R$-linear) \emph{dg-category of quasi-functors}, which is described as
    \[
    \RHom(\cat A, \cat B) = \hproj^{\mathrm{rqr}}(\cat A, \cat B) = \hproj^{\mathrm{rqr}}(\opp{\cat A} \lotimes \cat B),
    \]
    taking a resolution of $\cat A$ or $\cat B$ if necessary so that it is h-flat. If we want to stress $R$-linearity, we also write $\RHom_R(\cat A, \cat B)$ (and analogously in the $S$-linear case). The \emph{category of quasi-functors} is the homotopy category
    \[
    H^0(\RHom(\cat A, \cat B)),
    \]
    which is $H^0(R)$-linear.
    \item If $\cat A$ and $\cat B$ are ($R$-linear) dg-categories, a morphism $F \colon \cat A \to \cat B$ will always be a quasi-functor, namely an object of $\RHom(\cat A, \cat B)$. A quasi-functor $F$ induces a graded ($H^*(R)$-linear) functor  
    \[
    H^*(F) \colon H^*(\cat A) \to H^*(\cat B)
    \]
    between the graded homotopy categories, and a ($H^0(R)$-linear) functor
    \[
    H^0(F) \colon H^0(\cat A) \to H^0(\cat B).
    \]
    
    We say that two quasi-functors $F, G \colon \cat A \to \cat B$ are \emph{isomorphic} and write $F \cong G$ if they are isomorphic in $H^0(\RHom(\cat A, \cat B))$. A \emph{quasi-equivalence} is then a quasi-functor $F \colon \cat A \to \cat B$ such that there is a quasi-functor $G \colon \cat B \to \cat A$ with $GF \cong 1_{\cat A}$ and $FG \cong 1_{\cat B}$. This is the same as requiring that $H^*(F)$ is an equivalence. If $\cat A$ and $\cat B$ are pretriangulated, this is the same as requiring that $H^0(F)$ is an equivalence. We will say that $\cat A$ and $\cat B$ are \emph{quasi-equivalent}, and write
    \[
    \cat A \cong \cat B,
    \]
    if there is a quasi-equivalence $\cat A \xrightarrow{\sim} \cat B$.
\end{itemize}

With that, we may restate the results on change of rings of \S \ref{section:changeofrings}, as follows:
\begin{itemize}
    \item Let $\cat A$ be an $S$-linear dg-category. Then, there is an \emph{underlying $R$-linear dg-category} $\cat A_R$, obtained just restricting along $R \to S$. We will abuse notation and write again $\cat A$ for this underlying $R$-linear dg-category.
    \item Let $\cat B$ be an $R$-linear dg-category. Then, by Theorem \ref{thm:coext_univproperty}, there is an $S$-linear dg-category $\cat B_{(S)}$ together with a natural $R$-linear quasi-equivalence
    \begin{equation}
            \RHom_R(\cat A, \cat B) \cong \RHom_S(\cat A, \cat B_{(S)}).
    \end{equation}
    The dg-category $\cat B_{(S)}$ is uniquely determined up to $S$-linear quasi-equivalence, and it is called the \emph{coextension of scalars dg-category associated to $\cat B$} or also the \emph{dg-category of $S$-linear objects in $\cat B$}. A ``concrete model'' of $\cat B_{(S)}$ is given by the dg-category $\hproj^\mathrm{rqr}(S,\cat B)$, upon replacing $\cat B$ with a dg-category with h-flat complexes of morphisms if necessary.
    \item Let $\cat B$ be an $R$-linear dg-category. Then, by Proposition \ref{prop:scalar_leftextension}, there is an $S$-linear dg-category $S \lotimes_R \cat B$ together with a natural $R$-linear quasi-equivalence
    \begin{equation}
        \RHom_R(\cat B, \cat A) \cong \RHom_S(S \lotimes_R \cat B, \cat A).
    \end{equation}
    The dg-category $S \lotimes_R \cat B$ is uniquely determined up to $S$-linear quasi-equivalence, and it is called the \emph{extension of scalars dg-category associated to $\cat B$}.
    \item The construction $\cat B \mapsto \cat B_{(S)}$ is compatible with t-structures. We start with an $R$-linear pretriangulated dg-category $\cat B$ with a t-structure. Then, the $S$-linear dg-category $\cat B_{(S)}$ is pretriangulated and comes with a natural t-structure described in Theorem \ref{thm:tstruct_quasifunctors}. Any t-exact quasi-equivalence $\cat A \xrightarrow{\sim} \cat B$ between pretriangulated dg-categories with a t-structure induces (Corollary \ref{coroll:tstruct_changeofrings}) a t-exact quasi-equivalence
    \begin{equation} \label{eq:changeofrings_stability_texactequivalence}
        \cat A_{(S)} \xrightarrow{\sim} \cat B_{(S)}.
    \end{equation}
    
    Lemma \ref{lemma:hfp_tstruct_restrict} and Proposition \ref{prop:quasifunctors_duality_changeofrings_tstruct} give us natural t-exact quasi-equivalences, respectively:
    \begin{align}
        \hfp(\cat B_{(S)}) & \cong \hfp(\cat B)_{(S)}, \label{eq:quasieq_hfp_changeofrings}  \\
         \hfpbd(\cat B_{(S)}) & \cong \hfpbd(\cat B)_{(S)}, \label{eq:quasieq_hfpbd_changeofrings}  \\
        \opp{(\cat B_{(S)})} & \cong \opp{\cat B}_{(S)},  \label{eq:quasieq_opposite_changeofrings}
    \end{align}
    the first two ones holding if the t-structure on $\cat B$ is such that filtered colimits are exact in the heart $H^0(\cat B)^\heartsuit$ and $\hfp(\cat B)$ is pretriangulated (Lemma \ref{lemma:hfp_tstruct_restrict}).
    \item Let $\cat A \xrightarrow{\sim} \cat B$ be a t-exact quasi-equivalence between pretriangulated dg-categories with t-structures, and assume that filtered colimits are exact in both $H^0(\cat A)^\heartsuit$ and $H^0(\cat B)^\heartsuit$. Moreover, assume that both $\hfp(\cat A)$ and $\hfp(\cat B)$ are pretriangulated. Then (Lemma \ref{lemma:hfp_tstruct_restrict}), the above quasi-equivalence induces t-exact quasi-equivalences
    \begin{equation} \label{eq:quasieq_hfp}
    \begin{split}
        \hfp(\cat A) & \xrightarrow{\sim} \hfp(\cat B), \\
        \hfpbd(\cat A) & \xrightarrow{\sim} \hfpbd(\cat B).
    \end{split}
    \end{equation}
    \item Let $\smallcat a$ be a small $R$-linear dg-category with cohomology concentrated in nonpositive degrees. Then, the derived dg-category $\dercompdg_{,R}(\smallcat a)$ has a natural t-structure, and it is immediate to see that any quasi-equivalence $\smallcat a \xrightarrow{\sim} \smallcat b$ induces (by restriction) a t-exact quasi-equivalence 
    \begin{equation} \label{eq:quasieq_restriction_texact}
        \dercompdg_{,R}(\smallcat b) \xrightarrow{\sim} \dercompdg_{,R}(\smallcat a).
    \end{equation}
    
    Moreover, combining Lemma \ref{lemma:tstruct_hproj_coextension} and Proposition \ref{prop:comparison_dgmodules_Slinear_Rlinear}, we get natural t-exact $S$-linear quasi-equivalences:
    \begin{equation} \label{eq:quasieq_changeofrings_derdgcat}
        \dercompdg_{,R}(\smallcat a)_{(S)} \cong \dercompdg_{,R}(S \lotimes_R \smallcat a) \cong \dercompdg_{,S}(S \lotimes_R \smallcat a),
    \end{equation}
    recalling that $S \lotimes_R \smallcat a =\opp{S} \lotimes_R \smallcat a$ is naturally $S$-linear.
\end{itemize}
\begin{remark}
{A caveat on set-theoretic universes. In our applications, we will need to use the above \eqref{eq:quasieq_restriction_texact} and \eqref{eq:quasieq_changeofrings_derdgcat} with the small dg-categories $\smallcat a$ and $\smallcat b$ replaced by locally $\mathbb U$-small (hence $\mathbb V$-small; cf. \S \ref{subsec:sizes}) dg-categories. The technical results in \ref{subsection:changeofrings_tstructures} might need the introduction of yet another and bigger universe $\mathbb W$, for instance in order to take a cofibrant resolution such as $Q\hproj(\smallcat b)$. In any case, everything will be (in the worst case) locally $\mathbb V$-small at least up to quasi-equivalence.}
\end{remark}

We can now define \emph{derived deformations} rigorously, using the coextension of scalars.
\begin{definition} \label{def:derived_deformations}
Let $\cat B$ be an $S$-linear dg-category with a t-structure. A \emph{derived deformation} of $\cat B$ along $R \to S$ is an $R$-linear dg-category $\cat A$ with a t-structure together with an $R$-linear quasi-functor $\cat B \to \cat A$, which induces (recall Theorem \ref{thm:coext_univproperty}) a t-exact $S$-linear quasi-equivalence
\[
\cat B \xrightarrow{\sim} \cat A_{(S)}.
\]

We shall often abuse notation and say that \emph{$\cat A$ is a derived deformation of $\cat B$}, forgetting the quasi-functor $\cat B \to \cat A$.
\end{definition}
We also define \emph{left derived deformations} using the extension of scalars.
\begin{definition} \label{def:left_derived_deformation}
Let $\injcat B$ be an $S$-linear dg-category. A \emph{left derived deformation} of $\injcat B$ along $R \to S$ is a small $R$-linear dg-category $\injcat A$ together with an $R$-linear quasi-functor $\injcat A \to \injcat B$ which induces (recall Proposition \ref{prop:scalar_leftextension}) an $S$-linear quasi-equivalence
\begin{equation}
    S \lotimes_R \injcat A \xrightarrow{\sim} \injcat B.
\end{equation}
We shall often abuse notation and say that \emph{$\injcat A$ is a derived deformation of $\injcat B$}, forgetting the quasi-functor $\injcat A \to \injcat B$.
\end{definition}

\subsection{From abelian deformations to derived deformations} \label{subsec:fromabelian_toderived}
Let $R \to S$ be a surjective morphism of coherent commutative rings, with nilpotent kernel and such that $S$ is finitely presented as $R$-module. In particular, the kernel $\ker(R \to S)$ is also finitely presented as an $R$-ideal. The aim of this part is to prove that a deformation of abelian categories in the sense of \cite{lowen-vandenbergh-deformations-abelian}, under suitable assumptions, induces a derived deformation (Definition \ref{def:derived_deformations}) at the level of derived categories (taken with their natural t-structures).

We first recall some results from \cite{lowen-vandenbergh-deformations-abelian}. For a given $R$-linear category $\mathfrak A$, we denote by $\mathfrak A_{(S)}$ the abelian category of $S$-linear objects of $\mathfrak A$:
\begin{equation}
    \mathfrak A_{(S)} = \Fun_R(S,\mathfrak A).
\end{equation}
A \emph{deformation} of an $S$-linear abelian category $\mathfrak B$ is an $R$-linear abelian category $\mathfrak A$ together with an $R$-linear functor $\mathfrak B \to \mathfrak A$ which induces an $S$-linear equivalence $\mathfrak B \to \mathfrak A_{(S)}$. We can deduce from \cite[\S 6, \S 8.8]{lowen-vandenbergh-deformations-abelian} that if $\mathfrak A$ has enough injectives and $\mathfrak B \to \mathfrak A$ is a \emph{flat} deformation of $\mathfrak B$, then $\mathfrak A$ has enough injectives and moreover $\operatorname{Inj}(\mathfrak A)$ is a flat deformation of $\operatorname{Inj}(\mathfrak B)$, namely there is an $S$-linear equivalence
\begin{equation}
    \operatorname{Inj}(\mathfrak B) \cong S \otimes_R \operatorname{Inj}(\mathfrak A),
\end{equation}
and $\operatorname{Inj}(\mathfrak A)$ has flat hom-sets.
\subsubsection{The case of (left bounded) derived categories of Grothendieck categories}
We recall \cite[Theorem 1.3]{genovese-lowen-vdb-dginj} (see also Example \ref{example:hlc_hfp}): if $\cat B$ is a ($S$-linear) pretriangulated dg-category with a left bounded non-degenerate t-structure compatible with countable direct products (namely, the aisles $H^0(\cat B)_{\geq M}$ have countable products) and enough \emph{derived injectives}, then we have a ($S$-linear) t-exact quasi-equivalence\footnote{{We remark that every result in \cite{genovese-lowen-vdb-dginj} was given in the setting of dg-categories over a \emph{field}; however, the proofs carry over to dg-categories over arbitrary commutative (dg\nobreakdash-)rings, with few adjustments.}}
\begin{equation} \label{eq:dginj_reconstruction}
    \cat B \cong \opp{\hfp(\dercompdgmin(\opp{\operatorname{DGInj}(\cat B)}))}.
\end{equation}
Derived categories of Grothendieck abelian categories give typical examples of the above equivalence, indeed one can prove the following:
\begin{lemma} \label{lemma:dercat_dginj}
Let $\mathfrak B$ be any ($S$-linear) Grothendieck abelian category. Let $\dercatabdgplus(\mathfrak B)$ be a fixed dg-enhancement of the left bounded derived category $\dercatabplus(\mathfrak B)$. Then, the natural t-structure on $\dercatabdgplus(\mathfrak B)$ is left bounded, non-degenerate, compatible with countable products and has enough derived injectives. 

Moreover, the derived injectives coincide with the injective objects in $\mathfrak B$:
\begin{equation}
    \operatorname{DGInj}(\dercatabdgplus(\mathfrak B)) = \operatorname{Inj}(\mathfrak B)
\end{equation}
as a dg-subcategory of $\dercatabdgplus(\mathfrak B)$.
\end{lemma}
\begin{proof} 
We shall write $[X,Y]$ for the hom-set in the unbounded derived category $\mathfrak D(\mathfrak B)$. Since $\mathfrak B$ is a Grothendieck abelian category, the category $\mathfrak D(\mathfrak B)$ has both direct sums and products \cite[07D9]{stacks-project}. Now, let $\{ X_n : n \in \mathbb N \}$ be a family of objects in the right aisle $\dercatabplus(\mathfrak B)_{\geq M} = \mathfrak D(\mathfrak B)_{\geq M}$. We can form the product
\[
\prod_n X_n \in \mathfrak D(\mathfrak B).
\]
Then, for $Z \in \dercatabplus(\mathfrak B)_{< M}$, we have:
\begin{align*}
    [Z, \prod_n X_n] & \cong \prod_n [Z, X_n] \\
    & \cong 0,
\end{align*}
using that $\dercatabplus(\mathfrak B)_{< M}$ is left orthogonal to $\mathfrak D(\mathfrak B)_{\geq M}$, therefore $[Z,X_n] \cong 0$ for all $n$. We conclude that $\prod_n X_n$ lies in $\mathfrak D(\mathfrak B)_{\geq M}$, which in particular implies that it is a product of $\{ X_n : n \in \mathbb N \}$ in this right aisle. Thus, the t-structure on $\dercatabplus(\mathfrak B)$ is compatible with countable direct products; the other claimed properties (left boundedness and non-degeneracy) are clear.

For the second part of the proof, let $I \in \operatorname{Inj}(\mathfrak B)$. We have to prove that there is a natural isomorphism
\[
[X,I] \cong [H^0(X),I]
\]
for all $X$. First, we have $[X,I] \cong [\tau_{\geq 0} X, I]$ since $I$ lies in $\mathfrak A$ which is the heart of the t-structure of $\dercatabplus(\mathfrak B)$. Then, consider the (functorial) distinguished triangle
\[
H^0(X) \to \tau_{\geq 0} X \to \tau_{\geq 1} X.
\]
Composing with $H^0(X) \to \tau_{\geq 0} X$, we get a morphism
\[
[\tau_{\geq 0} X, I] \to [H^0(X), I]
\]
To check that this is an isomorphism, it is enough to prove that
\[
[\tau_{\geq 1} X,I]=0 \quad \text{and} \quad [(\tau_{\geq 1}X)[-1], I]=0.
\]
This follows from the fact that $I$ being injective implies that
\[
[Y,I] = \mathfrak K(\mathfrak B)(Y,I)
\]
for $Y$ concentrated in nonnegative degrees, where $\mathfrak K(\mathfrak B)$ is the homotopy category of $\mathfrak B$.
\end{proof}
We are now able to prove the following result, which can be summarized by the slogan \emph{``flat deformations of Grothendieck abelian categories induce derived deformations of the corresponding left bounded derived categories''}.
\begin{theorem} \label{thm:abdeform_derdeform}
Let $\mathfrak B$ be an $S$-linear Grothendieck abelian category, and let $\mathfrak A$ be a $R$-linear flat deformation of $\mathfrak B$ (which is again Grothendieck, cf. \cite[Theorem 6.29]{lowen-vandenbergh-deformations-abelian} ). Then, $\dercatabdgplus(\mathfrak A)$ is a $R$-linear derived deformation of the $S$-linear dg-category $\dercatabdgplus(\mathfrak B)$.
\end{theorem}
\begin{proof}
From \cite{lowen-vandenbergh-deformations-abelian} we know that $\injcat I=\operatorname{Inj}(\mathfrak A)$ is a (left) deformation of $\injcat J=\operatorname{Inj}(\mathfrak B)$, in particular we have an $S$-linear equivalence
\[
\injcat J \cong S \otimes_R \injcat I.
\]
$\injcat I$ has flat hom-modules, so it is h-flat if we view it as a dg-category concentrated in degree $0$. Hence, we have that
\[
\injcat J \cong S \lotimes_R \injcat I 
\]
as $S$-linear dg-categories. Taking opposite categories in the above discussion and recalling that $S=\opp{S}$, we deduce that 
\[
\opp{\injcat J} \cong S \lotimes_R \opp{\injcat I}
\]
as $S$-linear dg-categories. Now, invoking the results recollected in \S \ref{subsection:recollections_before_deformations} and the above discussion, we get a chain of $S$-linear t-exact quasi-equivalences:
\begin{align*}
    \dercatabdgplus(\mathfrak B) &\cong \opp{\hfp({\dercompdgmin}_{,S}(\opp{\injcat J}))} \qquad \text{(cf. \eqref{eq:dginj_reconstruction})}  \\
    &\cong \opp{\hfp({\dercompdgmin}_{,S}(S \lotimes_R \opp{\injcat I}))} \qquad \text{(cf. \eqref{eq:quasieq_restriction_texact})} \\
    &\cong \opp{\hfp({\dercompdgmin}_{,R}(S \lotimes_R \opp{\injcat I}))} \qquad \text{(cf. \eqref{eq:quasieq_changeofrings_derdgcat})}  \\
    & \cong  \opp{\hfp({\dercompdgmin}_{,R}(\opp{\injcat I})_{(S)})} \qquad \text{(cf. \eqref{eq:quasieq_changeofrings_derdgcat})} \\
    &\cong \opp{\hfp({\dercompdgmin}_{,R}(\opp{\injcat I}))}_{(S)} \qquad \text{(cf. \eqref{eq:quasieq_hfp_changeofrings} and \eqref{eq:quasieq_opposite_changeofrings})} \\
    &\cong \dercatabdgplus(\mathfrak A)_{(S)}. \qquad \text{(cf. \eqref{eq:changeofrings_stability_texactequivalence} and \eqref{eq:dginj_reconstruction})}
\end{align*}
We also remark that we used \eqref{eq:quasieq_hfp}
throughout. We conclude that $\dercatabdgplus(\mathfrak A)$ is indeed a derived deformation of $\dercatabdgplus(\mathfrak B)$.
\end{proof}
\begin{remark} \label{remark:abdeform_derdeform_enoughinj}
We can slightly change the above argument and show that Theorem \ref{thm:abdeform_derdeform} holds under the more general assumption that $\mathfrak B$ is an $S$-linear abelian category with enough injectives.

The idea is prove the t-exact quasi-equivalence
\[
\dercatabdgplus(\mathfrak B) \cong \opp{\hfp({\dercompdgmin}_{,S}(\opp{\injcat J}))},
\]
and the analogous for $\mathfrak A$, without invoking \eqref{eq:dginj_reconstruction}. This can be done by recalling from \cite[Proposition 6.25]{lowen-vandenbergh-deformations-abelian} that $\mathfrak B$ is equivalent to $\opp{\operatorname{mod}(\opp{\injcat J})}$ and then by directly showing that
\[
\dercatabdgplus(\opp{\operatorname{mod}(\opp{\injcat J})}) \cong \opp{\hfp({\dercompdgmin}_{,S}(\opp{\injcat J}))},
\]
and the analogous result for $\mathfrak A$.
\end{remark}

\subsubsection{The case of bounded derived categories of (small) abelian categories}
Let $\mathfrak B$ be a (small, $S$-linear) abelian category. It is well-known that the category of ind-objects $\Ind(\mathfrak B)$ is a Grothendieck abelian category. Moreover, we have an equivalence:
\begin{equation} \label{eq:fpind_id}
\fp(\Ind(\mathfrak B)) \cong \mathfrak B.
\end{equation}
We denote by $\dercatabdgbd(\mathfrak B)$ be (a fixed dg-enhancement of) the bounded derived category of $\mathfrak B$. The objects of $\dercatabdgbd(\mathfrak B)$ can be taken as the complexes $X$ of objects in $\mathfrak B$ whose cohomologies $H^k(X)$ are $0$ except for a finite number of indices $k$ (cf. \cite[05RR]{stacks-project}). We can prove a derived counterpart of the equivalence \eqref{eq:fpind_id}:
\begin{lemma} \label{lemma:hfpbd_derind}
Let $\mathfrak B$ be a (small) abelian category. We endow the pretriangulated dg-cartegory $\dercatabdgbd(\mathfrak B)$ with the natural t-structure. Moreover, let $\dercatabdgplus(\Ind(\mathfrak B))$ be (a fixed dg-enhancement of) the left bounded derived category of the Grothendieck category $\Ind(\mathfrak B)$. The exact fully faithful functor
\[
\mathfrak B \hookrightarrow \Ind(\mathfrak B)
\]
yields a t-exact fully faithful quasi-functor
\[
\dercatabdgbd(\mathfrak B) \hookrightarrow \dercatabdgplus(\Ind(\mathfrak B)),
\]
which induces a t-exact quasi-equivalence
\begin{equation}
    \dercatabdgbd(\mathfrak B) \xrightarrow{\sim} \hfpbd(\dercatabdgplus(\Ind(\mathfrak B))).
\end{equation}
\end{lemma}
\begin{proof}
If we identify $\mathfrak B$ as a full (exact) abelian subcategory of $\Ind(\mathfrak B)$, we may identify $\dercatabdgbd(\mathfrak B)$ as the full dg-subcategory of $\dercatabdgplus(\Ind(\mathfrak B))$ of complexes of objects of $\mathfrak B$ also having almost all zero cohomologies. The natural t-structure on $\dercatabdgbd(\mathfrak B)$ is then obtained from the natural t-structure on $\dercatabdgplus(\Ind(\mathfrak B))$ by restriction.

Now, if $X \in \dercatabdgbd(\mathfrak B)$, it is clear that it lies in $\hfpbd(\dercatabdgplus(\Ind(\mathfrak B)))$. Indeed, its cohomologies are almost all zero, and they lie in $\mathfrak B = \fp(\Ind(\mathfrak B))$. Conversely, let $Y \in \hfpbd(\dercatabdgplus(\Ind(\mathfrak B)))$. This means that $H^k(Y) \in \mathfrak B$ for all $k$, and $H^k(Y)=0$ for almost all $k$. By definition, this means that $Y \in \dercatabdgbd(\mathfrak B)$.
\end{proof}

We are now able to prove the following result, which can be summarized by the slogan \emph{``flat deformations of (small) abelian categories induce derived deformations of their bounded derived categories''}.
\begin{theorem} \label{thm:smallabdeform_derdeform}
Let $\mathfrak B$ be an $S$-linear (small) abelian category, and let $\mathfrak A$ be a $R$-linear flat deformation of $\mathfrak B$. Then, $\dercatabdgbd(\mathfrak A)$ is a $R$-linear derived deformation of the $S$-linear dg-category $\dercatabdgbd(\mathfrak B)$.
\end{theorem}
\begin{proof}
From \cite[\S 5]{lowen-vandenbergh-deformations-abelian} we know that the given flat deformation $\mathfrak A$ of $\mathfrak B$ induces a flat deformation $\Ind(\mathfrak A)$ of $\Ind(\mathfrak B)$. We can then apply Theorem \ref{thm:abdeform_derdeform} and conclude that $\dercatabdgplus(\Ind(\mathfrak A))$ is a derived deformation of $\dercatabdgplus(\Ind(\mathfrak B))$. Namely, there is a t-exact quasi-equivalence:
\[
\dercatabdgplus(\Ind(\mathfrak B)) \cong \dercatabdgplus(\Ind(\mathfrak A))_{(S)}.
\]
We apply $\hfpbd(-)$ and we obtain t-exact quasi-equivalences:
\begin{align*}
    \hfpbd(\dercatabdgplus(\Ind(\mathfrak B))) & \cong \hfpbd(\dercatabdgplus(\Ind(\mathfrak A))_{(S)}) \\
    & \cong \hfpbd(\dercatabdgplus(\Ind(\mathfrak A)))_{(S)} \qquad \text{(cf. \eqref{eq:quasieq_hfpbd_changeofrings})}
\end{align*}
Applying the above Lemma \ref{lemma:hfpbd_derind}, we obtain a t-exact quasi-equivalence
\[
\dercatabdgbd(\mathfrak B) \cong \dercatabdgbd(\mathfrak A)_{(S)}
\]
as we wanted.
\end{proof}

\subsection{Deformations of derived injectives} \label{subsec:deformations_dginj}
Thanks to the results in \cite{genovese-lowen-vdb-dginj} and in analogy with \cite{lowen-vandenbergh-deformations-abelian}, it is natural to expect that derived deformations of dg-categories with suitable t-structures can be understood in terms of appropriately left defined deformations (Definition \ref{def:left_derived_deformation}) of their derived injectives. We will indeed show (Proposition \ref{prop:properties_lifting_leftdeformation} and Theorem \ref{thm:dginj_derived_deformation}) that ``dg-categories of derived injectives'' are stable under (left) derived deformations, and that such deformations induce derived deformations of the associated t-structures.
\subsubsection{Setup and preliminaries} \label{subsubsection:dginj_deform_setup}
We fix a morphism $\theta \colon R \to S$ of commutative dg-rings strictly concentrated in nonpositive degrees. We make the following assumptions:
\begin{enumerate}[series=propertieslist]
    \item \label{item:dginj_deform_setup_strictsurj} The morphism $\theta \colon R \to S$ is strictly surjective. 
    \item \label{item:dginj_deform_setup_homotopicallycoherent} Both $R$ and $S$ are \emph{homotopically coherent} (see Definition \ref{def:homotopically_fp}), namely, $H^0(R)$ and $H^0(S)$ are coherent commutative rings, and
    \begin{align*}
    H^i(R) \in \operatorname{mod}(H^0(R)), \\
    H^i(S) \in \operatorname{mod}(H^0(S)),
    \end{align*}
    for all $i \in \mathbb Z$.
    \item \label{item:dginj_deform_setup_hfp} $S$ is homotopically finitely presented as an $R$-dg-module (see Definition \ref{def:homotopically_fp}), namely $H^i(S) \in \operatorname{mod}(H^0(R))$ for all $i \in \mathbb Z$.
    \item \label{item:dginj_deform_setup_cohomologicalnilpotency} The dg-ideal $I=\ker \theta$ is cohomologically nilpotent, namely $H^*(I^n)=0$ for some integer $n> 0$.
    \item \label{item:dginj_deform_setup_powers_hfp} The dg-ideals $I^k$ are homotopically finitely presented as $R$-dg-modules for all $k=1,\ldots, n-1$.
\end{enumerate}
\begin{remark}
{In principle, we would like purely cohomological conditions on $R \to S$ that, upon replacement up to quasi-isomorphism, give the \eqref{item:dginj_deform_setup_strictsurj}-\eqref{item:dginj_deform_setup_powers_hfp} above. How to achieve this is not yet obvious to us. Still, these conditions hold in our key examples (see Example \ref{example:deformations_dgrings_mainexample} below). More details about assumption \eqref{item:dginj_deform_setup_powers_hfp} are contained in the following Remark \ref{remark:setup_explanation}.}
\end{remark}
\begin{remark} \label{remark:setup_explanation}
We can deduce from assumptions \eqref{item:dginj_deform_setup_strictsurj}-\eqref{item:dginj_deform_setup_hfp} that $I=\ker \theta$ is a homotopically finitely presented $R$-dg-module. Indeed, we have a short exact sequence of $R$-dg-modules
\[
0 \to I \to R \to S \to 0,
\]
which induces the usual long exact sequence in cohomology. Since $R$ is homotopically coherent and $S$ is homotopically finitely presented as an $R$-dg-module, we may use a similar argument as in \cite[Lemma 5.10]{genovese-lowen-vdb-dginj} and conclude that $I$ is homotopically finitely presented as an $R$-dg-module.

Unfortunately, it's not evident how to deduce that also the powers $I^k$ are homotopically finitely presented (for $k=1,\ldots,n-1$). However, this property will still hold in the examples we consider (see Example \ref{example:deformations_dgrings_mainexample} below), so we just assume it by hypothesis.
\end{remark}
\begin{example} \label{example:deformations_dgrings_mainexample}
{Is this ok? Maybe more details?} Let $\basering k$ be a field. We take $S=\basering k$, and we set
\[
R = \basering k[\epsilon] / (\epsilon^n),
\]
for some $n\geq 2$, where $\epsilon$ has degree $2-k$ for some $k \geq 2$. $R$ is thus a dg-ring with trivial differentials, and it is easy to see that the natural projection
\begin{align*}
\basering k[\epsilon] / (\epsilon^n) & \to \basering k, \\
a + \epsilon b & \mapsto a
\end{align*}
satisfies the assumptions of \S \ref{subsubsection:dginj_deform_setup}.
\end{example}

\subsubsection{Statement of the main results} \label{subsubsection:dginj_deform_results}
We recall from \cite[\S 5.2]{genovese-lowen-vdb-dginj} that a dg-category $\injcat A$ is \emph{(left) homotopically locally coherent (hlc)} if it has the following properties:
\begin{itemize}
    \item It is cohomologically concentrated in nonpositive degrees.
    \item $H^0(\injcat A)$ is additive.
    \item $H^0(\injcat A)$ is (left) coherent. Assuming additivity, this is the same as requiring that any morphism $f \colon A \to B$ in $H^0(\injcat A)$ has a weak cokernel, namely a morphism $g \colon B \to C$ in $H^0(\injcat A)$ such that the sequence
    \[
    \injcat A(C,-) \xrightarrow{g^*} \injcat A(B,-) \xrightarrow{f^*} \injcat A(A,-)
    \]
    is exact.
    \item For all $A \in \injcat A$, the representable left dg-module $\injcat A(A,-)$ is homotopically finitely presented: $H^k(\injcat A(A,-)) \in \operatorname{mod}(\opp{H^0(\injcat A)})$ for all $k \in \mathbb Z$.
\end{itemize}
We know from \cite{genovese-lowen-vdb-dginj} that $\injcat A$ is a ``dg-category of derived injectives'' if it is hlc and $H^0(\injcat A)$ is Karoubian. Our goal is to show that being a dg-category of derived injectives is preserved under left derived deformations (cf. Definition \ref{def:left_derived_deformation}). More precisely:
\begin{proposition} \label{prop:properties_lifting_leftdeformation}
We assume the setup of \S \ref{subsubsection:dginj_deform_setup}. Let $\injcat J$ be an $S$-linear hlc dg-category such that $H^0(\injcat J)$ is Karoubian, and let $\injcat I$ be a left derived deformation of $\injcat J$ along $R \to S$. Then, $\injcat I$ is hlc and $H^0(\injcat I)$ is Karoubian.
\end{proposition}
We pospone the proof of Proposition \ref{prop:properties_lifting_leftdeformation} to \S \ref{subsubsection:dginj_deform_order2reduction} and  \S \ref{subsubsection:dginj_deform_liftingproperties_proof}. Now, we use it to prove the main result of this part, which can be summarized by the following slogan: \emph{left derived deformations of dg-categories of derived injectives induce derived deformations of the t-structures associated to them.}
\begin{theorem} \label{thm:dginj_derived_deformation}
Let $\injcat J$ be an $S$-linear dg-category which is homotopically locally coherent and such that $H^0(\injcat J)$ is Karoubian, and let $\injcat I \to \injcat J$ be a left derived deformation of $\injcat J$. We set:
\begin{align*}
    \cat B &= \opp{\hfp({\dercompdgmin}_{,S}(\opp{\injcat J}))}, \\
    \cat A &= \opp{\hfp({\dercompdgmin}_{,R}(\opp{\injcat I}))},
\end{align*}
endowed with the t-structures of \cite[Theorem 1.2]{genovese-lowen-vdb-dginj}. Then, $\cat A$ is a derived deformation of $\cat B$.
\end{theorem}
\begin{proof}
By the above Proposition \ref{prop:properties_lifting_leftdeformation}, $\injcat I$ is an hlc dg-category such that $H^0(\injcat I)$ is Karoubian. Since $\injcat I \to \injcat J$ describes $\injcat I$ as a left derived deformation of $\injcat J$, the opposite quasi-functor $\opp{\injcat I} \to \opp{\injcat J}$ describes $\opp{\injcat J}$ as a left derived deformation of $\opp{\injcat I}$. Indeed, let $Q(\injcat I) \to \injcat I$ be an $R$-linear h-flat resolution of $\injcat I$; clearly, $\opp{Q(\injcat I)} \to \opp{\injcat I}$ is an h-flat resolution of $\opp{\injcat I}$. The $S$-linear quasi-equivalence
\[
S \otimes_R Q(\injcat I) \cong S \lotimes_R \injcat I \cong \injcat J
\]
induces $S$-linear quasi-equivalences
\begin{align*}
    S \lotimes_R \opp{\injcat I} & \cong S \otimes_R \opp{Q(\injcat I)} \\
    & \cong \opp{(S \otimes_R Q(\injcat I))} \\
    & \cong (S \lotimes_R \injcat I)^\mathrm{op} \\
    & \cong \opp{\injcat J},
\end{align*}
obtained by taking opposites (see also \S \ref{sec:duality}).

Now, we have a chain of t-exact quasi-equivalences:
\begin{align*}
       \cat B &\cong \opp{\hfp({\dercompdgmin}_{,S}(\opp{\injcat J}))} \qquad \text{(cf. \eqref{eq:dginj_reconstruction})}  \\
    &\cong \opp{\hfp({\dercompdgmin}_{,S}(S \lotimes_R \opp{\injcat I}))} \qquad \text{(cf. \eqref{eq:quasieq_restriction_texact})} \\
    &\cong \opp{\hfp({\dercompdgmin}_{,R}(S \lotimes_R \opp{\injcat I}))} \qquad \text{(cf. \eqref{eq:quasieq_changeofrings_derdgcat})}  \\
    & \cong  \opp{\hfp({\dercompdgmin}_{,R}(\opp{\injcat I})_{(S)})} \qquad \text{(cf. \eqref{eq:quasieq_changeofrings_derdgcat})} \\
    &\cong \opp{\hfp({\dercompdgmin}_{,R}(\opp{\injcat I}))}_{(S)} \qquad \text{(cf. \eqref{eq:quasieq_hfp_changeofrings} and \eqref{eq:quasieq_opposite_changeofrings})} \\
    &\cong \cat A_{(S)}. \qquad \text{(cf. \eqref{eq:changeofrings_stability_texactequivalence} and \eqref{eq:dginj_reconstruction})}
\end{align*}
We also remark that we used \eqref{eq:quasieq_hfp}
throughout. We conclude that $\cat A$ is indeed a derived deformation of $\cat B$.
\end{proof}
\subsubsection{Reduction to nilpotency order $2$} \label{subsubsection:dginj_deform_order2reduction}
The proof of Proposition \ref{prop:properties_lifting_leftdeformation} is our last technical endeavor. First, we explain how to replace the (cohomological) nilpotency assumption \eqref{item:dginj_deform_setup_cohomologicalnilpotency} in \S \ref{subsubsection:dginj_deform_setup} with the following stronger one:
\begin{primenumerate}[resume=propertieslist,start=\getrefnumber{item:dginj_deform_setup_cohomologicalnilpotency}]
\item \label{item:dginj_deform_setup_order2strictnilpotency} The dg-ideal $I=\ker(R \xrightarrow{\theta} S)$ is strictly nilpotent of order $2$, namely $I^2=0$.
\end{primenumerate}
To do so, we first observe that $\theta \colon R \to S$ can be factored as the following chain of morphisms of commutative dg-rings (all concentrated in nonpositive degrees):
\begin{equation} \label{eq:dgrings_morphism_factorization}
    R \xrightarrow{\theta_n} R/I^n \xrightarrow{\theta_{n, n-1}} R/I^{n-1} \to \cdots \xrightarrow{\theta_{2,1}} R/I = S.
\end{equation}
Moreover, $\theta_n \colon R \xrightarrow{\approx} R/I^n$ is a quasi-isomorphism, since the dg-ideal $I$ satisfies $H^*(I^n)=0$ by assumption \eqref{item:dginj_deform_setup_cohomologicalnilpotency} in \S \ref{subsubsection:dginj_deform_setup}.
\begin{lemma} \label{lemma:dgring_morphism_factorization_properties}
For $k=1,\ldots,n-1$, the morphism $\theta_{k+1,k} \colon R/I^{k+1} \to R/I^k$ satisfies the assumptions \eqref{item:dginj_deform_setup_strictsurj}-\eqref{item:dginj_deform_setup_powers_hfp} in \S \ref{subsubsection:dginj_deform_setup}, where \eqref{item:dginj_deform_setup_cohomologicalnilpotency} is replaced by the stronger \textup{\ref{item:dginj_deform_setup_order2strictnilpotency}} above, namely, $\ker(\theta_{k+1,k})^2=0$.
\end{lemma}
\begin{proof}
$\ker(\theta_{k+1,k})$ is isomorphic to $I^k/I^{k+1}$, which is clearly nilpotent of order $2$, so \textup{\ref{item:dginj_deform_setup_order2strictnilpotency}} holds. The morphism $\theta_{k+1,k}$ is clearly surjective, so \eqref{item:dginj_deform_setup_strictsurj} also holds. Assumptions \eqref{item:dginj_deform_setup_homotopicallycoherent} and \eqref{item:dginj_deform_setup_hfp} follow from Sublemmas \ref{sublemma:quotient_H0_coherent}, \ref{sublemma:factors_hfp} and \ref{quotient_homotopicallycoherent}  below. Finally, \eqref{item:dginj_deform_setup_powers_hfp} reduces to the claim that $\ker(\theta_{k+1,k})$ is homotopically finitely presented, and this follows from a similar argument as in Remark \ref{remark:setup_explanation}.
\end{proof}
\begin{sublemma} \label{sublemma:finitelypresented_lift_surjective}
Let $g \colon A \to B$ be a surjective morphism of commutative rings, and let $M$ be a $B$-module. If $M$ is finitely presented as an $A$-module, it is finitely presented as a $B$-module.
\end{sublemma}
\begin{proof}
Since $g$ is surjective, we easily see that the restriction functor
\[
\Res_g \colon \Mod(B) \to \Mod(A)
\]
is fully faithful. Now, let $M \in \Mod(B)$ be such that $\Res_g(M)$ is a finitely presented $A$-module. We show that $M$ is finitely presented by checking that for any directed system $\{X_i : i \in I \}$ in $\Mod(B)$, the natural map
\[
\varinjlim_i \Hom_B(M,X_i) \to \Hom_B(M, \varinjlim_i X_i)
\]
is an isomorphism. Since $\Res_g$ is fully faithful and commutes with colimits, we have a commutative diagram:
\[
\begin{tikzcd}
{\varinjlim_i \Hom_B(M,X_i)} \arrow[d, "\sim"'] \arrow[r] & {\Hom_B(M,\varinjlim_i X_i)} \arrow[d, "\sim"] \\
{\varinjlim_i \Hom_A(\Res_g(M),\Res_g(X_i))} \arrow[r]    & {\Hom_A(\Res_g(M),\varinjlim_i \Res_g(X_i)).}  
\end{tikzcd}
\]
By hypothesis, the lower horizontal arrow is an isomorphism. We conclude that the upper horizontal arrow is an isomorphism, as claimed.
\end{proof}
\begin{sublemma} \label{sublemma:quotient_H0_coherent}
The commutative ring $H^0(R/I^k)$ is coherent for all $k=1,\ldots,n-1$.
\end{sublemma}
\begin{proof}
Consider the following short exact sequence:
\[
0 \to I^k \hookrightarrow R \xrightarrow{\theta_k} R/I^k \to 0
\]
and the induced cohomological long exact sequence:
\[
H^0(I^k) \to H^0(R) \to H^0(R/I^k) \to 0.
\]
We deduce the following isomorphism of commutative rings:
\[
H^0(R/I^k) \cong \frac{H^0(R)}{\operatorname{Im}(H^0(I^k) \to H^0(R))}.
\]
Since $H^0(I^k)$ is finitely generated by assumption \eqref{item:dginj_deform_setup_powers_hfp} in \S \ref{subsubsection:dginj_deform_setup}, we have that $\operatorname{Im}(H^0(I^k) \to H^0(R))$ is also a finitely generated $H^0(R)$-ideal. 

Now, $H^0(R/I^k)$ is a coherent ring, being the quotient of the coherent ring $H^0(R)$ by a finitely generated ideal. Indeed, let $\overline{J}$ be a finitely generated ideal of $H^0(R/I^k)$. Since $H^0(\theta_k) \colon H^0(R) \to H^0(R/I^k)$ is surjective, $\overline{J}$ is also a finitely generated as an $H^0(R)$-module. Moreover, it is (up to isomorphism of $H^0(R/I^k)$-modules) of the form
\[
\frac{J}{\operatorname{Im}(H^0(I^k) \to H^0(R))},
\]
for some ideal $J$ of $H^0(R)$. This ideal $J$ is finitely generated, being an extension of $\operatorname{Im}(H^0(I^k) \to H^0(R))$ and $\overline{J}$, which are finitely generated $H^0(R)$-modules. Since $H^0(R)$ is coherent, $J$ is also finitely presented. We conclude that
\[
\overline{J} \cong \frac{J}{\operatorname{Im}(H^0(I^k) \to H^0(R))}
\]
is a finitely presented $H^0(R)$-module, being a quotient of a finitely presented module by a finited generated module, cf. \cite[0517, Lemma 10.5.3]{stacks-project}. Applying the above Sublemma \ref{sublemma:finitelypresented_lift_surjective}, we conclude that $\overline{J}$ is finitely presented as a $H^0(R/I^k)$-module, as we wanted.
\end{proof}
\begin{sublemma} \label{sublemma:factors_hfp}
For all $k=1,\ldots,n-1$, the morphism $\theta_{k+1,k} \colon R/I^{k+1} \to R/I^k$ exhibits $R/I^k$ as a homotopically finitely presented $R/I^{k+1}$-dg-module.
\end{sublemma}
\begin{proof}
Let $k \in \{1,\ldots,n-1\}$. We consider the short exact sequence of $R$-dg-modules
\[
0 \to I^k \to R \xrightarrow{\theta_k} R/I^k \to 0.
\]
Since $R$ and $I^k$ are homotopically finitely presented by assumption, we may argue as in \cite[Lemma 5.10]{genovese-lowen-vdb-dginj} and deduce that $R/I^k$ is homotopically finitely presented as an $R$-dg-module. This means that $H^i(R/I^k)$ is a finitely presented $H^0(R)$-module for all $i \in \mathbb Z$. Now, the morphism $H^0(\theta_{k+1}) \colon H^0(R) \to H^0(R/I^{k+1})$ is surjective, so we may apply the above Sublemma \ref{sublemma:finitelypresented_lift_surjective} and conclude that $H^i(R/I^k)$ is finitely presented as an $H^0(R/I^{k+1})$-module for all $i \in \mathbb Z$, which is what we wanted.
\end{proof}
\begin{sublemma} \label{quotient_homotopicallycoherent}
The dg-ring $R/I^k$ is homotopically coherent for all $k=1,\ldots,n-1$.
\end{sublemma}
\begin{proof}
Thanks to the above Sublemma \ref{sublemma:quotient_H0_coherent}, we only need to prove that
\[
H^i(R/I^k) \in \operatorname{mod}(H^0(R/I^k))
\]
for all $i \in \mathbb Z$. We know from Sublemma \ref{sublemma:factors_hfp} that $H^i(R/I^k)$ is a finitely presented $H^0(R/I^{k+1})$-module. Since $H^0(\theta_{k+1,k}) \colon H^0(R/I^{k+1}) \to H^0(R/I^k)$ is surjective, we may apply the above Sublemma \ref{sublemma:finitelypresented_lift_surjective} and conclude that $H^i(R/I^k)$ is indeed a finitely presented $H^0(R/I^k)$-module, for all $i \in \mathbb Z$.
\end{proof}

With Lemma \ref{lemma:dgring_morphism_factorization_properties} in mind, the idea is now to ``factor'' any left derived deformation $\injcat I \to \injcat J$ along $R \to S$ into a finite sequence of left derived deformations
\[
\injcat I_{k+1} \to \injcat I_k
\]
along $R/I^{k+1} \to R/I^k$. In order to do so, we have to check that the derived extension of scalars (cf. \S \ref{subsection:scalar_leftextension}) is transitive.
\begin{lemma} \label{lemma:leftextension_derived_transitive}
Let $R_1 \xrightarrow{f} R_2 \xrightarrow{g} R_3$ be morphisms of commutative dg-rings concentrated in nonpositive degrees, and let $\cat A$ be an $R_1$-linear dg-category. Then, there is a natural $R_3$-linear quasi-equivalence:
\begin{equation}
    R_3 \lotimes_{R_1} \cat A \cong R_3 \lotimes_{R_2} (R_2 \lotimes_{R_1} \cat A).
\end{equation}
\end{lemma}
\begin{proof}
The strategy is to prove a corresponding result in the ``non-derived'' framework, and then use h-flat resolutions. Without less of generality, we assume that $\cat A$ is an h-flat $R_1$-linear dg-category, so that we may identify
\[
R_2 \lotimes_{R_1} \cat A  \cong R_2 \otimes_{R_1} \cat A.
\]
Let $\cat B$ any $R_2$-linear dg-category. We define the following dg-functors:
\begin{equation} \label{eq:leftextension_transitivity_nonderived}
    \begin{split}
        \cat B \otimes_{R_1} \cat A & \leftrightarrows \cat B \otimes_{R_2} (R_2 \otimes_{R_1} \cat A), \\
        g \otimes_{R_1} f & \mapsto g \otimes_{R_2} (1_{R_2} \otimes_{R_1} f), \\
        g r_2 \otimes_{R_1} f & \mapsfrom g \otimes_{R_2} (r_2 \otimes_{R_1} f),
    \end{split}
\end{equation}
and by the identity on objects. It is easy to see that they are $R_2$-linear isomorphisms, inverse to each other. Moreover, taking $\cat B =R_3$, we obtain a natural $R_3$-linear isomorphism:
\begin{equation} \label{eq:leftextension_transitivity_nonderived_R3linear}
R_3 \otimes_{R_1} \cat A  \leftrightarrows R_3 \otimes_{R_2} (R_2 \otimes_{R_1} \cat A).
\end{equation}

Next, we observe that $R_2 \otimes_{R_1} \cat A$ is an h-flat $R_2$-linear dg-category. Indeed, let $V$ be an acyclic $R_2$-dg-module. Then, the same argument as \eqref{eq:leftextension_transitivity_nonderived} yields an $R_1$-linear isomorphism:
\[
V \otimes_{R_2} (R_2 \otimes_{R_1} \cat A(A,B)) \cong V \otimes_{R_1} \cat A(A,B),
\]
for all objects $A,B \in \cat A$. Since $\cat A(A,B)$ is an h-flat $R_1$-dg-module, we conclude that $V \otimes_{R_1} \cat A(A,B)$ is acyclic. Hence, $V \otimes_{R_2} (R_2 \otimes_{R_1} \cat A(A,B))$ is also an acyclic $R_2$-dg-module (acylicity is independent of the choice of base).

Finally, since $\cat A$ and $R_2 \otimes_{R_1} \cat A$ are h-flat, we may identify
\begin{align*}
R_3 \lotimes_{R_2} (R_2 \lotimes_{R_1} \cat A) & \cong R_3 \otimes_{R_2} (R_2 \otimes_{R_1} \cat A), \\
R_3 \lotimes_{R_1} \cat A & \cong R_3 \otimes_{R_1} \cat A,
\end{align*}
as $R_3$-linear dg-categories. Using \eqref{eq:leftextension_transitivity_nonderived_R3linear}, we then immediately conclude.
\end{proof}
Finally, we are able to prove the desired reduction to nilpotency order $2$.
\begin{lemma} \label{lemma:nilpotency_order2_reduction}
Assume that Proposition \ref{prop:properties_lifting_leftdeformation} holds for all morphisms of dg-rings $R_1 \to R_2$ concentrated in nonpositive degrees, satisfying the assumptions in \S \ref{subsubsection:dginj_deform_setup} with \eqref{item:dginj_deform_setup_cohomologicalnilpotency} replaced by the stronger assumption
\begin{primenumerate}[resume=propertieslist,start=\getrefnumber{item:dginj_deform_setup_cohomologicalnilpotency}]
\item The dg-ideal $\ker(R_1 \to R_2)$ is strictly nilpotent of order $2$.
\end{primenumerate}
Then, it holds in its original version.
\end{lemma}
\begin{proof}
We fix $\theta \colon R \to S$ as in \S \ref{subsubsection:dginj_deform_setup}, we identify $S=R/I$, and we fix a left derived deformation $\injcat I \to \injcat J$ along $R \to S$, assuming that $\injcat J$ is hlc and $H^0(\injcat J)$ is Karoubian. We define iteratively:
\begin{align*}
    \injcat I_n &= R/I^n \lotimes_R \injcat I, \\
    \injcat I_{n-1} &= R/I^{n-1} \lotimes_{R/I^n} \injcat I_n, \\
    & \vdots \\
    \injcat I_k &= R/I^k \lotimes_{R/I^{k+1}} \injcat I_{k+1}, \\
    & \vdots \\
    \injcat I_1 &= R/I \lotimes_{R/I^2} \injcat I_2.
\end{align*}
$\injcat I_k$ is an $R/I^k$-linear dg-category, and $\injcat I_{k+1}$ is a left derived deformation of $\injcat I_k$ along $\theta_{k+1,k} \colon R/I^{k+1} \to R/I^k$. We know from Lemma \ref{lemma:dgring_morphism_factorization_properties} that $\theta_{k+1,k}$ satisfies the assumptions \eqref{item:dginj_deform_setup_strictsurj}-\eqref{item:dginj_deform_setup_powers_hfp} in \S \ref{subsubsection:dginj_deform_setup}, where \eqref{item:dginj_deform_setup_cohomologicalnilpotency} is replaced by  \textup{\ref{item:dginj_deform_setup_order2strictnilpotency}}. By assumption, we have an $R/I$-linear quasi-equivalence
\[
R/I \lotimes_R \injcat I \cong \injcat J.
\]
Applying the above Lemma \ref{lemma:leftextension_derived_transitive} iteratively, we deduce that
\[
\injcat I_1 = R/I \lotimes_{R/I^2} \injcat I_2 \cong \injcat J
\]
as $R/I$-linear dg-categories.

Now, we may apply the hypothesis recursively and deduce that $\injcat I_n$ is indeed hlc (as an $R/I^n$-linear dg-category) and $H^0(\injcat I_n)$ is Karoubian. Finally, since $\theta_n \colon R \to R/I^n$ is a quasi-isomorphism, we get an $R$-linear quasi-equivalence
\[
\injcat I_n = R/I^n \lotimes_R \injcat I \cong R \lotimes_R \injcat I \cong \injcat I.
\]
Thanks to the induced isomorphism $H^0(R) \xrightarrow{\sim} H^0(R/I^n)$, we get an equivalence $\operatorname{mod}(H^0(R)) \cong \operatorname{mod}(H^0(R/I^n))$. That said, it is easy to conclude that $\injcat I$ is hlc as an $R$-linear dg-category, and $H^0(\injcat I)$ is Karoubian.
\end{proof}
\subsubsection{The proof of Proposition \ref{prop:properties_lifting_leftdeformation}} \label{subsubsection:dginj_deform_liftingproperties_proof}

We finally turn to the actual proof of Proposition \ref{prop:properties_lifting_leftdeformation}. We recall that we are making the following assumptions:
\begin{itemize}
    \item We fix a morphism of commutative dg-rings concentrated in nonpositive degrees $\theta \colon R \to S$, satisfying the assumptions of \S \ref{subsubsection:dginj_deform_setup}. We shall often identify $S=R/I$.
    \item Thanks to the above Lemma \ref{lemma:nilpotency_order2_reduction}, we may (and will) assume that the dg-ideal $I=\ker \theta$ is strictly nilpotent of order $2$: $I^2=0$.
    \item We fix an $S$-linear dg-category $\injcat J$ and an $R$-linear left derived deformation $\injcat I \to \injcat J$. To ease notation, we shall identify $\injcat J = S \lotimes_R \injcat I$. This is allowed, for the results we want to prove depend only on the quasi-equivalence class of $\injcat J$.
\end{itemize}
We first deduce:
\begin{lemma} \label{lemma:kernel_order2nilpotent_Sdgmodule}
The $R$-dg-ideal $I$ has a natural structure of $S$-dg-module, and its restriction along $R \to S$ gives back the original $R$-linear structure.
\end{lemma}
\begin{proof}
We identify $S=R/I$ and define the following $S$-linear structure on $I$:
\[
[r]x = rx,
\]
if $[r] \in R/I$ and $x \in I$. It is easy to check that it is well defined (thanks to $I^2=0$) and it satisfies the desired properties.
\end{proof}
From now on, we shall view $I$ as an $S$-dg-module according to the above lemma.
\begin{lemma} \label{lemma:tensor_lift_Slinear}
 There is a natural $R$-linear quasi-isomorphism of $\injcat I$-$\injcat I$-dg-bimodules:
\[
    I \lotimes_R \injcat I \cong I \lotimes_S (S \lotimes_R \injcat I) = I \lotimes_S \injcat J.
\]
\end{lemma}
\begin{proof}
To the purpose of this proof, we we may take suitable resolutions and assume that $I$ is h-flat as an $S$-dg-module, and that $\injcat I$ is h-flat as an $R$-linear dg-category. Hence, we may take the ordinary tensor products instead of the derived ones. We have $R$-linear morphisms of $\injcat I$-$\injcat I$-dg-bimodules:
\begin{align*}
I \otimes_R \injcat I(-,-) & \to I \otimes_S (S \otimes_R \injcat I(-,-)), \\
x \otimes_R f & \mapsto x \otimes_S (1_S \otimes_R f),
\end{align*}
and
\begin{align*}
I \otimes_S (S \otimes_R \injcat I(-,-)) & \to I \otimes_R \injcat I(-,-), \\
x \otimes_S (s \otimes_R f) & \mapsto xs \otimes_R f.
\end{align*}
It is straightforward to see that they are inverse to each other.
\end{proof}

We start checking that the suitable properties of $\injcat J$ lift to its deformation $\injcat I$.
\begin{lemma} \label{lemma:deformation_lift_negativedegrees}
If $\injcat J$ is cohomologically concentrated in nonpositive degrees, the same is true for $\injcat I$.
\end{lemma}
\begin{proof}
The short exact sequence (of $R$-dg-modules)
\[
0 \to I \to R \xrightarrow{\theta} S \to 0
\]
clearly induces a distinguished triangle
\[
I \to R \to S
\]
in the derived category $\dercomp(R)$. From this, we obtain the following distinguished triangle in the derived category $\dercomp(\injcat I,\injcat I)$ of $\injcat I$-$\injcat I$-dg-bimodules:
\[
I \lotimes_R \injcat I \to \injcat I \to S \lotimes_R \injcat I
\]
Using the induced long exact sequence in cohomology, it is clear that $\injcat I$ is cohomologically concentrated in nonpositive degrees once we know that $I \lotimes_R \injcat I$ and $S \lotimes_R \injcat I$ are cohomologically concentrated in nonpositive degrees. By assumption, $S \lotimes_R \injcat I=\injcat J$ has this property. On the other hand, by the above Lemma \ref{lemma:tensor_lift_Slinear} we have that
\[
I \lotimes_R \injcat I \cong I \lotimes_S \injcat J
\]
is also cohomologically concentrated in negative degrees, being a tensor product of $I$ and $\injcat J$ which are both cohomologically concentrated in nonpositive degrees (cf. Lemma \ref{lemma:tensorproduct_nonpositive}).
\end{proof}
Now, we are already able to show that left derived deformations of dg-categories cohomologically concentrated in nonpositive degrees induce left linear deformations of the corresponding homotopy categories.
\begin{corollary} \label{coroll:induced_deformation_H0}
Assume that $\injcat J$ is cohomologically concentrated in nonpositive degrees. Then, there exists a natural $H^0(S)$-linear equivalence of categories:
\begin{equation*}
    H^0(S) \otimes_{H^0(R)}  H^0(\injcat I) \xrightarrow{\sim} H^0(S \lotimes_R \injcat I) = H^0(\injcat J).
\end{equation*}
In particular, $H^0(\injcat I)$ is a left linear deformation of $H^0(\injcat J)$ along $H^0(\theta) \colon H^0(R) \to H^0(S)$.

Moreover, $\ker (H^0(\theta))$ is an order $2$ nilpotent ideal of $H^0(R)$, hence $H^0(\injcat I)$ is an order $2$ nilpotent left linear deformation of $H^0(\injcat J)$.
\end{corollary}
\begin{proof}
Without loss of generality, assume that $\injcat I$ is an h-flat $R$-linear dg-category, so that $S \lotimes_R \injcat I = S \otimes_R \injcat I$. The natural functor
\begin{equation} \label{eq:induced_deformation_H0} \tag{$\ast$}
\begin{split}
    H^0(S) \otimes_{H^0(R)}  H^0(\injcat I) & \to H^0(S \otimes_R \injcat I), \\
    [s] \otimes [f] & \mapsto [s \otimes f]
\end{split}
\end{equation}
is clearly $H^0(S)$-linear. From the above Lemma \ref{lemma:deformation_lift_negativedegrees} we know that $\injcat I$ is cohomologically concentrated in nonpositive degrees. Hence, we may apply Lemma \ref{lemma:H^0_comparison_complexes} and conclude that \eqref{eq:induced_deformation_H0} is indeed an isomorphism.

For the second part of the claim, consider the exact sequence:
\[
H^0(I) \to H^0(R) \xrightarrow{H^0(\theta)} H^0(S).
\]
Hence, $\ker (H^0(\theta)) = \operatorname{Im}(H^0(I) \to H^0(R))$. To prove that it is nilpotent of order $2$, let $x,y \in Z^0(I)$: we want to prove that $xy$ is a coboundary in $R$. Since $I^2=0$, we even get $xy=0$ and we conclude.
\end{proof}

We can now lift properties that depend only on the homotopy categories.
\begin{lemma} \label{lemma:liftingproperties_KaroubianadditiveH0}
Assume that $\injcat J$ is cohomologically concentrated in nonpositive degrees, and moreover that $H^0(\injcat J)$ is additive, resp. Karoubian. Then, the same is true for $H^0(\injcat I)$.
\end{lemma}
\begin{proof}
From the above Corollary \ref{coroll:induced_deformation_H0} we know that $H^0(\injcat I) \to H^0(\injcat J)$ is an order $2$ nilpotent deformation of $H^0(\injcat J)$ along $H^0(\theta) \colon H^0(R) \to H^0(S)$. By hypothesis, $H^0(R)$ and $H^0(S)$ are coherent commutative rings and $H^0(S)$ is finitely presented over $H^0(R)$; moreover, $H^0(\theta)$ is surjective. Hence, we may directly apply \cite[Proposition A.5 and A.7]{lowen-vandenbergh-deformations-abelian}. {Notice that the flatness hypothesis is not used in the proofs of the cited results}.
\end{proof}

Next, we take care of the (homotopy) coherence properties. To that purpose, we first prove a result which essentially tells that homotopically finitely presented dg-modules (also having an upper bound on cohomology) are stable under derived tensor products.
\begin{lemma} \label{lemma:hfp_preserve_tensorproduct}
Let $J$ be an $S$-dg-module, which is homotopically finitely presented as an $R$-dg-module, such that $H^i(J)=0$ for all $i>0$. Let $\injcat C$ be a (left) homotopically locally coherent $S$-linear dg-category, and let $M \in \dercomp(\opp{\injcat C})^\mathrm{hfp}$ be a homotopically finitely presented dg-module (namely, $H^i(M) \in \operatorname{mod}(\opp{H^0(\injcat C)})$ for all $i \in \mathbb Z$). Moreover, assume that $H^i(M)=0$ for $i \gg 0$. Then, the same properties hold for the the tensor product $J \lotimes_S M$. 
\end{lemma}
\begin{proof}
Since the induced morphism $H^0(R) \to H^0(S)$ is surjective, we deduce by Sublemma \ref{sublemma:finitelypresented_lift_surjective} that, since $J$ is homotopically finitely presented as an $R$-dg-module by assumption, it is also homotopically finitely presented as an $S$-dg-module.

Upon shifting, we may assume that $M$ is cohomologically concentrated in nonpositive degrees. If $f \colon J_1 \to J_2$ is a morphism of $S$-dg-modules cohomologically concentrated in nonpositive degrees and such that $H^0(f)$ is surjective, then the induced morphism
\[
H^0(J_1 \lotimes_S M) \xrightarrow{H^0(f \otimes 1)} H^0(J_2 \lotimes_S M)
\]
is also surjective, thanks to Lemma \ref{lemma:tensorproduct_nonpositive_surjectivemorphism}.

We now refer the reader to \cite[\S 5.2]{genovese-lowen-vdb-dginj}. Since $S$ is a homotopically coherent dg-ring, we may directly adapt \cite[Theorem 5.9 (1)]{genovese-lowen-vdb-dginj} and deduce that $\dercomp(S)^\mathrm{hfp}$ has a natural t-structure induced from the one on $\dercomp(S)$, whose heart is $\operatorname{mod}(H^0(S))$. Then, \cite[Proposition 5.13]{genovese-lowen-vdb-dginj} can be applied with $\mathbf D$ being (a dg-enhancement of) $\dercomp(S)^\mathrm{hfp}$ and $\smallcat q$ being its full dg-subcategory spanned by finite direct sums of copies of $S$ (cf. also \cite[Proposition 5.17]{genovese-lowen-vdb-dginj}). Thanks to the observations in the previous paragraph of this proof, we may now actually apply the argument of \cite[Proposition 5.13]{genovese-lowen-vdb-dginj} to deduce a ``resolution'' of the tensor product $J \lotimes_S M$ from a ``resolution'' of $J$. We start with a distinguished triangle in $\dercomp(S)$:
\[
C \to S^{\oplus m} \to J
\]
for some $m>0$, where $H^0(S)^{\oplus m} \to H^0(J)$ is surjective and $C$ is also a homotopically finitely presented dg-module. By taking the tensor $- \lotimes_S M$, we get a distinguished triangle in $\dercomp(\opp{\injcat C})$:
\[
C \lotimes_S M \to S^{\oplus m} \lotimes_S M \to J \lotimes_S M,
\]
where $H^0(S^{\oplus m} \lotimes_S M) \to H^0(J \lotimes_S M)$ is surjective. We can now iterate the procedure according to the proof of \cite[Proposition 5.13]{genovese-lowen-vdb-dginj} (and applying $- \lotimes_S M$ accordingly). In the end, we find objects of the form $X_k \lotimes_S M$ for $k<0$ such that $X_k$ lies in the pretriangulated hull $\pretr(S)$, together with isomorphisms
\[
H^i(X_k \lotimes_S M) \xrightarrow{\sim} H^i(J \lotimes_S M),
\]
for $i>k$. Moreover, $X_k \lotimes_S M$ is by construction obtained from $M$ by a finite iteration of cone and shifts. Since $\injcat C$ is left hlc, the category $\dercomp(\opp{\injcat C})^\mathrm{hfp}$ is triangulated \cite[Theorem 5.9 (1)]{genovese-adjunctions}; we deduce that $X_k \lotimes_S M$ is homotopically finitely presented. Hence, choosing $k \ll 0$, we find out that $H^i(J \lotimes_S M)$ is a finitely presented left $H^0(\injcat C)$-module for any fixed $i$, which actually means that $J \lotimes_S M$ is homotopically finitely presented, as claimed.
\end{proof}
To finish the proof of Proposition \ref{prop:properties_lifting_leftdeformation}, we only need to deal with the (homotopy) coherence properties. It is worth noticing that we don't require any flatness assumptions in the following proofs, unlike in the abelian case.
\begin{lemma} \label{lemma:hlc_lift_coherentH0}
Assume that $\injcat J$ is (left) homotopically locally coherent. Then, $H^0(\injcat I)$ is (left) coherent.
\end{lemma}
\begin{proof}
Without loss of generality, we assume that $\injcat I$ is an h-flat $R$-linear dg-category. We shall identify $\injcat J = S \lotimes_R \injcat I = S \otimes_R \injcat I$ as usual, and often implicitly apply restriction of dg-modules along the natural dg-functor $\injcat I \cong R \otimes_R \injcat I \to \injcat J$. From Lemma \ref{lemma:deformation_lift_negativedegrees} we know that $\injcat I$ is cohomologically concentrated in nonpositive degrees, and from Lemma \ref{lemma:liftingproperties_KaroubianadditiveH0} we know that $H^0(\injcat I)$ is additive. Hence, left coherence amounts to the existence of weak cokernels (cf. \cite[Proposition A.11]{lowen-vandenbergh-deformations-abelian}). Let $f \colon B \to A$ be a closed degree $0$ morphism in $\injcat I$, and consider the pullback
\[
f^* \colon \injcat I(A,-) \to \injcat I(B,-).
\]
We want to find an object $C \in H^0(\injcat I)$ and a closed degree $0$ morphism $g \colon A \to C$ such that the sequence
\[
H^0(\injcat I(C,-)) \xrightarrow{H^0(g^*)} H^0(\injcat I(A,-)) \xrightarrow{H^0(f^*)} H^0(\injcat I(B,-))
\]
is exact.

Consider the ($R$-linear) commutative diagram in $\dercomp(\opp{\injcat I})$:
\[
\begin{tikzcd}
I \lotimes_S F' \arrow[r] \arrow[d]                                     & F \arrow[r] \arrow[d]                         & F' \arrow[d]                                     \\
{I \lotimes_S \injcat J(A,-)} \arrow[d, "I \lotimes_S f^*"] \arrow[r] & {\injcat I(A,-)} \arrow[d, "f^*"] \arrow[r] & {\injcat J(A,-)} \arrow[d, "S \lotimes_R f^*"] \\
{I \lotimes_S \injcat J(B,-)} \arrow[r]                               & {\injcat I(B,-)} \arrow[r]                  & {\injcat J(B,-),}     
\end{tikzcd}
\]
where $F=\cone(f^*)[-1]$ and $F'=\cone(S \lotimes_R f^*)[-1]$. A few explanations are needed. First, we identify $I \lotimes_R \injcat I = I \lotimes_S \injcat J$ according to Lemma \ref{lemma:tensor_lift_Slinear}. The tensor product
\[
S \lotimes_R F \cong F' \cong \cone(S \lotimes_R f^*)[-1]
\]
is endowed with a natural $S$-linear structure, making it a (right) $S \lotimes_R \opp{\injcat I} = \opp{\injcat J}$-module. Using the same argument as in Lemma \ref{lemma:tensor_lift_Slinear}, we have an ($R$-linear) isomorphism in $\dercomp(\opp{\injcat I})$:
\[
I \lotimes_R F \cong I \lotimes_S F',
\]
which we view as an identification. We also identify
\[
I \lotimes_R \injcat I \cong I \lotimes_S \injcat J,
\]
again according to Lemma \ref{lemma:tensor_lift_Slinear}. The above diagram is then obtained by applying $-\lotimes_R \injcat I$ to the short exact sequence of $R$-modules
\[
0 \to I \to R \xrightarrow{\theta} S \to 0
\]
and by making the identifications as outlined above.  Using functorial choices of cones in $\dercompdg(\opp{\injcat I})$, we also have that both rows and columns of the above diagram are distinguished triangles. 

The dg-category $\injcat J$ is (left) hlc by hypothesis, hence $\dercomp(\opp{\injcat J})^\mathrm{hfp}$ is pretriangulated \cite[Theorem 5.9 (1)]{genovese-lowen-vdb-dginj}. The left $\injcat J$-dg-module $F'$ sits in the vertical distinguished triangle
\[
F' \to \injcat J(A,-) \xrightarrow{S \lotimes_R f^*} \injcat J(B,-)
\]
and by hypothesis both $\injcat J(A,-)$ and $\injcat J(B,-)$ are homotopically finitely presented. We conclude that $F'$ is also homotopically finitely presented as a left $\injcat J$-dg-module; it is also clear that it has cohomology bounded from above. Applying Lemma \ref{lemma:hfp_preserve_tensorproduct}, we deduce that $I \lotimes_S F'$ is a homotopically finitely presented left $\injcat J$-dg-module. By Corollary \ref{coroll:induced_deformation_H0}, we have an $H^0(S)$-linear equivalence $H^0(\injcat J) \cong H^0(S) \otimes_{H^0(R)} H^0(\injcat I)$, and the restriction morphism $\Mod(\opp{H^0(\injcat J)}) \to \Mod(\opp{H^0(\injcat I)})$ preserves finitely presented modules (it can be checked directly, or by invoking \cite[Proposition 4.6]{lowen-vandenbergh-deformations-abelian}). We conclude that both $I \lotimes_S F'$ and $F'$ are homotopically finitely presented as left $\injcat I$-dg-modules. Since
\[
I \lotimes_S F' \to F \to F'
\]
is a distinguished triangle, we conclude that $F$ is also a homotopically finitely presented left $\injcat I$-dg-module.

Finally, we consider the central vertical distinguished triangle
\[
F \to \injcat I(A,-) \xrightarrow{f^*} \injcat I(B,-)
\]
and we take the long exact sequence in cohomology:
\[
H^0(F) \to H^0(\injcat I(A,-)) \xrightarrow{H^0(f^*)} H^0(\injcat I(B,-)).
\]
We know that $H^0(F)$ is finitely presented as a left $H^0(\injcat I)$-module, in particular it is finitely generated. Using that $H^0(\injcat I)$ is additive, we find an object $C \in \injcat I$ and a surjection
\[
H^0(\injcat I(C,-)) \twoheadrightarrow H^0(F).
\]
This gives an exact sequence
\[
H^0(\injcat I(C,-)) \to H^0(\injcat I(A,-)) \xrightarrow{H^0(f^*)} H^0(\injcat I(B,-)).
\]
By the Yoneda Lemma, $H^0(\injcat I(C,-)) \to H^0(\injcat I(A,-))$ is of the form $H^0(g^*)$ for a suitable $g \colon A \to C$, as we wanted.
\end{proof}
\begin{lemma} \label{lemma:hlc_lift_deform}
If $\injcat J$ is (left) hlc, the same is true for $\injcat I$. 
\end{lemma}
\begin{proof}
Thanks to the above results, it now remains to prove only that $H^i(\injcat I(A,-))$ is a finitely presented left $H^0(\injcat I)$-module for all $i \in \mathbb Z$. As in the proof of the above Lemma \ref{lemma:hlc_lift_coherentH0}, we make the following identifications:
\begin{align*}
    \injcat J &= S \lotimes_R \injcat I, \\
    I \lotimes_S \injcat J &= I \lotimes_R \injcat I. \qquad \text{(see Lemma \ref{lemma:tensor_lift_Slinear})}
\end{align*}

Applying $-\lotimes_R \injcat I(A,-)$ to the short exact sequence
\[
0 \to I \to R \to S \to 0
\]
and recalling the above identifications, we get the following distinguished triangle in $\dercomp(\opp{\injcat I})$:
\[
I \lotimes_S \injcat J(A,-) \to \injcat I(A,-) \to \injcat J(A,-).
\]
We know apply Lemma \ref{lemma:hfp_preserve_tensorproduct} and deduce that $I \lotimes_S \injcat J(A,-)$ is a homotopically finitely presented left $\injcat J$-dg-module; the same is true for $\injcat J(A,-)$ by hypothesis. By Corollary \ref{coroll:induced_deformation_H0}, we have an $H^0(S)$-linear equivalence $H^0(\injcat J) \cong H^0(S) \otimes_{H^0(R)} H^0(\injcat I)$, and the restriction morphism $\Mod(\opp{H^0(\injcat J)}) \to \Mod(\opp{H^0(\injcat I)})$ preserves finitely presented modules (it can be checked directly, or by invoking \cite[Proposition 4.6]{lowen-vandenbergh-deformations-abelian}). We conclude that both $I \lotimes_S \injcat J(A,-)$ and $\injcat J(A,-)$ are homotopically finitely presented as left $\injcat I$-dg-modules.

Next, we argue as in the proof of \cite[Lemma 5.10]{genovese-lowen-vdb-dginj}. We take the long exact sequence in cohomology:
\[
H^{i-1}(\injcat J(A,-)) \xrightarrow{\alpha} H^i(I \lotimes_S \injcat J(A,-)) \to H^i(\injcat I(A,-)) \to H^i(\injcat J(A,-)) \xrightarrow{\beta} H^{i+1}(I \lotimes_S \injcat J(A,-)).
\]
From this, we obtain the following short exact sequence in $\Mod(\opp{H^0(\injcat I)})$:
\[
0 \to \coker \alpha \to H^i(\injcat I(A,-)) \to \ker \beta \to 0. 
\]
We know that the domain and codomain of both $\alpha$ and $\beta$ are objects in $\operatorname{mod}(\opp{H^0(\injcat I)})$. By the above Lemma \ref{lemma:hlc_lift_coherentH0}, we know that $H^0(\injcat I)$ is left coherent, which implies that
\[
\coker \alpha,\ \ker \beta \in \operatorname{mod}(\opp{H^0(\injcat I)}).
\]
Since finitely presented modules are closed under extension \cite[0519]{stacks-project}, we conclude that $H^i(\injcat I(A,-))$ is finitely presented, as we wanted.
\end{proof}
Proposition \ref{prop:properties_lifting_leftdeformation} finally follows by combining Lemma \ref{lemma:liftingproperties_KaroubianadditiveH0} and Lemma \ref{lemma:hlc_lift_deform}.

\bibliographystyle{amsplain}

\providecommand{\bysame}{\leavevmode\hbox to3em{\hrulefill}\thinspace}
\providecommand{\MR}{\relax\ifhmode\unskip\space\fi MR }
% \MRhref is called by the amsart/book/proc definition of \MR.
\providecommand{\MRhref}[2]{%
  \href{http://www.ams.org/mathscinet-getitem?mr=#1}{#2}
}
\providecommand{\href}[2]{#2}

\end{document}